\documentclass{amsart}

\usepackage[hmargin=3.5cm,top=3cm,bottom=3cm,includehead]{geometry}

\usepackage{amsmath,mathrsfs}
\usepackage{amssymb, amsmath}
\usepackage{graphicx}
\usepackage{enumerate,color,xcolor}

\usepackage[normalem]{ulem}

\newtheorem{thm}{Theorem}[section]
\newtheorem{lem}{Lemma}[section]
\newtheorem{rmk}{Remark}[section]
\newtheorem{cor}{Corollary}[section]
\newtheorem{prop}{Proposition}[section]

\numberwithin{equation}{section}

\newcommand{\beq}{\begin{equation}}
\newcommand{\eeq}{\end{equation}}
\newcommand{\ben}{\begin{eqnarray}}
\newcommand{\een}{\end{eqnarray}}
\newcommand{\beno}{\begin{eqnarray*}}
\newcommand{\eeno}{\end{eqnarray*}}
\let\f=\frac

\newcommand{\be}{\begin{equation} \label}
	\newcommand{\ee}{\end{equation}}
\newcommand{\bea}{\begin{eqnarray}\label}
	\newcommand{\eea}{\end{eqnarray}}
\newcommand{\bas}{\begin{eqnarray*}}
	\newcommand{\eas}{\end{eqnarray*}}
\newcommand{\bit}{\begin{itemize}}
	\newcommand{\eit}{\end{itemize}}

\newcommand{\N}{{\mathbb N}}

\newcommand{\R}{{\mathbb R}}

\newcommand{\pa}{\partial}

\newcommand{\var}{\varepsilon}

\newcommand{\ba}{\begin{aligned}}
\newcommand{\ea}{\end{aligned}}
 \def\na{\nabla}
 \newcommand{\lr}[1]{\langle #1 \rangle}
\def\eqdefa{\buildrel\hbox{\footnotesize def}\over =}

\begin{document}

\title[A new monotonicity formula for Landau equation]{A new monotonicity formula for the spatially homogeneous Landau equation with Coulomb potential and its applications}

\author{Laurent Desvillettes}
\address[Laurent Desvillettes]{ Universit\'{e} de Paris and Sorbonne Universit\'e, CNRS, Institut de Math\'{e}matiques de Jussieu-Paris Rive Gauche,  F-75013, Paris, France } 
\email{desvillettes@math.univ-paris-diderot.fr}

\author{Ling-Bing He}
\address[Ling-Bing He]{Department of Mathematical Sciences, Tsinghua University, Beijng, 100084, P.R. China.}
\email{hlb@tsinghua.edu.cn}

\author{Jin-Cheng Jiang}
\address[Jin-Cheng Jiang]{Department of Mathematics, National Tsing Hua University, Hsinchu, Taiwan 30013, R.O.C}
\email{jcjiang@math.nthu.edu.tw}

\begin{abstract}  
We describe a time-dependent  functional  involving the relative entropy 
and the $\dot{H}^1$ seminorm, which decreases along solutions to the spatially homogeneous Landau equation with Coulomb potential.  
The study of this monotone functionial sheds light on the competition between the dissipation and the nonlinearity for this equation. 
It enables to obtain new results concerning  regularity/blowup issues for the Landau equation 
with Coulomb potential. 
%This monotone function has strong connections to the issues of global regularity, no blowup after a finite time and the description of the potential %blowup of solutions to this equation. In particular, we propose quantitative estimates which shed light on the competition between the dissipation %and the nonlinearity.   
\end{abstract}
	


\maketitle
%\setcounter{tocdepth}{1}
%\tableofcontents

\section{Introduction}

We  consider   the spatially homogeneous Landau equation with Coulomb potential
\begin{equation}\label{landau}
\partial_t f= Q(f,f)(v),
\end{equation}
complemented with initial data $f_0=f_0(v) \ge 0$. Here $f := f(t,v) \ge 0$ stands for the distribution of particles that at time $t \in \R_+$ possess the velocity $v \in \R^3$. The Landau operator (with Coulomb potential) $Q$ is a bilinear operator acting only on the velocity variable $v$. It writes
\begin{equation} \label{12d}
Q(g,h) =\nabla \cdot {\Big (}[ a*g ]\;\nabla h- [ a*\nabla g ] \; h{\Big )},
\end{equation}
with 
\begin{equation} \label{13d}
a(z)=|z|^{-1} \, \left(Id -\frac{z\otimes z}{|z|^2} \right).
\end{equation}
% Here $a$ is a matrix-valued function that is symmetric, (semi-definite) positive.
This equation,  first obtained by Landau in 1936,
% from the Boltzmann equation with Rutherford scattering cross section,
 is  used to describe the evolution in time of a (spatially homogeneous) plasma due to collisions between charged particles under the Coulomb potential. 

Introducing the quantity
\ben\label{Defbc}
b_i(z): = \sum_{j=1}^3 \partial_j a_{ij}(z) = -2 \,z_i\, |z|^{-3}, \quad 
%c(z): = \sum_{i=1}^3\sum_{j=1}^3\partial_{ij} a_{ij}(z) = -8 \pi\, \delta_0(z),
\een
 the Landau operator with Coulomb potential can also be written as
\ben\label{Qbis}
\ba
Q(f,f) 
&= \sum_{i=1}^3 \partial_i \bigg( \sum_{j=1}^3(a_{ij}*f)\, \partial_j f - (b_i *f)\, f  \bigg)   \\
&= \sum_{i=1}^3\sum_{j=1}^3(a_{ij}*f) \,\partial_{ij}f + 8\pi \,f^2 ,
\ea
\een
where we used the identity $\sum_{i=1}^3\pa_i b_i(z)= - 8\pi\delta_0(z)$.
\subsection{Basic properties of the equation and notations} 

The weak 
 formulation of the Landau operator $Q$, for a suitable test function $\varphi$, 
is written in the following way:
\ben\label{Qweak1}
&& \int_{\R^3} Q(f,f) (v)\, \varphi(v) \, dv= - \frac12 \, \sum_{i=1}^3\sum_{j=1}^3 \iint_{\R^3 \times \R^3} a_{ij}(v-v_{*}) \\
&&\quad\times
 \left\{ \frac{\partial_j f}{f}(v) - \frac{\partial_j f}{f}(v_{*})  \right\} \left\{ \partial_i \varphi(v) - \partial_i \varphi(v_{*})  \right\}   
 f(v) f(v_{*}) \, dv \, dv_{*}. \nonumber
\een

From formula~\eqref{Qweak1}, we can obtain the fundamental properties of the Landau operator $Q$. 
The operator indeed conserves (at the formal level) mass, momentum and energy, more precisely
\ben\label{cons}
\int_{\R^3} Q(f,f)(v) \, \varphi(v) \, dv = 0 \quad\text{for}\quad \varphi(v) = 1, v_i, \frac{|v|^2}2,\;i=1,2,3.
\een

We also deduce from formula (\ref{Qweak1}) the entropy structure of the operator (still at the formal level) by taking the test function $\varphi(v) = \log f(v)$, 
%\be\label{entr}
%\frac{d}{dt}H(f) = -D(f) \le 0,
%\ee
%where $H(f) := \int f(v) \log f(v) \, dv$ is the entropy, and 
that is 
%$D(f) := - \int Q(f,f)(v) \log f(v) \, dv$ is the entropy dissipation functional given by
\ben\label{D(f)}
D(f)  := - \int_{\R^3} Q(f,f)(v)\, \log f(v) \, dv
\een
$$ =  \frac12 \sum_{i=1}^3\sum_{j=1}^3\iint_{\R^3 \times \R^3} a_{ij}(v-v_{*}) \left\{ \frac{\partial_i f}{f}(v) - \frac{\partial_i f}{f}(v_{*})  \right\} \left\{ \frac{\partial_j f}{f}(v) - \frac{\partial_j f}{f}(v_{*})  \right\}   \, f(v)\, f(v_{*}) \, dv \, dv_{*}.  $$

Note that $D(f) \ge 0$ since the matrix $a$ is (semi-definite) positive.  Note also that for any $f$ such that $D(f)=0$,  it 
 can be shown (cf. \cite{D} and \cite{DX} for a rigorous statement and proof) that $f$ is a Maxwellian distribution, that is $f = \mu_{\rho,u,T}$, with
\ben\label{MaxGen}
\mu_{\rho,u,T}(v) = \frac{\rho}{(2\pi T)^{3/2}} \, e^{-\frac{|v-u|^2}{2T}},
\een
where $\rho\ge 0$ is the density, $u \in \R^3$ is the mean velocity and $T>0$ is the temperature of the plasma. 
They are defined by
\begin{equation} \label{cco}
\rho = \int_{\R^3} f(v) \, dv, \quad 
u = \frac{1}{\rho} \int_{\R^3} v \,f(v) \, dv, \quad
T = \frac{1}{3 \rho} \int_{\R^3} |v-u|^2\, f(v) \, dv.
\end{equation}

Thanks to the conservation of mass, momentum and energy, we have  (when $f: = f(t,v)$ is a solution of eq. (\ref{landau}) -- (\ref{13d}) and $\rho,u,T$ are defined above, at the formal level),
\ben\label{conssh}
\forall t \ge 0, \qquad \rho(t)  = \rho (0), \quad u(t)=u(0),
\quad T(t)=T(0),
\een
which implies that the parameters $\rho,u, T$ are constant (along solutions of eq. (\ref{landau}) -- (\ref{13d})).
%  in formula  \eqref{MaxGen} are constants. 
\medskip

Denoting  (when $f: = f(t,v)$ is a solution of eq. (\ref{landau}) -- (\ref{13d}))  by
 \ben\label{DefHt}
H(t):=H(f|\mu_{\rho,u,T})(t) := \int_{\R^3} \bigg(f(t,v) \log \bigg( \f{f(t,v)}{\mu_{\rho,u,T}} \bigg) -f(t,v) +\mu_{\rho,u,T}\bigg) \, dv,\een 
the relative entropy with respect to $\mu_{\rho,u,T}$ (defined by \eqref{MaxGen}, \eqref{cco}), we see that 
% The entropy enjoys the property
 (still at the formal level),
\ben\label{entropyeq}
\frac{d}{dt}H(t) = -D(f(t,\cdot)) \le 0.
\een
%where \ben\label{DefHt}
%H(t):=H(f|\mu_{\rho,u,T})(t) := \int_{\R^3} \bigg(f(t,v) \log \f{f(t,v)}{\mu_{\rho,u,T}}-f(t,v) +\mu_{\rho,u,T}\bigg) \, dv,\een is the relative 
%entrop,y and $D(f)$ defined by (\ref{D(f)}) is the entropy dissipation. 
Note that in the above definition, $H(t)$ differs from the usual (non relative) entropy 
$\int f(t,v) \log f(t,v)dv$ only by a constant, thanks to identities (\ref{conssh}).
\medskip

Throughout this paper, we shall assume that $f_0 \ge 0$ and $f_0 \in L^1_2 \cap L \log L (\R^3)$. Furthermore,  without loss of generality, we shall also assume that
 $f_0$ satisfies the normalization identities 
\ben\label{f0}
\int_{\R^3} f_0(v)\, dv = 1, \quad \int_{\R^3} f_0(v)\, v\, dv =0, 
\quad \int_{\R^3} f_0(v)\, |v|^2 \, dv =3,
\een
which can be rewritten $\rho(0)=1$, $u(0)=0$, $T(0)=1$. Finally, we denote by 
\ben\label{Defmu}\mu(v):= (2\pi)^{-3/2} e^{-|v|^2/2}, \een the Maxwellian distribution (centred reduced Gaussian) with same mass, momentum and energy as $f_0$ satisfying (\ref{f0}).
 \medskip

Next we introduce some function spaces which will be used throughout the paper:

  \noindent $\bullet$ Let $\lr{v}:=(1+|v|^2)^{1/2}$ denote the Japanese bracket. For any $p \in [1, +\infty[$, $l \in \R$, the $L^p_{l}$ norm is defined by
$$ \| f \|_{L^p_{l}}^p := \int_{\R^3} |f(v)|^p \, \lr{v}^{pl}\, dv.$$

   \noindent $\bullet$ The following quantity, for functions of $L\log L$,  is written as if it were a norm, and defined by   
   $$\|f\|_{L\log L}: =\int_{\R^3}   |f| \log (1+ |f|) \, dv.$$ 

 \noindent  $\bullet$ For any $p\in(1,\infty), q \in [1, +\infty]$, the standard Lorentz space $L^{p,q}$ is defined by the norm
 \begin{equation}\label{Lorentz}
 \|f\|_{L^{p,q}}:= 
 \left\{
 \begin{array}{l}
 \bigg(\int_0^{\infty} \big(t^{1/p}f^{**}(t) \big)^{q}\frac{dt}{t}\bigg)^{1/q},\;1\leq q<\infty\\
 \sup\limits_{t>0}\; t^{1/p}f^{**}(t), \;q=\infty,
 \end{array}
 \right.
 \end{equation}
where $f^{**}(t):=t^{-1}\int_0^t f^*(s)ds$, and  $f^*$ is the decreasing rearrangement  of $f$. We also denote when
$l \in \R$
 the weighted Lorentz norm by 
\beno   \|f\|_{L^{p,q}_{l}}:=\|f(\cdot)\lr{\cdot}^{l}\|_{L^{p,q}}.\eeno
More details on Lorentz spaces including the case when $p=1,\;p=\infty$ can be found in the Appendix.

 \noindent $\bullet$ The homogeneous Sobolev norm $\dot{H}^m$ with $m\in\R$ is defined by
 $$\|f\|_{\dot{H}^m}^2:=\int_{ \R^3} |\xi|^{2m}|\hat{f}(\xi)|^2d\xi , $$
while the weighted inhomogeneous Sobolev norm $H^m_l$ with $m\in\N,l\in\R$ is defined by
\beno \|f\|_{H^m_l}^2:=\sum_{|\alpha|\le m}\int_{ \R^3}   |\pa^\alpha (f\lr{v}^l)|^2dv.\eeno

\subsection{Short review on the Landau equation with Coulomb potential}

 We briefly  review the works on the Landau equation with Coulomb potential \eqref{landau}  -- \eqref{13d}.

\smallskip

\noindent$\bullet$ {\bf Existence and uniqueness of solutions:} In \cite{Vi}, Villani~ proved the global existence of the so-called $H$-solutions 
for equation~\eqref{landau} -- (\ref{13d}) when the
 initial data have finite mass, energy and entropy. The key part of the proof lies in the use of 
the entropy dissipation $D(f)$, rewritten as
\ben D(f(t))=2\int\int \frac{1}{|{ v-v}_*|}\big|\Pi({ v-v}_*)\nabla_{{ v-v}_*}
\sqrt{f(t,{ v}) f(t,{ v}_*)}
\big|^2
{ d}{ v}{d}{ v}_*. \een
Here $\Pi(z)\nabla:=(Id -\frac{z}{|z|}\otimes\frac{z}{|z|})\nabla$, is called the {\it weak projection gradient} (see \cite{HLP} and \cite{Vi}). 
In all generality (when an estimate for $\na_v f$ is not available),  $\Pi({ v-v}_*)\nabla_{{ v-v}_*}$ is not equal to $\Pi({ v-v}_*)\nabla_v-\Pi({ v-v}_*)\nabla_{v_*}$.
  This means that the construction of the approximated solutions to an $H$-solution plays a significant role.
 When the solutions are well-constructed (that is, using a suitable approximation process), we have
 \ben\label{nablav-v*}\Pi({ v-v}_*)\nabla_{{ v-v}_*}=\Pi({ v-v}_*)\nabla_v-\Pi({ v-v}_*)\nabla_{v_*}.\een
% However for those approximated solutions come from the Boltzmann equation,  we have no idea when \eqref{nablav-v*} holds in general. 
We refer readers to \cite{HLP} for more details. When \eqref{nablav-v*} holds, we can use the estimate for the entropy dissipation $D(f)$ in \cite{D} to show that an $H$-solution is  a weak solution of the equation. More precisely, there is an explicitly computable constant $C_0 = C_0(\bar{H}) >0$ such that, for all (normalized) $f\ge 0$ satisfying $H(f) \le \bar{H}$, the following inequality holds:
\ben\label{L3<D}
\| f \|_{L^3_{-3}} \le C_0 \,( 1 + D(f) ).
\een
Therefore,  we know that such an $H$-solution of equation \eqref{landau} -- (\ref{13d}) lies in $ L^1_{loc}([0,\infty); L^3_{-3}(\R^3))$, and this estimate is sufficient to show that it is indeed a weak solution in the usual sense.
\medskip
 
  Fournier \cite{Fournier} showed that uniqueness holds for the solutions of~\eqref{landau} -- (\ref{13d}) 
lying in the class $L^\infty_{loc}([0,\infty); L^1_2(\R^3)) \cap L^1_{loc}([0,\infty);L^\infty(\R^3))$, and this result implies a local well-posedness result assuming further that the initial data lie in $L^\infty (\R^3)$, thanks to the local existence result of Arsenev-Peskov~\cite{Arsenev} for such initial data. We also refer to~\cite{CG20} for uniqueness of higher integrable solutions, and to \cite{KrSt} for the study of an equation sharing significant features 
with eq.~\eqref{landau} -- (\ref{13d}). 

\medskip
  In the spatially inhomogeneous context, we quote \cite{Vi2} for the existence of
renormalized solutions and~\cite{Guo}  and  \cite{heyang}  for the global well-posedness near Maxwellian and the local well-posedness 
%for the inhomogeneous equation
 in weighted Sobolev spaces. We finally refer to \cite{CM} for a general perturbation result, and to \cite{HeSn1}, \cite{HeSn2} for conditional
regularity results.
%\textcolor{red}{I propose to indicate here the latest works on Landau-Coulomb equation. Laurent}
%\textcolor{blue}{Not sure which reference is missing}
%\textcolor{blue}{Please add them to the content as you wish.}
\medskip

\noindent$\bullet$ {\bf Long time behavior:} In a perturbative and spatially inhomogeneous framework, Guo and Strain~\cite{GS} (see also \cite{CM}) proved for solutions of (\ref{landau}) -- (\ref{13d}) the stretched exponential decay to equilibrium in a high-order Sobolev space with fast decay in the velocity variable. 
For (uniformly w.r.t time) {\it{a priori}} smooth solutions with large initial data,  LD and Villani~\cite{DV-boltzmann} proved 
the algebraic convergence to equilibrium. 
\par 
In the homogeneous setting, Carrapatoso, LD and LH  proved the following result which plays 
an essential role in the present paper:

\begin{thm}\label{thm:decay} {\rm(}{\rm Cf.} Theorem 2 and Lemma 8 of~\cite{CDH}{\rm)}
Let $f_0 \in L \log L (\R^3)\cap  L^1_\ell(\R^3)$ with  $ \ell > \frac{19}{2}$ satisfy the normalization \eqref{f0}, 
and consider a (well-constructed) weak (or H-) solution $f$ to eq. ~\eqref{landau} -- \eqref{13d}  with initial datum $f_0$.   
Then for any strictly positive $ \beta < \frac{2\ell^2-25\ell+57}{9(\ell - 2)} $, there exists some computable constant $C_{\beta}>0$ (depending only on $\beta$,  and $K>0$ such that
$  \|f_0\|_{L^1_\ell}  + \|f_0 \|_{L\,\log L} \le K$), such that   the relative entropy satisfies
\ben\label{decayH}
 \forall\, t \ge 0, \qquad H(t) \le C_\beta \,(1+t)^{ - \beta}.
\een
Moreover, for all $\ell >2$, there exists $C_{\ell} >0$ (which only
 depends on $\ell$ and $K$ such that $||f_0||_{L^1_{\ell}(\R^3)} + \|f_0 \|_{L\, \log L} \le K$ ),  such that 
% \textcolor{red}{to be checked, it is not explicitly written in the lemma}
\ben\label{momentpropa} \forall t >0, \qquad ||f( t, \cdot)||_{L^1_{\ell}(\R^3)}  \le C_{\ell} \, (1+t). \een
\end{thm}

\noindent$\bullet$ {\bf Functional estimates:}
 In \cite{D}, it is shown that (for normalized $f \ge 0$) the following estimate holds,
\ben\label{Dfdissipation11} D(f)+ 1\ge C_{D,1} \, \|\sqrt{f}\|_{H^1_{-\f32}}^2, \een
where $C_{D,1}>0$ depends only on an upper bound of $H(f|\mu)$.
\medskip

Using  the precised Sobolev embedding inequality $\|f\|_{L^{6,2}}\le C\,\|\na f\|_{L^2}$ (see~\cite{AdFu}) and the O'Neil inequality in Lorentz spaces (see Proposition \ref{oneil} in the Appendix), we end up with the following inequality (holding for normalized $f \ge 0$):
\ben\label{Dfdissipation} D(f)+ 1\ge C_{D,2}\, \|f\|_{L^{3,1}_{-3}}, \een 
where $C_{D,2}>0$ depends only on an upper bound of $H(f|\mu)$.
\par 
We refer to \cite{D}, \cite{DX} and \cite{CDH} for variants of inequality \eqref{Dfdissipation}.

% \begin{rmk} To derive  \eqref{Dfdissipation}, we use the Sobolev embedding inequality $\|f\|_{L^{6,2}}\le C\|\na f\|_{L^2}$ and the O'Neil inequality in Lorentz space(see it in the Appendix). 
%\end{rmk}

%\begin{rmk} Unfortunately we cannot show the precise dependence of $C_\beta$ on $H(t_0)$ and $\|f_0\|_{L^1_\ell}$ since it is too tedious to keep track of these in the proof.
%\end{rmk}
\medskip

\noindent$\bullet$ {\bf Partial regularity issue:} Very recently Golse, Gualdani, Imbert  and Vasseur~\cite{GGIV}  proved that 
the set of singular times for (suitable) weak solutions of the
spatially homogeneous Landau equation with Coulomb potential  has Hausdorff dimension at most $1/2$ if the initial data possesses all polynomial moments. 
The key ingredient of the proof lies in the application of De Giorgi's method to a scaled suitable solution.
% which is construct by the entropy and the entropy dissipation equality.
 They also observed that the solution to Landau equation with Coulomb interaction enjoys a scaling property which is similar to that of the 3D incompressible Navier-Stokes equation.  This explains the link between the bound on the Hausdorff dimension of the set of singular times in both equations. We also cite the papers \cite{GG1} and \cite{GG2} where Gualdani and Guillen provide estimates which are useful to understand 
the issues of regularity/appearance of blowup and the role played by the various terms in the  Landau equation with Coulomb potential.
%\textcolor{red}{I propose to cite \cite{GG1, GG2},  Laurent}
%\bigskip

\subsection{Main result}
 A very challenging problem for the (spatially homogeneous) 
Landau equation with Coulomb potential (\ref{landau}) -- (\ref{13d}) is to answer whether  
 the smoothness is propagated for all positive times, or if some blowup may occur after a finite time. If such a blowup appears, a further challenging issue is to understand what really happens at the blowup time
(Cf. \$1.3 (2) in Chapter 5 of Villani's monograph~\cite{Villani-BoltzmannBook}). 
The main result of this paper provides new partial answers to the first question, while another result of this paper deals with the second question. 
In particular, our results shed some new light on the competition between the dissipation and the nonlinearity (see more details at the end of this section) for Landau equation with Coulomb potential.
%\smallskip

%Before stating our main results, we list two moment assumptions on the initial data $f_0\in L^1_\ell$. Suppose $\lambda_\ell(\theta):=\f{(\ell-\theta)(\ell-55/2)}{9(\ell-2)}-\f{\theta}{\ell}$.

%\noindent{\bf Moment Assumption (I).}  $\ell$ satisfies that  $\ell>31$ and  $\lambda_\ell(25)>\f72$. Moreover there exists a constant $\tau$ fulfilling that  $\tau \in [31,\ell)$  and $ \lambda_\ell(\tau)\ge0$. 
%\smallskip

%\noindent{\bf Moment Assumption (II).}    $\ell$ satisfies that  $\ell>31$ and $\lambda_\ell\big(15/4+7\big)>\f12$.
%\smallskip

%We now are in a position to state our results. The
Our main result is concerned with the new monotonicity formula for equation  (\ref{landau}) -- (\ref{13d}) announced in the title, and its byproducts:

\begin{thm}\label{maintheorem1} Let $f_0 \in  L \log L (\R^3)\cap  L^1_{55}(\R^3) \cap \dot{H}^1(\R^3)$ be a nonnegative initial datum  satisfying the normalization \eqref{f0}.\par 
 Then  there exist (explicitly computable) constants $B^*, C_6>0$, $k_2> 7/2$, $k>0$ (depending only on $K$ satisfying
 $ \|f_0\|_{L^1_{55}(\R^3)} +  \|f_0\|_{L\,\log L(\R^3)}  \le K$) such that the three following statements hold:
\medskip

\noindent {\rm (i)(Monotonicity of a functional).}   We consider $T>0$ and denote by $f:= f(t,v)$ a smooth and quickly decaying
 ($C^2_t(\mathcal{S})$)
%well constructed (cf. \cite{HLP}) weak (or $H-$)
 nonnegative solution on the interval $[0,T]$ to eq. (\ref{landau}) -- (\ref{13d}) with initial 
datum $f_0$.  We define 
$h:=f-\mu$, where $\mu$ is given by eq. (\ref{Defmu}) (recall also  that $H(t)$ is the relative entropy given by~\eqref{DefHt}).
\par 
 Then
% there exist (explicitly computable) constants $B^*, C_6>0$, $k_2> 7/2$, $k>0$ (depending only on $K$ satisfying$ ||f_0||_{L^1_{55}(\R^3)} +  ||f_0||_{L\,\log L(\R^3)}  \le K$)
%such that
 the following {\it a priori} estimate (that we call monotonicity property) holds for $t \in [0, T]$:
%  a time-dependent function, \ben\label{MoFuc}\mathcal{M}(t):=H(t)-\f52\big(\|h(t)\|_{\dot{H}^1}^2+C_\ell(1+t)^{-\lambda_\ell(25)+1}\big)^{-\f25},\een where   $C_\ell=C(\ell, H(f_0))$ is a computable constant, and some computable constants $C_2=C(\ell, H(f_0))$ and $k\le \min\{k_1,\f25k_2-\f{7}5\}$ with  $k_1=\f45\lambda_\ell(14/5), 
%k_2=\lambda_\ell(25)$ such that
\ben\label{MonoFom}  \f{d}{dt} \bigg[H(t)-\f52\bigg(\|h(t)\|_{\dot{H}^1}^2+ B^*\, (1+t)^{- k_2 +1}\bigg)^{-\f25}   \bigg] + \,C_6\, (1+t)^{k}\le 0.\een  
%In particular, one of the following scenarios holds.
\medskip

\noindent {\rm (ii)(Global regularity for initial data below threshold).} If moreover $H(0) \, \big(\|h(0)\|_{\dot{H}^1}^2+ B^*\, \big)^{\f25}\le \f52$, then eq. (\ref{landau}) -- (\ref{13d}) 
admits a (unique)  global and strong (that is, lying in $L^{\infty}(\R_+; H^1(\R^3))$) nonnegative solution satisfying that
\ben\label{smoesti}
\forall t> 0, \qquad  \|h(t)\|_{\dot{H}^1} \, \bigg(H(t)+\f{C_6}{k+1}\, \bigg[(1+t)^{1+k}-1 \bigg]\bigg)^{\f54} \le (\f{2}{5})^{-\f54},
% \|h(t)\|_{\dot{H}^1}\le \bigg(\f{2C_2}{5(1+k)}\bigg)^{-\f54}(1+t)^{-\f54(1+k)}. 
\een 
where we used the same notations for $h, \mu$ and $H$ as in statement (i).
\medskip

\noindent {\rm (iii)(No blowup after a finite time).} If finally  $H(0) \, \big(\|h(0)\|_{\dot{H}^1}^2+ B^*\, \big)^{\f25} > \f52$, we denote
% $$T= \bigg(\f{1+k}{C_2}\mathcal{M}(0)+1\bigg)^{\f1{k+1}}-1 \quad\mbox{and}\quad  \mathcal{M}(T)\le 0.$$
\ben\label{DEFT} T^* := \bigg(\f{1+k}{C_6} \bigg[  H(0)-\f52\, [\|h(0)\|_{\dot{H}^1}^2 + B^*]^{-2/5} \bigg] +1\bigg)^{\f1{k+1}}-1 .\een
Then
%For any initial data satisfying the assumptions of the Theorem, 
one can construct a global  weak (or $H$-) nonnegative solution of eq. (\ref{landau}) -- (\ref{13d}), such that for $t>T^*$, it becomes  global and strong (that is, it lies in $L^{\infty}((T^*,\infty); H^1(\R^3))$), and satisfies the estimates
%If $\mathcal{M}(0)>0$, then there exists a time $T$ such that
\ben 
&& \label{nnn}
 H(t) \, \bigg[\|h(t)\|_{\dot{H}^1}^2 + B^*\, (1+t)^{1-k_2} \bigg]^{-2/5}  \le \f52;\\
&&\label{noblowupht}\|h(t)\|_{\dot{H}^1} \, \bigg( \f{C_6}{k+1}\, \bigg[(1+t)^{1+k}-(1+T^*)^{1+k} \bigg]\bigg)^{\f54} \le  (\f{2}{5})^{-\f54},
%{\rm where}
 \een 
where we used the same notations for $h, \mu$ and $H$ as in statement (i).
% $$T= \bigg(\f{1+k}{C_2}\mathcal{M}(0)+1\bigg)^{\f1{k+1}}-1 \quad\mbox{and}\quad  \mathcal{M}(T)\le 0.$$
%\ben\label{DEFT} T := \bigg(\f{1+k}{C_6} \bigg[  H(0)-\f52\, [\|h(0)\|_{\dot{H}^1}^2 + B^*]^{-2/5} \bigg] +1\bigg)^{\f1{k+1}}-1 .\een
%\textcolor{red}{consequence for smoothness of the solution when the time is large to be determined}
\end{thm}

Using variants of the estimates above, it is possible to get more standard results of local (in time) well-posedness for large initial data  (in  $\dot{H}^1$ norm),
 and global (in time) well-posedness for small initial data (in  $\dot{H}^1$ norm). It is also possible to give estimates concerning a possible blowup (of the $\dot{H}^1$ norm). These results are stated in the three following propositions, where
we recall that  $\mu$ is the Maxwellian given by eq. (\ref{Defmu}), and we denote $h:= f - \mu$ and $h_0:= f_0 - \mu$.
\medskip 

We begin with the local well-posedness of the equation:

\begin{prop}
%[Local well-posedness of Landau equation]
\label{localwell} 
Let $f_0 \in  L \log L (\R^3)\cap  L^1_{55}(\R^3) \cap 
\dot{H}^1(\R^3)$ be a nonnegative initial datum  satisfying the normalization \eqref{f0}. 
Then there exists    a time $\mathcal{T}:=\f54(\| h_0 \|_{\dot{H}^1}^2+C_7^{-1})^{-\f45}$ (where $C_7>0$ only depends on 
$K$ such that $\|f_0\|_{L^1_{55}} + \|f_0\|_{L\log L} \le K$), such that  the  Landau equation \eqref{landau} -- \eqref{13d} admits a unique strong
solution on the interval $[0, \mathcal{T}]$. By strong solution, we mean here
 that $f\in C([0,\mathcal{T}]; \dot{H}^1) \cap L^2([0,\mathcal{T}]; H^{2}_{-3/2})$.
% and $\mu$ is the Maxwellian given by eq. (\ref{Defmu}), such that   the  equation \eqref{landau} admits a unique solution $f\in C([0,\mathcal{T}); \dot{H}^1)$. Moreover,
% for any $0<t_0,\eta\ll1$, 
%\ben\label{smoothingH2}
 %\|f(t)\|_{L^\infty([t_0,\mathcal{T}-\eta];H^1_{19/2}\cap H^2)}^2+\|f\|_{L^2([t_0,\mathcal{T}-\eta];H^3_{-\f32}\cap H^2_8)}^2\le C_8,
%(t_0,\|f_0\|_{L\log L\cap L^1_{55}\cap H^1}).\een  
%where  $C_8>0$ only depends on $\|f_0\|_{L^1_{55}}$, $\|f_0\|_{L\log L}$ and $\|f_0\|_{H^1}$.
\end{prop}

We turn then to the  global well-posedness for small initial data:

\begin{prop}\label{mainresult2}   Let $f_0 \in  L \log L (\R^3)\cap  L^1_{55}(\R^3)\cap H^{1}_{2}(\R^{3})$ be a nonnegative initial datum satisfying the normalization \eqref{f0}, and 
%We denote 
%$h:=f-\mu$ (and 
$h_0:= f_0 - \mu$.
% and $H(0)$ is fixed (which is allowed to be  large), 
\par 
Then there exists a (small) constant $\epsilon_0>0$ (depending only on $K>0$ such that $\|f_0\|_{L^1_{55}} + \|f_0\|_{L \log L}  \le K$), 
%more precisely, this dependence can be estimated in the following way: $\epsilon_0\ll K^{-2/3}$) 
such that if $\|h_0\lr{\cdot}^{2}\|_{\dot{H^1}}\le \epsilon_0$,
  the Landau equation with Coulomb potential (\ref{landau}) -- (\ref{13d}) admits a (unique) global smooth (that is, lying in $L^{\infty}([0, +\infty); H^1_2(\R^3))$) and nonnegative solution, denoted by $f:=f(t,v)$. Moreover, (under the same assumption on the initial datum) there exists a constant $C>0$ only depending on $K$ 
%and $\epsilon_0$
 such that
(with the notation $h:=f-\mu$)
\beno \|h(t, \cdot )\|_{H^1_2}\le C\, (1+t)^{- \f{15}4}. \eeno
%\textcolor{red}{I suggest to Many details in this statement are to be checked. Laurent} 
\end{prop}

 Finally, we give some clues about the behavior of solutions close to a potential blowup:

\begin{prop}
\label{des}
%\noindent{\rm(iii)(Description of the potential blow-up).} 
Let  $f:= f(t,v)$ be a nonnegative solution of the Landau equation with Coulomb potential (\ref{landau}) -- (\ref{13d}), corresponding to initial data satisfying the assumptions of Theorem \ref{maintheorem1}.  We suppose that  $f \in L^{\infty}([0, t]; {H}^1(\R^3))$ for all $t \in [0, \bar{T}[$, and that $\|\nabla f(t)\|_{L^2(\R^3)}$  blows up at  time $\bar{T}$. Then  
%$T^*<T$ of (2). And
 for $\bar{T}-t\ll 1$ and some explicitly computable constants $c,C>0$, $C_1, C_2>0$ (depending only on $K$ satisfying
 $ \|f_0\|_{L\, \log L(\R^3)} +  \|f_0\|_{L^1_{55}(\R^3)} \le K $),
 %\textcolor{red}{precise the dependence of the constants}
\beno  
\|h(t)\|_{\dot{H}^1}&\ge& 
 C(H(t)- \bar{H})^{-\f54} \quad\mbox{with}\quad H(t)- \bar{H} \ge C\,(\bar{T}-t)(1+\bar{T})^{k+1};\\
  \inf_{s\in[t,\bar{T}]} \|h(s)\|_{\dot{H}^1}& \le& \bigg(\mathcal{B}(c\,(\bar{T}-t)) \f{2\,(\bar{T} - t)}{C_1}(1+ \bar{T})^{-(k_1 +k_2)}\bigg)^{\f5{14}},  \eeno 
	where 
 $\bar{H}:=\lim_{t\rightarrow \bar{T}-}H(t)$ and
 $\mathcal{B}(x) :=C_2\, x^{-13}\exp\{7x^{-\f{450}{14}}\}$.
\end{prop}

\subsection{Comments}  

\subsubsection{\bf Comment on the monotonicity formula \eqref{MonoFom}.} 
To the best of our knowledge, inequality~\eqref{MonoFom} of statement (i) of Theorem \ref{maintheorem1} is a new monotonicity formula for  the (smooth solutions of the) Landau equation with Coulomb potential. 
The  explicit increasing rate $C_6\,(1+t)^k$  comes from the   dissipation effect of the equation.  
We denote the  monotone functional  by 
\ben\label{monoquantity}
\mathcal{M}(t):=H(t)-\f52\bigg(\|h(t)\|_{\dot{H}^1}^2+ B^*\, (1+t)^{- k_2 +1}\bigg)^{-\f25},
\een   
and notice that the differential inequality~\eqref{MonoFom} formally allows  $\|h(t)\|_{\dot{H}^1}$ to blow up. The global dynamics of 
$\mathcal{M}(t)$ described by inequality~\eqref{MonoFom} gives clues about the global dynamics for the original solution:
\begin{itemize}
\item
When $\mathcal{M}(t_0)$ is below its critical value, (that is, $\mathcal{M}(t_0) \le 0$, cf. comment below), the solution to  equation (\ref{landau}) -- (\ref{13d}) after time $t_0$ will remain bounded in $\dot{H}^{1}$ and 
converge to the equilibrium; this is indicated in Statement $(ii)$ of Theorem \ref{maintheorem1}. We call this situation the  {``stable regime''}. 
\item 
When $\mathcal{M}(t)$ is above its 
critical value (that is, $\mathcal{M}(t_0) \ge 0$), some blowup may occur, but there exists a computable time $T^*$ (strictly bigger than the blowup time if it occurs)
 such that $\mathcal{M}(t)$ gets inside 
the stable regime for any $t>T^*$; this is indicated in Statement $(iii)$ of Theorem \ref{maintheorem1}.
\end{itemize}
Finally, note that in Statement $(i)$ of Theorem \ref{maintheorem1}, the differential inequality \eqref{MonoFom} is shown to rigorously hold for all smooth and
quickly decaying (when $|v| \to \infty$) solutions of eq. (\ref{landau}) -- (\ref{13d}). It also rigorously holds for (smooth and
quickly decaying when $|v| \to \infty$) solutions
 to an approximated problem (that is, problem \eqref{approeq} described in subsection \ref{sub26}), of  equation (\ref{landau}) -- (\ref{13d}).
 Finally, when it is integrated with respect to time  (see \eqref{IntegratedMF}), it is shown in the proof of Proposition \ref{des} that it also
rigorously holds  for strong (that is, lying in $L^{\infty}_t(H^1_m)$ for $m$ large enough) solutions to the original equation (\ref{landau}) -- (\ref{13d}), such as those appearing (on suitable time intervals)  in Propositions \ref{localwell} and \ref{mainresult2}.  
%We refer readers to the proof of Theorem \ref{maintheorem1} for more details. For the sufficiently smooth initial data, we may expect from the smoothing estimate \eqref{smoothingH2} that \eqref{MonoFom} holds for the solution. This formula can also be generalized to  
%smooth solutions of the approximated equation \eqref{approeq}, uniformly in $\epsilon$.  
% We also remind the reader that the constant $C_\ell$ contains the information from the initial relative entropy and the initial moment.
%\smallskip

%\noindent$\bullet$ {\bf Comment on ``threshold of global regularity''.} The word `` {\it threshold} '' may be abused in our statement. Here we  use it to address that for the  new monotonic function $\mathcal{M}(t)$  there exists a critical value   which will induce the global regularity and global dynamics for the solution.

\subsubsection{\bf Comment on the non optimality of the presented results}
%global regularity for initial data below the threshold} 
%We recall that the condition $H(0) \, \big(\|h(0)\|_{\dot{H}^1}^2+ B^*\, \big)^{\f25} \le \f52$ is equivalent to $\mathcal{M}(0)\le0$, and gives the threshold for the global regularity. 
%\par 
We notice that
% this threshold 
%is in fact not optimal.  For instance,
  the lifespan of local wellposedness is not optimal  in Proposition \ref{localwell}. For example, it can be extended by the effect of the dissipation term $C_6\,(1+t)^k$, which is not used in the proof of this Proposition.
% and thus  improve the result slightly.
\par 
It is also possible to use Proposition \ref{localwell} in order to relax the condition $\mathcal{M}(0)\le0$ in statement (ii) of  Theorem \ref{maintheorem1}. Using the fact that the Landau equation with Coulomb interaction admits a local solution $f\in C([0,\mathcal{T}];\dot{H}^1)$  where $\mathcal{T}$ depends only on the initial data $f_0$ (see Proposition \ref{localwell} for more details), this condition can 
be transformed in
% In this situation, we claim that  the threshold for the global regularity can be relaxed to 
\ben\label{TGRC} \mathcal{M}(0)\le \f{C_6}{k+1}\big((1+\mathcal{T})^{k+1}-1\big). \een 
 Indeed,  thanks to estimates \eqref{MonoFom} and \eqref{monoquantity}, one gets 
\beno 
 \mathcal{M}(\mathcal{T})+C_6\, \int_0^{\mathcal{T}}(1+t)^{k}dt\le  \mathcal{M}(0),
\eeno
%This implies that if \eqref{TGRC} holds,
so that $\mathcal{M}(\mathcal{T})\le0$ and we can use statement $(ii)$ of Theorem \ref{maintheorem1} starting at time $\mathcal{T}$ (the
equation being invariant by translation in time).
%, it indicates that the solution $f\in L^\infty([0,\infty);L\log L\cap L^1_{55}\cap H^1)$. It ends the proof for the claim.
%\par
 %Finally we address that   in \eqref{smoesti} the $\dot{H}^1$ estimate of the solution depends only on  $K$  when the time leaves the initial moment. This is the key point to prove the impossibility of blowup after a finite time.
%\textcolor{red}{Text of the comment above to revise. Laurent}

\subsubsection{\bf Comment on the impossibility of blowup after a finite time.} This is a direct consequence of  inequality \eqref{MonoFom} since after the time $T^*$, the monotone functional $\mathcal{M}(t)$ will enter the stable regime (defined in Comment 1.4.1). 
\begin{itemize}
\item 
If the solution has not blown up  in $\dot{H}^1$ before the time $T^*$,
 then the solution will remain strong (that is, will lie in $\dot{H}^1$) for all time thanks to inequality \eqref{noblowupht}. Then  
%by Remark~\ref{L-infty} after Proposition~\ref{esqml}  
 thanks to the uniqueness result established in \cite{Fournier} and the regularity obtained in Proposition \ref{localwell}, the constructed solution is the unique strong solution with the initial data $f_0$ satisfying the conditions stated in Theorem \ref{maintheorem1}.   

\item Looking at definition \eqref{monoquantity}, we see that $\mathcal{M}(t)$ is still well-defined if $\|h(t)\|_{\dot{H}^1}=\infty$.
When such a blowup (in $\dot{H}^1$ norm) happens, the constructed solution is the  unique strong solution before the first blowup time, and 
becomes strong again after time $T^*$.
Note that in order to give a rigorous proof of these facts, we  apply the estimates obtained in this paper to  
solutions of an approximated problem and then pass to the limit.  
\end{itemize}

Finally, combining our result with the previous result in \cite{GGIV}, we see that
   the set of singular times for weak solutions is included in a subset of the interval $[\mathcal{T}, T^*]$ 
% \big(\f{1+k}{C_2}\mathcal{M}(0)+1\big)^{\f1{k+1}}-1]$
 whose  Hausdorff dimension is at most $1/2$. 
%We remind the reader that here $\mathcal{M}(0)$ satisfies  $ \mathcal{M}(0)> \f{C_6}{k+1}\big((1+\mathcal{T}-\eta)^{k+1}-1\big)$ due to \eqref{TGRC}. 
 
 \subsubsection{\bf Comment on the description of the potential blowup.} 
%To the best of our knowledge,
 Proposition \ref{des}  describes a potential blowup phenomenon for solutions to the Landau equation with Coulomb potential.  We recall that restrictions are given in \cite{GG1} and \cite{GGIV} on the possible appearance of such a blowup. Our lower bound for the blowup rate is given in terms of relative entropy.
% but seems not enough to estimate the Hausdorff dimension of the set of the singular times.  
Our upper bound enables to exclude a double exponential (that is, exponential of an exponential) growth of the $H^1$ norm of the solution close 
to the first blowup time. This bound  heavily depends on the number of initial moments
 which are  assumed.

 \subsubsection{\bf Comment on the  dependence of the coefficients appearing in the main Theorem with respect to the $L^1$ moments.} 
We can provide estimates for the explicit dependence of all coefficients in the Theorem \ref{maintheorem1}. Moreover we can 
extend the validity of this Theorem somewhat, when the initial data have less than $55$ moments. Indeed, when
 $f_0\in L^1_\ell$, let us define
   \beno 
q_{\ell, \theta} :=  - \frac{2\,\ell^2 - 25\,\ell + 57}{18\,(l-2)}\, \left(1 - \frac{\theta}{\ell}\right)  + \frac{\theta}{\ell},
\eeno
and choose $\ell>31$ and 
$\tau\in [31,\ell)$ such  that $q_{\ell,99/4}>7/4,\;q_{\ell,\tau}>0$.  Then
% motivated by Proposition \ref{esqml}, it is not difficult to 
one can check that it is possible to take
$k:= \min\{k_1,\f25k_2-\f{7}5\}$ with  $k_1=\f45q_{\ell,14/5}, 
k_2=q_{\ell,99/4}$,
 in such a way that estimate (\ref{MonoFom}) holds.
%and  $\mathcal{B}(x)=x^{-13}\exp\{7x^{-\f{10\tau}{\tau-31}}\}$.
In our main Theorem, we selected $\ell=55$ and $\tau=45$, for the sake of readability.

\subsubsection{\bf Comment on Landau equation with   very soft potentials.} We can generalize the result of Theorem~\ref{maintheorem1} 
to the Landau equation with very soft potential in the range $\gamma\in ]-3,-2]$, that is, when
\ben\label{agamma} a(z):=|z|^{\gamma+2}\left(Id -\frac{z\otimes z}{|z|^2} \right).\een 
Indeed, the main difference in the proof with the Coulomb case  lies in the estimate of the term
$\int\int_{|v-v_*| \le 1} f(v_*)\,|v-v_*|^{\gamma+1}|\na h(v)|\,|\na^2 h(v)|\, dv_*dv$, which appears
 in Proposition \ref{EstI1}. The Coulomb potential case $\gamma=-3$ 
is critical in the sense that it requires (in order to close the differential inequality (\ref{MonoFom})) the use of the Lorentz space $L^{3,1}$, since
% but not $L^3$ as we know that   
\beno 
\int\int_{|v-v_*| \le 1} f(v_*)\,|v-v_*|^{\gamma+1}|\na h(v)|\,|\na^2 h(v)|\,dv_*dv\le C\|f\|_{L^3}\|\na h\|_{L^2}\|\na^2 h\|_{L^2}
 \eeno
holds when $\gamma >-3$ but not when $\gamma = -3$.
\par 
We think therefore that when $\gamma \in ]-3,-2]$, it is possible to avoid the use of the Lorentz spaces
 and still get a closed inequality
 in the same spirit as inequality (\ref{MonoFom}).
%is false by the standard Hardy-Littlewood-Sobolev inequality. Fortunately the Lorentz norm is still controllable due to the entropy dissipation which is a key observation in our argument. 

% We finally remark that by rough computation, for very soft potentials, the following differential inequality holds for the propagation of the regularity in  $\dot{H}^1$ space:
 %\beno 
%&&  \f{d}{dt}  \|\nabla h\|_{L^2}^2+ C_1\,(1+t)^{k_1} \, \|\na h\|_{L^2}^{\f{14}5}\notag\\&&
%\le   \, C_3\, D(f)\, \|\nabla h\|_{L^2}^{\f45(-\gamma-2)+2}+C_2 (1+t)^{-k_2}.
 %\eeno
% We comment that it corresponds to Lemma \ref{esqml} with  $\gamma=-3$  and it yields the global smooth and bounded solution when $\gamma=-2$.  From that, we could derive the monotonicity formula:
 %\beno 
%&&   \f{d}{dt} \bigg[H(t)-\f52(-\gamma-2)^{-1}\big(\|h(t)\|_{\dot{H}^1}^2+ B\, (1+t)^{- k_2 +1}\big)^{-\f25(-\gamma-2)}   \bigg]\\&&+ \,C_6\, (1+t)^{k}\big(\|h(t)\|_{\dot{H}^1}^2+ B\, (1+t)^{- k_2 +1}\big)^{\f25(\gamma+3)} \le 0.
 %\eeno
 %In particular,
We also believe that if  $\gamma=-2-\eta$ with $\eta>0$ sufficiently small, then 
 the equation will generate a global smooth and bounded solution if initially $\|\nabla h_0\|_{L^2}^2\le  (C_1\,\eta)^{- C_2\,\eta^{-1}} - C_3$ for some $C_1$, $C_2$, $C_3>0$ (depending on $H(0)$). This is coherent with the existing theory of 
existence of global strong solutions when $\gamma \in [-2, 0[$, cf. \cite{KCWu} for example.
% This indicates the stability with respect to the parameter $\gamma$ around the value $-2$. We also address that such result can not be derived easily by the standard energy estimates.

\subsubsection{\bf Comment on the comparison with Leray's work for 3D Incompressible Navier-Stokes.} We recall that the 3D incompressible Navier-Stokes equations  reads
\ben\label{NSeq}\left\{
      \begin{array}{ll} &
\pa_t u+u\cdot \na u-\triangle u+\na p=0;\\&
\mathrm{div} \, u=0;\\&
u|_{t=0}=u_0.
\end{array}
    \right.
\een
In the classical work \cite{Leray}(see also~\cite{OP} and reference therein), Leray proved the following results:
\begin{itemize}
\item[(i)] If $\|u_0\|_{L^2}\|\na u_0\|_{L^2}\ll 1$,  the 3D incompressible Navier-Stokes equations admits 
a global smooth solution, which nowadays are called Leray solutions.
% We also remark that the condition $\|u_0\|_{L^2}\|\na u_0\|_{L^2}\ll 1$ is actually relevant to the so-called ``critical spaces'' for the equation. The typical example is $\dot{H}^{\f12}$ space.

\item[(ii)] He also considered the potential blowup phenomenon.  Using the lower bound of the blowup rate for the potential singularity,
 one can show that  the set of singular times for suitable weak solutions  has Hausdorff dimension at most $1/2$. 
%\par
%\textcolor{red}{I think this result exists somewhere (maybe not by Leray): it would be good to find a reference. Laurent} \textcolor{blue}{I added~\cite{Tsai}, it should include the reference.}
\end{itemize}
 For a  result about longtime regularity, we refer to~\cite{Pierre,Tsai}.
\medskip

 We are in a position to compare our results with Leray's.  
\medskip

% \begin{itemize}
 %	\item[(i)] 
 If we consider that the relative entropy $H$ plays for the Landau equation with Coulomb interaction the same role as the energy $\|u\|_{L^2}$ for the  Navier-Stokes equations, it is natural to compare the Leray condition $\|u_0\|_{L^2}\, \|\na u_0\|_{L^2}\ll 1$ 
to the condition $\mathcal{M}(0) \le 0$, written under the form $H(0)\,( \|h(0)\|_{\dot{H^1}}^2 + B^*)^{2/5} \le 5/2$. We see then that as 
in the Navier Stokes equation, the $L^2$ norm of the gradient of the solution plays a decisive role. Note however that no equivalent of 
the term $B^*$ exists in Leray's condition for Navier Stokes equation, which constitutes a significant difference.
\medskip

%for the Navier-Stokes equations to two kinds of conditins on the initial data  for Landau equation with Coulomb potential:
%\smallskip

%\noindent{\it Case 1:} $H(0)$ is small but $\|h_0\|_{\dot{H}^1}$ may be large;\quad  {\it Case 2:} $H(0)$ may be large but $\|h_0\|_{\dot{H}^1}$ is small.

%\smallskip

%For the initial data of {\it Case 1}, we can use Theorem \ref{maintheorem1}. Indeed, the condition 
% \[ \mathcal{M}(0)=H(0)- \f52\big(\|h(0)\|_{\dot{H}^1}^2+ B^*\, \big)^{-\f25}\leq 0\]
 The condition $\mathcal{M}(0) \le 0$  includes the case in which the  initial relative entropy  $H(0)$ is small,
  while $\|h(0)\|_{\dot{H}^1}$ may be  large.
 %  For the initial data corresponding to {\it Case 2}, we shall introduce below Theorem \ref{mainresult2}. 
% Thus Theorem~\ref{maintheorem1} together with Theorem~\ref{mainresult2} imply the existence of 
%Leray type solutions for Landau equation.  
%Finally we comment on the ``critical spaces'' for Landau equation. It is shown in  \cite{GGIV} that both Navier-Stokes and Landau 
%equations almost enjoy the same kind of scaling properties. We propose a semi-explicit example in Proposition~\ref{example} to show 
%that the ``critical spaces'' for Landau equation might be more complicated than that of Navier-Stokes. 
%\textcolor{red}{To be discussed.  The threshold condition can be written $H(0)\,( ||h(0)||_{\dot{H^1}}^2 + B^*)^{2/5} \le 5/2$, which is somehow very close to the condition $\|u_0\|_{L^2}\|\na u_0\|_{L^2}\ll 1$. I also think that it is clear that Case 1 exists: $H(0)$ can be small, while $||h(0)||_{\dot{H^1}}$ is large, since 
Note that such (normalized) initial data exist. Indeed one can take
%it is sufficient to take 
initial data $f(0)$ close (in weighted $L^{1}$) to the Maxwellian $\mu$, but having quick oscillations, so that  $\|h(0)\|_{\dot{H}^1}$  is large  (see Proposition \ref{example} for a concrete example).
 %Case 2 does not exist from my point of view: indeed if $||h(0)||_{\dot{H^1}}$ is small, then the gradient of $f(0)$ is close to the gradient of $\mu$ (in $L^2$), and since they both tend to $0$ at infinity,
%$f(0)$ and $\mu$ are also close in $L^2$, so that $H(0)$ is close to $0$. Laurent}

%The   example  is constructed as follows:
%\ben\label{Speexample} \qquad f_0=(1-\eta+\eta\epsilon^{2})^{\f32}\big[(1-\eta)\mu\big((1-\eta+\eta\epsilon^{2})^{\f12}v\big)+
%\eta\epsilon^{-3}\phi_0\big(\epsilon^{-1}(1-\eta+\eta\epsilon^{2})^{\f12}v\big)\big],\een
%where $\epsilon,\eta\ll1$ and $\phi_0\geq 0$ is a smooth function  satisfying the normalization \eqref{f0}. It is easy to see   that $f_{0}$ also  %satisfies the normalization \eqref{f0}. 
%By the calculation in Proposition~\ref{example}, we see that the $H(0)$ of $f_{0}$ is not big while its $\dot{H}^{\f12}$ norm (the critical space for %incompressible Naviver-Stokes equations) could be very large.

\medskip

Note that in Proposition~\ref{des}, we get not only a lower  bound, but also an upper bound for the rate at which a potential blowup occurs.
%  for the potential blow-up phenomenon. 
However, the lower bound is given in terms of relative entropy and thus probably cannot be used  to estimate  the size of 
singular times. We recall that the size of that set for the Landau equation with Coulomb potential is anyway estimated in \cite{GGIV}. 
% \end{itemize}
 
%\bigskip

%Our second result confirms that initial data corresponding to {\it Case 2} 
%generate a global smooth solution for the Landau equation with Coulomb potential.
% i.e., the other end of Leray type  solution.  
 %Though we do no know how to construct a semi-explicit initial data of that type, we logically have the result for such data below. 

\subsection{Sketch of the proof}

We present here the main ideas which are used in the proofs of Theorem \ref{maintheorem1} and the related results. In particular, we point out that the mechanisms which enable to build the global strong solutions to eq. (\ref{landau}) -- (\ref{13d}) when $\mathcal{M}(0) \le 0$ (in Theorem \ref{maintheorem1}, statement (ii))
 and when $||h(0)\lr{\cdot}^2||_{\dot{H}^1}$ is small (in Proposition \ref{mainresult2}),  
%  is quite different  to the initial data in  {\it Case 1} and in {\it Case 2} 
are quite different. 
\medskip

Let us first recall that thanks to a previous study of the large time behavior of the Landau equation with Coulomb potential (cf. \cite{CDH}), the $L^1$ moments of $h=f-\mu$ decrease with a power law (cf. Theorem \ref{thm:decay}).
\par 
As a consequence,
by interpolation, we see that the dissipation  of  $\dot{H}^1$ energy typically
 increases as time goes. Roughly speaking, for some $C, k_1>0$, this dissipation is lower bounded in the following way:
  \beno  \|\na^2 h\|_{L^2_{-\f32}}^2 \ge  C\, \|h\|_{L^1_{15/4}}^{-\f45}\|\na h\|_{L^2}^{\f{14}5}\ge C\, (1+t)^{k_{1}}\|\na h\|_{L^2}^{\f{14}5}.
\eeno

\noindent $\bullet$\,{\bf When the initial data is far from equilibrium regime (measured in terms of $\dot{H}^{1}$ norm):} In this situation,  the main challenge is to show that the nonlinear terms can be controlled. Indeed,
by interpolation,
 the behavior of the nonlinear term with respect to the $\dot{H}^1$ energy 
  is of the same order as the dissipation term in the following sense:
  \beno \mbox{Nonlinearity}\precsim D(f)\|\na h\|_{L^2}^{\f{14}5}. \eeno
 These observations suggest that   the competition between the dissipation and  nonlinearity in $\dot{H}^1$ energy can be characterized as 
 $(1+t)^{k_{1}}\|\na h\|_{L^2}^{\f{14}5}$ versus $D(f)\|\na h\|_{L^2}^{\f{14}5}$.  Since it is expected that the dissipation 
 will dominate the nonlinear term after some time (remembering that $(1+t)^{k_1} \to +\infty$ and 
$\int_0^{\infty} D(f)(s)\,ds < +\infty$), one can understand the emergence of the new monotonicity formula
that we propose.
%to describe the propagation of regularity. 
 %Then the  different scenarios comes from this new formula under the different initial conditions. 
 The detailed arguments are included in Section 2.
 \medskip

\noindent $\bullet$\,{\bf When the initial data is close to equilibrium regime (measured in terms of $\dot{H}^{1}$ norm):} 
In this situation,  we have
  \beno  \mbox{Tail of linear term plus nonlinearity}\precsim  \|\na h\|_{L^2}^{4}+\|\na h\|_{L^2}^2, \eeno
%\textcolor{blue} { 
as we can observe from equation~\eqref{weighted-h13} by neglecting the weights. 
Since we assumed that the $\dot{H}^1$ norm of the initial data is sufficiently small, we  see that
 the competition  occurs between $(1+t)^{k_{1}}\|\na h\|_{L^2}^{\f{14}5}$ and $\|\na h\|_{L^2}^2$. 

Suppose now that $\|\na h_0\|_{L^2}^2\sim \epsilon$. Then the smallness of $\|\na h\|_{L^2}^{2}$ can be kept at least for
an interval of time of length $|\ln \epsilon|$ (cf. the proof of Proposition \ref{mainresult2}).  It implies that at some point,
 the dissipation will be lower bounded 
% behave like
 in the following way:
\beno
 (1+t)^{k_{1}}\|\na h\|_{L^2}^{\f{14}5}\ge C\, (1+|\log \epsilon|)^{k_{1}}\|\na h\|_{L^2}^{\f{4}5} \|\na h\|_{L^2}^2.
 \eeno
%\textcolor{red}{Please give some explanation for the above statement. Laurent}
Then, when
 $\|\na h\|_{L^2}^{\f{4}5}$ is not small, 
% = o(|\ln \epsilon|^{-k_2})$ with $k_2<k_1$, 
the dissipation still prevails and prevents a blowup of the $\dot{H}^1$ norm.
% in the way:
%\beno (1+t)^{k_{1}}\|\na h\|_{L^2}^{\f{14}5}\ge (1+|\ln \epsilon|)^{k_{1}-k_2}\|\na h\|_{L^2}^2. \eeno
%Thus it indicates that from this moment  the $\dot{H}^1$ energy of the solution will never exceed its peak.
We refer readers to the content of Section 4 for detailed and rigorous arguments. Note however that in the description above, weights are not taken into account, whereas they play a significant role in the proof of Proposition \ref{mainresult2}. Finally, we refer to \cite{GG2} for extra considerations on 
the competition between dissipation and nonlinearities.

\section{$\dot{H}^1$ estimate and the proof of Theorem \ref{maintheorem1}}

This section is devoted to the $\dot{H}^1$ estimate for the Landau equation  with Coulomb potential, 
which leads to the monotonicity formula (\ref{MonoFom}). We first provide 
a set of {\it a priori} estimates for the terms appearing in the equation (this is done in Subsections \ref{sub21} to \ref{sub25}). Then we
 show that all estimates rigorously hold  by passing to
 the limit in an approximated problem (this is done in Subsection \ref{sub26}), which enables us to  complete  the proof of Theorem \ref{maintheorem1}.

\subsection{Decomposition of the derivative in time of the $\dot{H}^1$ norm of the solutions to the Landau equation} \label{sub21}

To make full use of the results on  the long-time behavior of the solution (cf. Theorem \ref{thm:decay}),
 we write the Landau equation with Coulomb potential as follows (at the
formal level), setting
 $h:=f-\mu$, with $\mu$ defined by (\ref{Defmu}):
% so that eq. (\ref{landau}) becomes
\begin{equation}\label{Eqh}
\partial_t h=Q(f,h)+Q(h,\mu) .
\end{equation} 

Then we focus (at the formal level) on the $\dot{H}^1$ norm of $h$. We write the equation (for $k=1,2,3$) satisfied by $\partial_k h$:
\begin{equation}
\partial_t (\partial_k h)=Q(f,\partial_k h)+Q(\partial_k f, h)+Q(\partial_k h, \mu)+Q(h,\partial_k \mu).
\end{equation}
Then we multiply it by $\partial_k h$, integrate with respect to $v$, and sum over all $k$.
 It gives
\ben\label{i14}
 \f12\f{d}{dt}\|\nabla h\|_{L^2}^2=I_1+I_2+I_3+I_4,
\een 
where $I_1, I_2, I_3$ and $I_4$ are defined (and subdivided) as follows:
\begin{enumerate}
\item $I_1 := \sum_{k=1}^3 \int_{\R^3} Q(f,\partial_k h)\,  \partial_k h\, dv $. We also write 
\begin{equation} \label{i1}
I_1 :=  - I_{1,1} + I_{1,2},
\end{equation}
where
$$  I_{1,1}  :=   \sum_{k=1}^3 \int_{\R^3}  (a * f) : \nabla \pa_k h \otimes \nabla \pa_k h \, dv, 
\quad I_{1,2}  :=   \sum_{k=1}^3 \int_{\R^3}  (b * f) \cdot \nabla \pa_k h \,  \pa_k h \, dv . $$

\item $I_2= \sum_{k=1}^3 \int_{\R^3} Q(\partial_k f, h)\,  \partial_k h\, dv $. We also write 
\begin{equation} \label{i2}
I_2 := - I_{2,1} +  I_{2,2},
\end{equation}
where
$$  I_{2,1} : =   \sum_{k=1}^3 \int_{\R^3}  (a * \pa_k f) : \nabla \pa_k h \otimes \nabla  h \, dv    =  \sum_{k=1}^3 \int_{\R^3}  (\pa_k a *  f) : \nabla \pa_k h \otimes \nabla  h \, dv , $$
\beno I_{2,2}  &:=&   \sum_{k=1}^3 \int_{\R^3}  (b * \pa_k f) \cdot \nabla \pa_k h \,  h \, dv   =  \sum_{k=1}^3 \int_{\R^3}  (\pa_k b * f) \cdot \nabla \pa_k h \,  h \, dv\\
&=& \sum_{k=1}^3    \sum_{i=1}^3  \bigg( - \int_{\R^3} (\pa_i\pa_k b_i * f)  \pa_k h \,  h \, dv  - \int_{\R^3} (\pa_k b_i * f)  
\pa_k h \, \pa_i h \, dv  \bigg)\\
&=& \sum_{k=1}^3\bigg( -8\pi\int_{\R^3} f(h\pa_k^2h +(\pa_k h)^2) \,dv+\int_{\R^3}(b*f)\cdot(\nabla h\pa_k^2h+\nabla
\pa_k h \pa_k h)dv\bigg) \\
&\le& \sum_{k=1}^3\bigg( -8\pi\int_{\R^3} f\, h\pa_k^2h\,dv+\int_{\R^3}(b*f)\cdot(\nabla h\pa_k^2h+\nabla \pa_kh \pa_kh)dv\bigg). \eeno  
%Here we use the integration by parts and the fact \eqref{Defbc}.

\item $I_3 := \sum_{k=1}^3 \int_{\R^3} Q(\partial_k h, \mu)\,  \partial_k h\, dv $. We also write 
\begin{equation} \label{i3}
I_3 := - I_{3,1} + I_{3,2},
\end{equation}
where
 $I_{3,1}  :=   \sum_{k=1}^3 \int_{\R^3}  (a * \pa_k h) : \nabla \pa_k h \otimes \nabla  \mu \, dv$,  
%=    \sum_{i=1}^3 \int_{\R^3}  (\pa_i a *  h) : \nabla \pa_i h \otimes \nabla  \mu \, dv 
\beno   I_{3,2} & :=&   \sum_{k=1}^3 \int_{\R^3}  (b * \pa_k h) \cdot \nabla \pa_k h \,  \mu \, dv \\
 &=&   \sum_{k=1}^3    \sum_{i=1}^3  \bigg( - \int_{\R^3} (\pa_i b_i *  \pa_k h) \, \pa_k h \, \mu dv  - \int_{\R^3}  (b_i *  \pa_k h) \, \pa_k h \, \pa_i \mu dv  \bigg)    \\
&=& \sum_{k=1}^3\bigg( 8\pi\int_{\R^3} \mu|\pa_kh|^2 \,dv+\int_{\R^3}(b*h)\cdot(\pa_k h\pa_k \nabla\mu+\pa^2_k h \nabla \mu)dv\bigg).
\eeno

\item $I_4 := \sum_{k=1}^3 \int_{\R^3} Q(h, \partial_k \mu)\,  \partial_k h\, dv $. We  also write 
\begin{equation} \label{i4}
I_4 := - I_{4,1} + I_{4,2},
\end{equation}
where
$$  I_{4,1}  :=   \sum_{k=1}^3 \int_{\R^3}  (a * h) : \nabla \pa_k h \otimes \nabla \pa_k \mu \, dv,\quad   
I_{4,2}  :=   \sum_{k=1}^3 \int_{\R^3}  (b * h) \cdot \nabla \pa_k h \,  \pa_k \mu \, dv   .  $$
\end{enumerate}
\medskip
 
In subsections \ref{sub21} to \ref{sub25}, the {\it{a priori}} estimates are proven as if the considered functions are smooth and quickly decaying (when $|v| \to \infty$). They are used later for solutions of an approximated problem which satisfy those properties.

\subsection{Coercivity estimate for $I_{1,1}$}

  In order to treat the term $I_{1,1}$, we prove the following (rather classical) coercivity estimate:

\begin{prop} \label{coercivity}
For all $j\in \{1,2,3\}$, $m \in \R$, $f \ge 0$, $p\in  W^{1,1}_{loc}(\R^3)$, we have 
\begin{equation}\label{rr}
 \int_{\R^3} |\nabla p(v)|^2 \, \langle v\rangle^{m-3} \,dv \le  4\, \|f\|_{L^1_5(\R^3)}\,  \, A_j(f)^{-2} 
\end{equation}
$$ \times\, \int_{\R^3} \int_{\R^3} |v-v_*|^{-3} \,
\bigg\{ |v-v_*|^2 - (v-v_*) \otimes (v-v_*) \bigg\} :  \nabla p(v) \otimes \nabla p(v) \, f(v_*)\, \langle v\rangle^m  \, dv dv_* 
$$
where $A_j(f):=\int_{\R^3} fv_j^2dv$.
\end{prop}

\begin{proof}
We denote
$$ q_{ij}(v,v_*) = (v_i - v_{*i})\, \pa_j p(v) -  (v_j - v_{*j})\, \pa_i p(v). $$
 As observed in~\cite{TV},  by choosing a suitable system of coordinates, we can assume that
 $\int_{\R^3} f(v)\, v_{i}v_{j}dv=\delta_{ij}A_{i}(f)$. Then 
$$ \int_{\R^3} q_{ij}(v,v_*)  \, f(v_*)\, v_{*j}\,dv_* = A_j(f)\, \pa_i p(v), $$
and for all $n\in \R$, $i,j \in \{1,2,3\}$, $i\neq j$, thanks to Cauchy-Schwarz inequality,
$$  A_j(f)^2\, \int_{\R^3} |\pa_i p(v)|^2\, \langle v\rangle^n \, dv \le \int_{\R^3} \langle v\rangle^n \bigg\{ \int_{\R^3} q_{ij}(v,v_*)^2  \, f(v_*)\,  \frac{\langle v\rangle^{m-n}}{|v-v_*|^3} \,dv_* $$
$$ \times \,\,   \int_{\R^3}  f(v_*)\,  |v-v_*|^3\,  |v_{*j}|^2\, \langle v\rangle^{n-m} \,dv_* \bigg\}\, dv . $$
We end up with estimate (\ref{rr}) by using $n= m-3$, the bound
$$  \sup_{v \in \R^3} \int_{\R^3}  f(v_*)\,  |v-v_*|^3\,  |v_{*j}|^2\, \langle v\rangle^{-3} \,dv_* \le 4\,\|f\|_{L^1_5(\R^3)} , $$
 and the identity
$$  \bigg\{ |v-v_*|^2 - (v-v_*) \otimes (v-v_*) \bigg\} :  \nabla p(v) \otimes \nabla p(v) = |(v-v_*) \times \nabla p(v)|^2. $$
\end{proof}

 We now state two corollaries which can easily be obtained from Proposition \ref{coercivity}.

\begin{cor} \label{coercivity2}
Let $f \ge 0$ be such that $\int_{\R^3} f(v) \, dv =1$,   $\int_{\R^3} f(v)\, |v|^2 \, dv =3$, and such that $ \|f\|_{L^1_{5}(\R^3)} +  \| f\|_{L\,\log L} \le K$ for some $K>0$. We denote $h=f - \mu$. Then there exists a constant $C(K)>0$ (depending only on $K$) such that for all $h  \in L^1_{loc}(\R^3)$, $m \in \R$,
\begin{equation}\label{rrbis}
 %I_{1,1}  =   
\sum_{k=1}^3 \int_{\R^3}  (a * f) : \nabla \pa_k h \otimes \nabla \pa_k h \, \langle v\rangle^m\, dv 
  \ge C(K) \, ||\nabla^2 h||^2_{L^2_{m/2-3/2}(\R^3)}  . 
\end{equation}

\end{cor}

\begin{proof}
We just observe that under the assumption $\int_{\R^3} f(v) \, dv =1$,   $\int_{\R^3} f(v)\, |v|^2 \, dv =3$, and $ \| f\|_{L\,\log L} \le K$, the quantity $A_j(f) $ is bigger than some strictly positive quantity (which depends only on $K$). 
\end{proof}

\begin{cor} \label{coercivity3}
Let $f \ge 0$ be such that $\int_{\R^3} f(v) \, dv =1$,   $\int_{\R^3} f(v)\, |v|^2 \, dv =3$, and such that $ \|f\|_{L^1_{5}(\R^3)} +  \| f\|_{L\,\log L} \le K$ for some $K>0$.  We denote $h=f - \mu$. Then there exist constants $C(K), C^*(K)>0$ depending only on $K$,
such that for all $h  \in L^1_{loc}(\R^3)$, $m \in \R$,
\begin{equation}\label{rrter}
  I_{1,1} \ge  C(K) \,\|\na^2 h\|_{L^2_{-\frac32}}^2  + C(K) \, \|h\|_{L^1_{15/4}}^{-\f45} \, \|\na h\|_{L^2}^{\f{14}5} - C^*(K) \,\|h\|_{L^1}^2.
\end{equation}

\end{cor}

\begin{proof}
Note first  that by taking $m=0$ in Corollary \ref{coercivity2}, we get that 
%\textcolor{red}{Here a formula was removed from the previous version. Is it a misprint? Laurent}
$$ I_{1,1} \ge C(K) \,  ||\nabla^2 h||^2_{L^2_{-3/2}}.$$
 Then using Proposition \ref{interpH1a}, with $m=0$, we see that for some constant $C>0$,
%, C^*>0$,
$$  \|\na h\|_{L^2} \le C\, \|h\|_{L^1_{15/4}}^{\f27}\, (\|h\|_{L^1} + \|\na^2 h\|_{L^2_{-\frac32}})^{5/7}. $$
This inequality implies that (for some constant $C^*>0$)
$$ \|\na^2 h\|_{L^2_{-\frac32}}^2 \ge C \, \|h\|_{L^1_{15/4}}^{-\f45} \, \|\na h\|_{L^2}^{\f{14}5} - C^*\, ||h||_{L^1}^2 , $$
whence estimate (\ref{rrter}).
\end{proof}

\subsection{Estimates for the remainder terms.} \label{sub23}

\subsubsection{Estimates for  $I_3$ and $I_4$}

 \begin{prop}\label{EstI0} Let $f \ge 0$ and $h = f - \mu$.
%such that $ \|f\|_{L^1_{l}(\R^3)} +  \|f\|_{L\log L}\le K$ with $l\ge 31$.
Then  for all $\eta \in ]0,1[$,
% there exists $C_{\eta}>0$ such that
\begin{equation}\label{i3i4}
 I_3 + I_4 \le   C\,\eta^{-1} \, \|h\|_{L^2_2}^2+ C \, \|\nabla h\|_{L^2}^2 + \frac{\eta}4 \, \|\nabla^2 h\|^2_{L^2_{-3/2}},
\end{equation}  
for some absolute constant $C>0$.
 \end{prop}

\begin{proof} 
 Indeed, for some constant $C>0$,
$$  I_3 + I_4 \le    C \int_{\R^3}    |\nabla h|^2  \mu  \,dv +
 C \iint_{\R^6}    |v-v_*|^{-2}  |h_*| (|\nabla^2 h|+|\nabla h|)  \mu^{\f12}  \,dv_*dv $$
$$  \le C \int_{\R^3}    |\nabla h|^2   \,dv  + \frac{\eta}4 \, \|\nabla^2 h\|^2_{L^2_{-3/2}} 
+ C\,\eta^{-1}\, \int_{\R^3}  \bigg( \int_{\R^3}   |v-v_*|^{-2}  |h_*| dv_* \bigg)^2 \, \lr{v}^3 \,\mu\, dv $$
$$ +\, C  \int_{\R^3}  \bigg( \int_{\R^3}   |v-v_*|^{-2}  |h_*| dv_* \bigg)^2 \,  \,\mu\, dv .$$
Then, we see that 
$$  \int_{\R^3}  \bigg( \int_{\R^3}   |v-v_*|^{-2}  |h_*| dv_* \bigg)^2 \, \lr{v}^3 \,\mu\, dv 
\le  2  \int_{\R^3}  \bigg( \int_{  |v-v_*| \le 1 }   |v-v_*|^{-2}  |h_*| dv_* \bigg)^2 \, \lr{v}^3 \,\mu\, dv $$
$$ + \,2  \int_{\R^3}  \bigg( \int_{  |v-v_*| \ge 1 }   |v-v_*|^{-2}  |h_*| dv_* \bigg)^2 \, \lr{v}^3 \,\mu\, dv $$
$$  \le  C \int_{v_* \in \R^3} |h_*|^2 \,  \int_{ |v-v_*| \le 1}    |v-v_*|^{-2}  \, \lr{v}^3 \, \mu \, dv  \,dv_* 
+ 2 \,\bigg( \int_{\R^3} |h_*|\, dv_* \bigg)^2 \,  \int_{\R^3} \lr{v}^3\,  \mu \, dv . $$
We conclude thanks to Cauchy-Schwarz inequality.  
\end{proof}
\bigskip

%Due to the fact $|\pa_ib*f|+|b*f|\lesssim   |\cdot|^{-2}*f$, by Hardy-Littlewood-Sobolev inequality, we see that for all $\eta \in ]0,1[$, there exists $C_{\eta}>0$ such that 
 Then we observe that for some constant $C>0$,
 \begin{equation}\label{i12i2}
  I_{1,2} + I_{2}  \le  C\, \iint_{\R^6} |v-v_*|^{-2}f_*|\nabla h||\nabla^2 h|dv_*dv+ C\,\int f|\nabla^2h|\, |h| dv ,
\end{equation}
and we define (with the constant $C$ being the same as in the inequality above)
\ben\label{DefI}
\mathcal{I}:= C\, \iint_{\R^6} |v-v_*|^{-2}f_* \,|\nabla h|\, |\nabla^2 h| \, dv_*dv,
\een
and 
\ben\label{DefII}
\mathcal{II} := C\,  \int_{\R^3}  f\, |\nabla^2h|\, |h| \, dv .
\een
 
\subsubsection{Estimate for $\mathcal{I}$.} 

 We state the:
% following proposition:

\begin{prop}\label{EstI1} Let $f \ge 0$ (and $h= f - \mu$) such that  $ \|f\|_{L^1_{2}(\R^3)} = 4$ and  $ \|f\|_{L^1_{\tau}(\R^3)} +  
 \|f\|_{L\,\log L} \le K$ for some $K>0$, $\tau >31$.
Then for all $\eta \in ]0,1[$ and $A>1$, the following estimate holds: 
 \beno 
\mathcal{I}&\le& \frac{\eta}2\, \|\nabla^2 h\|_{L^2_{-\f32}}^2+\big(C\,\eta+C(K)\eta^{-1}(\log A)^{-\f{\tau-31}{5\tau}}\big)\big(D(f)+1\big)\|\nabla h\|_{L^2}^{\f{14}5}\\
&& + \, C\, (\eta + \eta^{-1} + \eta^{-3}A^2)\, \|\nabla h\|_{L^2_6}^2 +C(K)\, \eta^{-1}(\log  A)^{-\f{\tau-31}{5\tau}}\|\nabla h\|_{L^2}^2,
 \eeno
where $C(K)>0$ is a constant depending only on $K$, and $C>0$ is an absolute constant. 
 \end{prop}
\begin{proof} 
We  write $\mathcal{I} = \mathcal{I}_1  + \mathcal{I}_2$, with 
\ben\label{DefI12}
\mathcal{I}_1:=\iint_{|v-v_*|\le 1} |v-v_*|^{-2}f_* \,|\nabla h|\, |\nabla^2 h| \, dv_*dv, \\ \mathcal{I}_2:=\iint_{|v-v_*|\ge 1} |v-v_*|^{-2}f_* \,|\nabla h|\, |\nabla^2 h| \, dv_*dv. \nonumber 
\een
  %  separate the above integration into two regions: $\{|v-v_*|\le 1\}$ and  $\{|v-v_*|\ge 1\}$ and denote the corresponding integrations by   $\mathcal{I}_1$ and $\mathcal{I}_2$.  For $\mathcal{I}_2$, it is easy to
We see that
\begin{equation}\label{bi2}
\begin{split}
\mathcal{I}_2&\le C\iint   \lr{v-v_*}^{-2}f_*|\nabla h| |\nabla^2 h|dv_*dv\\
&\le C \iint   \lr{v}^{-2}\lr{v_*}^{2}f_*|\nabla h| |\nabla^2 h|dv_*dv\\
&\le
 C\, \|f\|_{L^1_2}\|\na h\|_{L^2}\|\na^2 h\|_{L^{2}_{-2}} \\
&\le
 \frac{\eta}4 \, \|\nabla^2 h\|_{L^2_{-2}}^2 + C\, \eta^{-1}\, \|\na h\|_{L^2}^2 .
\end{split}
 \end{equation}
\medskip

We now turn to  $\mathcal{I}_1$. We note first  that in the region $\{|v-v_*|\le1\}$ 
holds the estimate  $ \frac1{\sqrt{3}}\, \langle v\rangle \le \langle v_*\rangle  \le \sqrt{3} \, \langle v\rangle$. Thus by Cauchy-Schwartz inequality, we have
\beno
\mathcal{I}_1&\le& C\, \bigg(\underbrace{\iint_{|v-v_*|\le 1} |v-v_*|^{-2}(f_*\langle v_*\rangle^{-3})|\nabla h| |\nabla^2 h\langle v\rangle^{-\f32}|dv_*dv}_{\eqdefa B_1}\bigg)^{\f12}\\&&\times\bigg(\underbrace{\iint_{|v-v_*|\le 1} |v-v_*|^{-2}(f_*\langle v_*\rangle^{6})|\nabla h| |\nabla^2 h\langle v\rangle^{-\f32}|dv_*dv}_{\eqdefa B_2}.\bigg)^{\f12}
 \eeno
We observe that for $k\in ]0,3[$, $ 1_{\{|\cdot| \le 1\}} \, |\cdot|^{-k}\in L^{3/k,\infty}$. Therefore, we get
\beno  B_1\le C\, \|f\|_{L^{3,1}_{-3}}\|\na h\|_{L^2}\|\na^2h\|_{L^2_{-\f32}}, \eeno
thanks to O'Neil inequality (cf. Proposition~\ref{oneil} in the Appendix). 
 
Concerning $B_2$ (and for any $A>1$), we split $f$ into two parts: $f_A=f\chi(f/A)$ and $f^A:=f-f_A$, where $\chi$ is a nonnegative $C^1$ function satisfying that $\chi=1$ in $B_1$  and $\chi=0$ outside of $B_2$. Thanks to this decomposition, we get
\beno 
B_2\le C\,\|f^A\|_{L^{3,1}_{6}}\|\na h\|_{L^2}\|\na^2h\|_{L^2_{-\f32}}+ C\,A\|\na h\|_{L^2_6}\|\na^2h\|_{L^2_{-\f32}}.
\eeno
Here we use again O'Neil inequality (cf. once again Proposition~\ref{oneil} in the Appendix) for $f^A$, and the bound $f_A\le C\,A$.
% for $f_A$.

Putting together the estimates for $B_1$ and $B_2$ yields 
$$ \mathcal{I}_1 \le C \, \big(\|f\|_{L^{3,1}_{-3}}^{\f12}\|f^A\|_{L^{3,1}_6}^{\f12}\|\na h\|_{L^2}\|\na^2h\|_{L^2_{-\f32}}+A^{\f12}\|f\|_{L^{3,1}_{-3}}^{\f12}\|\na h\|_{L^2}^{\f12}\|\na h\|_{L^2_6}^{\f12}\|\na^2h\|_{L^2_{-\f32}}\big) $$
\begin{equation} \label{bi1}  
\le \frac{\eta}4\,  \|\na^2h\|_{L^2_{-\f32}}^2+\eta \|f\|_{L^{3,1}_{-3}}^2\|\na h\|_{L^2}^2+ C\, \big(\eta^{-1}\|f\|_{L^{3,1}_{-3}}\|f^A\|_{L^{3,1}_6}\|\na h\|_{L^2}^2+\eta^{-3}A^2\|\na h\|_{L^2_6}^2\big).
\end{equation}
In what follows, we estimate  the quantities $ \|f\|_{L^{3,1}_{-3}}^2\|\na h\|_{L^2}^2$ and $\|f\|_{L^{3,1}_{-3}}\|f^A\|_{L^{3,1}_6}\|\na h\|_{L^2}^2$.
\bigskip

\underline{\it Estimate of $\|f\|_{L^{3,1}_{-3}}^2\|\na h\|_{L^2}^2$:}  Remembering that $\|f\|_{L^1_2} = 4$ and using Proposition \ref{interpL31},
we get the estimate
\ben\label{f31e}  \|f\|_{L^{3,1}_{-3}}\le C\, \|f\lr{\cdot}^{-3}\|_{H^1}^{\f45}\le C\, (\|\na h\|_{L^2}^{\f45}+1),\een
where we used the interpolation estimate
 %facts $f=h+\mu$ and
 \begin{equation} \label{ies}
\|h\|_{L^2}\le C\,(\|h\|_{L^1}+\|\na h\|_{L^2})\le C\, (1+\|\na h\|_{L^2}).
\end{equation}
It yields 
\ben\label{aut}   \|f\|_{L^{3,1}_{-3}}^2\|\na h\|_{L^2}^2&\le&   C\,(\|\na h\|_{L^2}^{\f45}+1)\|f\|_{L^{3,1}_{-3}}\|\na h\|_{L^2}^2.\een  
 We end up with the bound
\beno
 \|f\|_{L^{3,1}_{-3}}^2\|\na h\|_{L^2}^2&\le& C\, ( \|f\|_{L^{3,1}_{-3}}\|\na h\|_{L^2}^{\f{14}{5}}+\|f\|_{L^{3,1}_{-3}}\|\na h\|_{L^2}^2)\\
&\le&
 C\, ( D(f)\|\na h\|_{L^2}^{\f{14}5}+\|\na h\|_{L^2}^{\f{14}5}+\|\na h\|_{L^2}^2),
\eeno
 where we used the estimates \eqref{Dfdissipation} for the first term, and \eqref{f31e} for the second term.
\bigskip

\underline{\it Estimate of $\|f\|_{L^{3,1}_{-3}}\|f^A\|_{L^{3,1}_6}\|\na h\|_{L^2}^2$.} Thanks to the definition of $f^A$ and the interpolation estimate (\ref{ies}), we first see that
\begin{equation}\label{faabi} 
 \|f^A\|_{H^1}\le C\, \sqrt{1+A^{-2}}\, (\|\na h\|_{L^2}+1).
\end{equation}
For $R>0$ and $A>1$, we  know that 
$$\|f^A\|_{L^1_{31}}\le R^{31}(\log A)^{-1}\, \int_{f \ge 1} f\, \log f \, dv +R^{-(\tau-31)}\|f\|_{L^1_\tau}, $$
 so that  
\begin{equation}\label{faa}
 \|f^A\|_{L^1_{31}}\le C(K)(\log A)^{-\f{\tau-31}{\tau}}. 
\end{equation}
Using Proposition \ref{interpL31}, one gets
$$ \|f\|_{L^{3,1}_{-3}}\|f^A\|_{L^{3,1}_6}\|\na h\|_{L^2}^2\le  C\,\|f\|_{L^{3,1}_{-3}}\|f^A\|_{L^1_{31}}^{\f15}\|f^A\|_{H^1}^{\f45}\|\na h\|_{L^2}^2 . $$
Then, using estimates (\ref{faabi}) and (\ref{faa}), 
   \begin{equation}\label{aut2}
 \|f\|_{L^{3,1}_{-3}}\|f^A\|_{L^{3,1}_6}\|\na h\|_{L^2}^2\le   C(K)\,(\log A)^{-\f{\tau-31}{5\tau}} \|f\|_{L^{3,1}_{-3}}\big(\|\na h\|_{L^2}^{\f{14}5}+\|\na h\|_{L^2}^2\big). 
\end{equation}
Finally, using both estimates (\ref{Dfdissipation}) and (\ref{f31e}), we end up with 
\begin{equation}\label{nez}
 \|f\|_{L^{3,1}_{-3}}\|f^A\|_{L^{3,1}_6}\|\na h\|_{L^2}^2\le       C(K)\,(\log A)^{-\f{\tau-31}{5\tau}} \bigg[(D(f)+1)\|\na h\|_{L^2}^{\f{14}5}+\|\na h\|_{L^2}^2\bigg].  
\end{equation}
%From these two estimates with Proposition \ref{interpL31}, \eqref{Dfdissipation} and \eqref{f31e}, we derive 
%\ben\label{L31-36}  &&\|f\|_{L^{3,1}_{-3}}\|f^A\|_{L^{3,1}_6}\|\na h\|_{L^2}^2\le  cst\,\|f\|_{L^{3,1}_{-3}}\|f^A\|_{L^1_{31}}^{\f15}\|f^A\|_{H^1}^{\f45}\|\na h\|_{L^2}^2\notag
%\\ &&\le   C(K)\,(\ln A)^{-\f{l-31}{5l}} \|f\|_{L^{3,1}_{-3}}\big(\|\na h\|_{L^2}^{\f{14}5}+\|\na h\|_{L^2}^2\big)\notag
 %\\ &&\le   C(K)\,(\ln A)^{-\f{l-31}{5l}} \big((D(f)+1)\|\na h\|_{L^2}^{\f{14}5}+\|\na h\|_{L^2}^2\big). 
%\een
Finally, we see that 
 \begin{equation}\label{esti1}
  \mathcal{I}_1 \le  \frac{\eta}4\,  \|\na^2h\|_{L^2_{-\f32}}^2 + C\,\eta^{-3}A^2\|\na h\|_{L^2_6}^2 + C\, \eta\,  ( D(f)\|\na h\|_{L^2}^{\f{14}5}+\|\na h\|_{L^2}^{\f{14}5}+\|\na h\|_{L^2}^2)
\end{equation}
$$   
 +\,  C(K)\, \eta^{-1}\, (\log A)^{-\f{\tau-31}{5\tau}} \bigg[(D(f)+1)\|\na h\|_{L^2}^{\f{14}5}+\|\na h\|_{L^2}^2\bigg] $$
$$ \le \frac{\eta}4\, \|\nabla^2 h\|_{L^2_{-\f32}}^2+\big(C\,\eta+C(K)\eta^{-1}(\log A)^{-\f{\tau-31}{5\tau}}\big)\big(D(f)+1\big)\|\nabla h\|_{L^2}^{\f{14}5}$$
$$ + \,C\, (\eta + \eta^{-3}A^2)\, \|\nabla h\|_{L^2_6}^2 +C(K)\eta^{-1}(\log A)^{-\f{\tau-31}{5\tau}}\|\nabla h\|_{L^2}^2 . $$
We deduce the desired result by combining this estimate with the estimate for $ \mathcal{I}_2$.
\end{proof}

\subsubsection{Estimate for $\mathcal{II}$} 

%Since
%\ben\label{DefII}
%\mathcal{II}:= \int f|(\pa_i^2h h+|\pa_ih|^2)|dv,
%\een
We now prove  the following bound:
\begin{prop}\label{EstII}  Consider $f \ge 0$  (and $h= f - \mu$) such that  $ \|f\|_{L^1_{2}(\R^3)} = 4$ and  $ \|f\|_{L^1_{\tau}(\R^3)} +   \|f\|_{L\log L} \le K$ for some $K>0$, $\tau > 31$.
Then for all $\eta \in ]0,1[$ and $A>1$, the following estimate holds: 
%such that $ \|f\|_{L^1_{l}(\R^3)} +  \|f\|_{L\log L}\le K$ with $l\ge 31$. Then
 \beno 
\mathcal{II}&\le& \frac{\eta}4\,\|\nabla^2 h\|_{L^2_{-\f32}}^2+ C(K)\,\eta^{-1}(\log A)^{-\f{\tau-31}{5\tau}} \big(D(f)+1\big)\|\nabla h\|_{L^2}^{\f{14}5}+ C\, \eta^{-1}A^2\|h\|_{L^2_{\f32}}^2\\
&&+\, C(K)\, \eta^{-1}(\log A)^{-\f{\tau-31}{5\tau}}\|\nabla h\|_{L^2}^2,
 \eeno
where $C(K)>0$ only depend on $K$ (and $\tau$), and $C>0$ only depends on $\tau$.
\end{prop}
\begin{proof}
 We recall that 
$$ \mathcal{II} \le  \int f \,\pa_i^2h\, h\, dv =  \int f_A\,\pa_i^2h\, h\, dv 
+  \int f^A \,\pa_i^2h\, h\, dv . $$
Then
\ben\label{aut3} 
\mathcal{II}&\le& C\, A\|\na^2 h\|_{L^2_{-\f32}}\|h\|_{L^2_{\f32}}
+ C\, \|f^A\|_{L^3_{\f32}} \, \|\na^{2}h\|_{L^2_{-\f32}} \,  \|h\|_{L^6} \notag  \\
&\le& \frac{\eta}4\,  \|\na^2 h\|_{L^2_{-\f32}}^2
+ C\, \eta^{-1}(A^2\|h\|_{L^2_{\f32}}^2+\|f\|_{L^3_{-3}}\|f^A\|_{L^3_6}\|\na h\|_{L^2}^2).
 \een 
The desired result is obtained by using  an estimate almost identical to that of (\ref{nez}). 
\end{proof}

\subsubsection{Summary of the estimate for the remainder terms}

  We regroup the results of Propositions \ref{EstI0}, \ref{EstI1} and \ref{EstII} in the 
\begin{prop}\label{Esum}  Let $f \ge 0$  (and $h= f - \mu$) such that  $ \|f\|_{L^1_{2}(\R^3)} = 4$ and  $ \|f\|_{L^1_{\tau}(\R^3)} +   \|f\|_{L\log L} \le K$ with $K>0$, $\tau > 31$.
Then for all $\eta \in ]0,1[$ and $A>e$, the following estimate holds: 
 $$
 I_{1,2} + I_2 + I_3 + I_4 \le  \eta\, \|\nabla^2 h\|_{L^2_{-\frac32}}^2 
+ (C\,\eta + C(K)\,\eta^{-1}(\log A)^{-\f{\tau-31}{5\tau}})\, \big(D(f)+1\big)\|\nabla h\|_{L^2}^{\f{14}5}$$
$$+\, C\, \eta^{-1}A^2\|h\|_{L^2_{2}}^2 +\, C(K)\, ( \eta^{-1} + \eta^{-3}\, A^2)\,  \|\nabla h\|_{L^2_6}^2, $$
where $C(K)>0$ only depend on $K$ and $\tau$, and $C>0$ only depends on $\tau$.
\end{prop}

\begin{proof}
This estimate is directly obtained from  Propositions \ref{EstI0}, \ref{EstI1} and \ref{EstII}, remembering that $\eta <1$, $\log A >1$ and $-\f{\tau-31}{5\tau} <0$.
\end{proof}

\subsubsection{Summary of the estimate for all terms}

 We now regroup the results of Proposition \ref{Esum} and Corollary  \ref{coercivity3}. From now on, we typically denote by $C^*$ constants which can be replaced by 
a larger constant, and by $C$ constants which can be replaced by a smaller (strictly positive) constant. We get the
% as follows:
\begin{prop}\label{Esum2}  Let $f \ge 0$  (and $h= f - \mu$) such that $\int_{\R^3} f(v) \, dv =1$,   $\int_{\R^3} f(v)\, |v|^2 \, dv =3$, and  $ \|f\|_{L^1_{\tau}(\R^3)} +   \|f\|_{L\log L} \le K$ with $\tau > 31$.
Then for all $\eta \in ]0,1[$, the following estimate holds: 
 \begin{equation} \label{qqq}
 I_1 + I_2 + I_3 + I_4 \le  -\f{C(K)}2 \,\|\na^2 h\|_{L^2_{-\frac32}}^2  -\,\frac{C(K)}{2} \, \|h\|_{L^1_{15/4}}^{-\f45} \, \|\na h\|_{L^2}^{\f{14}5} 
\end{equation}
$$  +\,  C^*(K)\,  \eta^{-13}\,\exp( 7 \,\eta^{- \frac{10 \tau}{\tau - 31}})\, \|h\|_{L^1_{99/4}}^{2}  + C^*(K)\,\eta\, (1 + D(f))\, \|\nabla h\|_{L^2}^{\f{14}5} , $$
where $C(K), C^*(K)>0$ only depend on $K$ and $\tau$.
\end{prop}

\begin{proof}
Using  Proposition \ref{Esum} and Corollary  \ref{coercivity3}, we see that
 \begin{equation}\label{nineg}
  I_1 + I_2 + I_3 + I_4 \le    - \,C(K) \,\|\na^2 h\|_{L^2_{-\frac32}}^2  - C(K) \, \|h\|_{L^1_{15/4}}^{-\f45} \, \|\na h\|_{L^2}^{\f{14}5} + C^*(K) \,\|h\|_{L^1}^2
\end{equation}
$$ + \, \eta\, \|\nabla^2 h\|_{L^2_{-\frac32}}^2 
+ (C^*\,\eta + C^*(K)\,\eta^{-1}(\log A)^{-\f{\tau-31}{5\tau}})\, \big(D(f)+1\big)\|\nabla h\|_{L^2}^{\f{14}5}$$
$$+\, C^*\, \eta^{-1}A^2\|h\|_{L^2_{2}}^2 +\, C^*(K)\, ( \eta^{-1} + \eta^{-3}\, A^2)\,  \|\nabla h\|_{L^2_6}^2. $$
Using Proposition \ref{interpH1a} for $m=6$, we see that
$$ \|\nabla h\|_{L^2_6} +  \| h\|_{L^2_6} \le C^*\, \|h\|_{L^1_{99/4}}^{2/7} \,  \|h\|_{L^1}^{5/7} +  C^*\, \|h\|_{L^1_{99/4}}^{2/7} \, \|\na^2 h\|_{L^2_{-\frac32}}^{5/7} , $$
so that thanks to Young's inequality, for any $\zeta>0$,
 $$ \|\nabla h\|_{L^2_6}^2 +  \| h\|_{L^2_2}^{2} \le C^*\, (1 + \zeta^{-7/2})\, \|h\|_{L^1_{99/4}}^{2} +  C^*\, {\zeta}^{7/5} \, \|\na^2 h\|_{L^2_{-\frac32}}^{2} . $$
 Taking $\zeta := C^* \,\eta^{10/7}\, A^{-10/7}$, we see that
\begin{equation}\label{ninegter}
 C^* \, \eta^{-1}\,A^2\,\| h\|_{L^2_2}^{2} \le C^*\, \eta^{-1}\, A^2 \,(1+ \eta^{-5}\,A^5) )\, \|h\|_{L^1_{99/4}}^{2} + \frac{\eta}2\, \|\na^2 h\|_{L^2_{-\frac32}}^2,
\end{equation}
while taking $\zeta := C^*(K)\, \eta^{20/7}\, A^{-10/7}$ (and observing that $\eta^{-1} \le \eta^{-3}\,A^2$), we see that
\begin{equation}\label{ninegbis}
 C^*(K) \,(\eta^{-1} + \eta^{-3}\,A^2)\, \|\nabla h\|_{L^2_6}^2  \le C^*(K)\, \eta^{-3}\, A^2 \,(1+ \eta^{-10}\,A^5) \, \|h\|_{L^1_{99/4}}^{2} + \frac{\eta}2\, \|\na^2 h\|_{L^2_{-\frac32}}^2.
\end{equation}
Using this bound in estimate  (\ref{nineg}), we see that
\begin{equation}\label{sum1}
  I_1 + I_2 + I_3 + I_4 \le   (2\eta  - C(K)) \,\|\na^2 h\|_{L^2_{-\frac32}}^2   - C(K) \, \|h\|_{L^1_{15/4}}^{-\f45} \, \|\na h\|_{L^2}^{\f{14}5} 
\end{equation}
$$  +  \bigg[ C^*(K) + C^*\, \eta^{-1}\, A^2 \,(1+ \eta^{-5}\,A^5)  + C^*(K)\, \eta^{-3}\, A^2 \,(1+ \eta^{-10}\,A^5)  \bigg]\, \|h\|_{L^1_{99/4}}^{2} $$
$$ + (C^*\,\eta + C^*(K)\,\eta^{-1}(\log A)^{-\f{\tau-31}{5\tau}})\, \big(D(f)+1\big)\|\nabla h\|_{L^2}^{\f{14}5}, $$
so that when $\eta < C(K)/4$, 
\begin{equation}\label{sum2}
I_1 + I_2 + I_3 + I_4 \le  -\f{C(K)}2 \,\|\na^2 h\|_{L^2_{-\frac32}}^2  - C(K) \, \|h\|_{L^1_{15/4}}^{-\f45} \, \|\na h\|_{L^2}^{\f{14}5}    + C^*(K)\,  \eta^{-13}\,A^7 \, \|h\|_{L^1_{99/4}}^{2}
\end{equation}
$$ + (C^*\,\eta + C^*(K)\,\eta^{-1}(\log A)^{-\f{\tau-31}{5\tau}})\, \big(D(f)+1\big)\|\nabla h\|_{L^2}^{\f{14}5} . $$
 We now select $A>e$ such that $(\log A)^{-\f{\tau-31}{10\tau}} = \eta$ , and get (changing the names of the constants) estimate \eqref{qqq}.
%$$   I_1 + I_2 + I_3 + I_4 \le   -\f{C(K)}2 \,\|\na^2 h\|_{L^2_{-\frac32}}^2 -\, C(K) \, \|h\|_{L^1_{15/4}}^{-\f45} \, \|\na h\|_{L^2}^{\f{14}5} $$
%$$  +\,  C^*(K)\,  \eta^{-13}\,\exp( 7 \,\eta^{- \frac{10 \tau}{\tau - 31}})\, \|h\|_{L^1_{99/4}}^{2}  + C^*(K)\,\eta\, \big(D(f)+1\big)\|\nabla h\|_{L^2}^{\f{14}5} . $$
%Remembering that $\|h\|_{L^1_{15/4}} \ge \|h\|_{L^1} \ge 1$, when $\eta>0$ is small enough (depending on $K$), we get the statement of the Proposition.
\end{proof}

\subsection{Application of the estimates to the solutions of Landau equation} \label{sub24}

\begin{lem}\label{ml} Let $\ell > 19/2$, $f_0\in L^1_\ell\cap L\log L$ be a nonnegative function such that $\int_{\R^3} f_0(v)\, dv~=1$,  $\int_{\R^3} f_0(v)\, |v|^2\, dv =3$. We consider 
 $f := f(t,v) $ a  weak (well constructed) nonnegative solution to the Landau equation with Coulomb potential  \eqref{landau} -- (\ref{13d}),
and $h=f-\mu$. 
\par 
Then for all $\theta \in [0, \ell]$, $q < q_{\ell, \theta}$ with
\begin{equation}\label{qltheta}
q_{\ell, \theta} :=  - \frac{2\,\ell^2 - 25\,\ell + 57}{18\,(l-2)}\, \left(1 - \frac{\theta}{\ell}\right)  + \frac{\theta}{\ell}, 
\end{equation}
 there exists $C>0$ (depending on $\theta$, $\ell$ and $K$ such that $\|f_0\|_{L^1_l} + \|f_0\|_{L\,\log L} \le K$) 
such that 
\begin{equation}\label{eqaz} 
\forall t \ge 0, \qquad \qquad \|h(t,\cdot) \|_{L^1_{\theta}}  \le C\, (1+t)^q.
\end{equation}

More specifically,
% if $\ell \ge 42$, then for some $\theta > 31$, $r>0$, 
%\begin{equation}\label{eqaz2} 
%\forall t \ge 0, \qquad \qquad ||h(t,\cdot) ||_{L^1_{\theta}(\R^3)}  \le C\, (1+t)^{-r}, 
%\end{equation}
if $\ell \ge 55$, then for some $r_1>7/4$, $r_2 >0$,
\begin{equation}\label{eqaz2} 
\forall \;t \ge 0,  \quad ||h(t,\cdot) ||_{L^1_{99/4}}  \le C\, (1+t)^{-r_1}, 
 \quad ||h(t,\cdot) ||_{L^1_{45}}  \le C\, (1+t)^{-r_2}.
\end{equation}
\end{lem}

\begin{proof}
We first recall that thanks to Theorem \ref{thm:decay}, for $\beta < \frac{2 \ell^2 - 25\ell + 57}{9(\ell - 2)}$, the relative entropy decays according to the inequality
$$ \forall t >0, \qquad H(t) \le C_{\beta} \, (1+t)^{- \beta}, $$
where $C_{\beta}>0$ only depends on $\ell$ and $K$ such that $||f_0||_{L^1_{\ell}(\R^3)} +  \|f_0\|_{L\,\log L}  \le K$.\par 
Using Cziszar-Kullback-Pinsker inequality (cf.~\cite{Csiszar,CKP2}),
we see that 
$$ \forall t >0, \qquad \|h(t,\cdot)\|_{L^1}  \le C_{\beta} \, (1+t)^{- \beta/2}. $$
\par 
Then, (using Theorem \ref{thm:decay} again) for all $\ell >2$, there exists $C_{\ell} >0$ (which only
 depends on $\ell$ and $K$ such that $||f_0||_{L^1_{2}(\R^3)} +  \|f_0\|_{L\,\log L}  \le K$ ), %\textcolor{red}{to be checked} 
such that 
$$ \forall t >0, \qquad \|h(t, \cdot)\|_{L^1_{\ell}}  \le C_{\ell} \, (1+t). $$
\par 
Finally, we interpolate between the two previous inequalities, for $\theta \in [0,\ell]$:
% and using interpolation, we have for $\theta\in[0,\ell)$, there exists a function 
%$\lambda_\ell(\theta):=\f{(\ell-\theta)(\ell-55/2)}{9(\ell-2)}-\f{\theta}{\ell}$ such that
\ben\label{hdecay} \|h\|_{L^1_\theta}\le \|h\|_{L^1}^{1-\theta/\ell} \, \|h\|_{L^1_\ell}^{\theta/\ell}
\le C\, (1+t)^{q}, \een 
for $q < q_{\theta, \ell}$, and $C>0$ as described in the Lemma.
\medskip

The special case (when $\theta = 99/4$, or $\theta=45$) is directly obtained thanks to this estimate.
\end{proof} 

 We now write the $H^1$ estimate that will yield the differential inequality (\ref{MonoFom}).

\begin{prop}\label{esqml}  
Let $f_0\in L^1_{55}(\R^3)\cap L\log L(\R^3)$ be a nonnegative function such that $\int_{\R^3} f_0(v)\, dv =1$ and  $\int_{\R^3} f_0(v)\, |v|^2\, dv =3$. We consider 
 $f := f(t,v) $ a  nonnegative smooth and quickly decaying when $|v| \to \infty$ ($C^2_t(\mathcal{S})$) solution (on an interval of time $[0,T]$) to the Landau equation  \eqref{landau} -- \eqref{13d},
% given by Theorem \ref{thm:decay},
and $h=f-\mu$. 
\par Then for some $k_1>0$, $k_2 > 7/2$, $C_1,C_2,C_3>0$ (depending only on $K$ such that $\|f_0\|_{L^1_{55}} +  
 \|f_0\|_{L\,\log L} \le K$),
 the following differential inequality holds (on $[0,T]$) for all $\eta \in ]0,1[$ sufficiently small (depending on $K$):
\ben
\label{EE1} 
&&  \f{d}{dt}  \|\nabla h\|_{L^2}^2+ C_1\,(1+t)^{k_1} \, \|\na h\|_{L^2}^{\f{14}5}\notag\\&&
\le   \eta\, C_3\, D(f)\, \|\nabla h\|_{L^2}^{\f{14}5}+C_2\, \eta^{-13}\exp\{7\eta^{-\f{450}{14}}\} \,(1+t)^{-k_2} .
%\\ \label{Elocal} 
%&&  \f{d}{dt}  \|\nabla h\|_{L^2}^2+\f{C(K)}2 \,\|\na^2 h\|_{L^2_{-\frac32}}^2+\f14I_{1,1}+\f14C_1\,(1+t)^{k_1} \, \|\na h\|_{L^2}^{\f{14}5}\notag\\&&
%\le    \|\nabla h\|_{L^2}^{\f{18}5}+C_4(1+t)^{-k_2}.
\een
\end{prop}

\begin{proof} 
We consider a  smooth and quickly decaying when $|v| \to \infty$ solution $f:= f(t,v)\ge 0$ to \eqref{landau} -- (\ref{13d}) (on a given interval of time $[0,T]$). According to Lemma \ref{ml} (more precisely to
 the special case described in this Lemma),
 this solution  is bounded in
$L^1_{45}(\R^3)$ (with a bound controlled by $K$ such that $\|f_0\|_{L^1_{55}(\R^3)} +  ||f_0||_{L\,\log L}  \le K$). 
\par
%We first prove \eqref{E1}. 
Recalling the computation (\ref{i14}), which rigorously holds, we can use Proposition \ref{Esum2} with $\tau=45$ 
(for a smooth solution to eq. \eqref{landau} -- \eqref{13d}),  and we see that (for some $C, C_4, C_5>0$ 
depending only on $K$ such that $\|f_0\|_{L^1_{55}} + \|f_0\|_{L\,\log L} \le K$),
\begin{equation}\label{nineg2}
  \f{d}{dt}  \|\nabla h\|_{L^2}^2   + C \, \|h\|_{L^1_{15/4}}^{-\f45} \, \|\na h\|_{L^2}^{\f{14}5}  
\end{equation}
$$  \le C_4\,  \eta^{-13}\,\exp( 7 \,\eta^{- \frac{450}{14}})\, \|h\|_{L^1_{99/4}}^{2}  + C_5\,\eta\,(1 + D(f))\, \|\nabla h\|_{L^2}^{\f{14}5} . $$
Using again the special case described at the end of Lemma \ref{ml} (and observing that $ \|h\|_{L^1_{15/4}} \le  \|h\|_{L^1_{99/4}}$),
 we complete the proof of the differential inequality \eqref{EE1}.

\end{proof}

% We are in a position to get the conclusion of $\dot{H}^1$ estimate for the equation. We have the following lemma.

%\begin{lem}\label{H1energy} Recall $\lambda_\ell(\theta)=\f{(\ell-\theta)(\ell-55/2)}{9(\ell-2)}-\f{\theta}{\ell}$. Choose $\ell>31$ and 
 %$\tau\in [31,\ell]$ such that $\lambda_\ell(\tau)\ge0$. Suppose that $f$ is a solution to \eqref{landau} satisfying the normalization \eqref{f0}, $f_0\in \dot{H}^1\cap L^1_\ell\cap L\log L$ and $h=f-\mu$. Then there exist constants $C_1,C_2$ and $k_1=\f45\lambda_\ell(\f{15}4),k_2=\lambda_\ell(25)$ such that for any $\eta>0$,
%\ben\label{E1} 
%&&  \f{d}{dt}  \|\nabla h\|_{L^2}^2+ C_1\,(1+t)^{k_1} \, \|\na h\|_{L^2}^{\f{14}5}\notag\\&&
%\le   \eta\, C(K)\, D(f)\, \|\nabla h\|_{L^2}^{\f{14}5}+C_2\, \eta^{-13}\exp\{7\eta^{-\f{10\tau}{\tau-31}}\} \,(1+t)^{-k_2}.
%\een
%\end{lem}
%\begin{proof}
 
Thanks to Proposition \ref{esqml}, we now can reduce the main results in Theorem \ref{maintheorem1} to the analysis of some ordinary differential inequality.

\subsection{Analysis of a differential inequality} \label{sub25}

 We start with the following Lemma, which corresponds to the special case $\eta = C_3^{-1}$, $B^* := \mathcal{B}(\eta)$, $\mathcal{B}(x) := C_2\, x^{-13}\,\exp\{7\,x^{-\f{450}{14}}\}$ in Proposition~\ref{esqml}.

\begin{lem}\label{edo1} 
Let $X, H$ be  $C^1$ functions from $[0,T]$ to $\R_+$ (for $T \in ]0, +\infty]$), $C_1,B^*,k_1>0$, $k_2> 7/2$, and $D := - H'$ such that
% for $t \ge 0$,
  \begin{equation}\label{ODE}
\forall t \in [0,T],
\qquad 
\frac{d}{dt}X(t)^2+C_1(1+t)^{k_1}\, X(t)^{\frac{14}{5}}\leq   D(t)\, X(t)^{\frac{14}{5}} +B^*\, (1+t)^{-k_2}.
\end{equation}
\par 
Then for $k:= \min( \frac{2 k_2 - 7}5, k_1)$ and some constant $C_6>0$ depending only on $C_1, B^*, k_2$, the following differential inequality holds:
\begin{equation}\label{InqM1}
 \forall t \in [0,T], \qquad \frac{d}{dt}\bigg( H(t)-\f52\, [X(t)^2 + B^*\, (1+t)^{1-k_2}]^{-2/5} \bigg)+C_6\, (1+t)^{k} \le 0.  
\end{equation}
\end{lem}

\begin{proof}

 %We will focus on the ordinary differential inequality:
  %\begin{equation}\label{ODE}
%\frac{d}{dt}X^2+C_1(1+t)^{k_1} X^{\frac{14}{5}}\leq   D(t) X^{\frac{14}{5}} +B(1+t)^{-k_2},
%\end{equation}
%where $\f{d}{dt}H(t)=-D(t)$, $k_{1}=\f45\lambda_\ell(\f{15}4),\;k_{2}=\lambda_\ell(25)$. It corresponds to \eqref{E1} with $\eta=1$. 
 We first observe that denoting $Y(t):=B^*\,(1+t)^{-k_{2}+1}$
%there exist constants $k_0\le\f25k_{2}-\f{7}5$ and 
and  $c_1 := (B^*)^{-\f25}\,(k_{2}-2)$, the following differential inequality holds:
% satisfying that 
\beno \forall t \in [0,T],
\qquad  \f{d}{dt} Y(t)+c_1\, (1+t)^{\frac{2 k_2 - 7}5}\, Y^{\f{7}5}(t)\le -B^*\,(1+t)^{-k_2}. \eeno
Therefore for some $C_6>0$ depending only on $C_1, B^*, k_2$, the following differential inequality also holds:
% if $\Phi^2(t):=X^2+B(1+t)^{-k_2+1}$, then there exist two constants $k=\min\{k_0, k_1\}, C_2=C(\min\{c_1,C_1\})$ such that
$$ \forall t \in [0,T], \qquad   \frac{d}{dt} \bigg[X(t)^2+B^*\,(1+t)^{1-k_2}\bigg]  + C_6\, (1+t)^{k}\, \bigg[X(t)^2+B^*\,(1+t)^{1-k_2}\bigg]^{7/5}  $$
$$ \leq  D(t) \bigg[X(t)^2+B^*\,(1+t)^{1-k_2}\bigg]^{7/5}. $$
The differential inequality stated in the Lemma is then obtained by dividing this differential inequality by $\bigg[X(t)^2+B^*\,(1+t)^{1-k_2}\bigg]^{7/5}$.
\end{proof}
 
Next we turn to the following consequence of Lemma \ref{edo1}:

\begin{lem}\label{LemODI1} 
Let $X, H$ be  $C^1$ functions from $[0,T]$ to $\R_+$ (for $T \in ]0, +\infty]$), $C_1,B^*,k_1>0$, $k_2> 7/2$, and $D := - H'$, such that the differential inequality
(\ref{ODE}) holds.
%Suppose that \eqref{ODE} holds with $k_{2}>\f72$. If $\mathcal{M}(t):=H(t)-\f52\Phi^{-\f45}(t)$ with $\Phi^2(t):=X^2(t)+B(1+t)^{-k_{2}+1}$, then there exist constants $k\le\min\{k_{1},\f25k_{2}-\f{7}5\}$ and $C_2=C(C_1,  
%B^{-\f25}(k_{2}-2))$ such that 
%\ben\label{InqM} \frac{d}{dt}\mathcal{M}(t)+ C_2(1+t)^{k}\le0.\een
%Moreover, 
\begin{itemize}
\item If 
$ H(0)\, [X(0)^2 + B^*]^{2/5} \le \f52$, 
then for  some constant $C_6>0$ depending only on $C_1, B^*, k_2$,
%$\mathcal{M}(0)\le 0$, then
% $X(t)$ is globally bounded, indeed,
\beno \forall t \in [0,T], \qquad  X(t)\le (\f{2}{5})^{-\f54}\bigg(H(t)+\f{C_6}{k+1}\, \bigg[(1+t)^{1+k}-1 \bigg]\bigg)^{-\f54}; \eeno

\item If $ H(0)\, [X(0)^2 + B^*]^{2/5} > \f52$,  then for
$$T^* := \bigg(\f{1+k}{C_6} \bigg[  H(0)-\f52\, [X(0)^2 + B^*]^{-2/5} \bigg] +1\bigg)^{\f1{k+1}}-1 , $$
 one has (for $T > T^*$) $H(T^*) \le \f52\, [X(T^*)^2 + B^*\, (1+T^*)^{1-k_2}]^{-2/5}$ and for $t\in [0,T - T^*]$,
 \[
X(T^* +t) \le  (\f{2}{5})^{-\f54}\bigg(H(T^*)+\f{C_6}{k+1}\, \bigg[(1+T^*+t)^{1+k}-(1+T^*)^{1+k} \bigg]\bigg)^{-\f54}.\] 
\end{itemize}
%\textcolor{red}{I am not sure that it is sufficient to ensure existence after $T$ because 
%$ \f52\, [X(T)^2 + B\, (1+T)^{1-k_2}]^{-2/5} >  \f52\, [X(T)^2 + B]^{-2/5}$.}

\end{lem}

\begin{proof} 
%We only need to provide the proof for the analysis of \eqref{InqM}. 
By integrating both sides of inequality \eqref{InqM1} on the interval $[t_{1},t_{2}],\; T > t_{2}> t_{1}\geq 0$, we see that
$$   H(t_2)-\f52\, [X(t_2)^2 + B^*\, (1+t_2)^{1-k_2}]^{-2/5} +\f{C_6}{1+k}\, \bigg[(1+t_2)^{1+k}-(1+t_1)^{1+k} \bigg] $$
\ben\label{Mt} 
\le  H(t_1)-\f52\, [X(t_1)^2 + B^*\, (1+t_1)^{1-k_2}]^{-2/5} . 
\een  
%By the definition of $\mathcal{M}(t)$, it implies 
%\beno \f52\Phi^{-\f45}(t)\ge H(t)-H(0)+\f52 \Phi^{-\f45}(0)+\f{C_2}{1+k}((1+t)^{1+k}-1).\eeno  
Taking $t_1=0,\;t_2=t$, and using the condition
% $\mathcal{M}(0)\le 0$ which is equivalent to
 $\f52\, [X(0)^2 + B^*]^{-2/5} \ge H(0) $, we rewrite the above inequality as  
\beno \f52\, [X(t)^2 + B^*\, (1+t)^{1-k_2}]^{-2/5} \ge H(t) +\f{C_6}{1+k}\, \bigg[(1+t)^{1+k}-1 \bigg].\eeno
From this, we get
\beno 
X(t)\le \bigg[ (\f{2}{5})^{-\f52}\bigg(H(t)+\f{C_6}{k+1}\, \bigg[(1+t)^{1+k}-1 \bigg]\bigg)^{-\f52}-B^* \,(1+t)^{1-k_2}\bigg]^{\f12}, 
\eeno
which proves the first result.
\medskip

The second result follows from estimate \eqref{Mt} by taking $t_1=0$ and $t_{2}=T^*$, and solving 
\[
\f{C_6}{1+k}\, \bigg[(1+T^*)^{1+k}-1 \bigg]=H(0)-\f52\, [X(0)^2 + B^*]^{-2/5} .
\]
Now let $t_1=T^*$ and $t_{2}=T^*+t,\;t>0$, then $H(T^*) - \f52\, [X(T^*)^2 + B^*\, (1+T^*)^{1-k_2}]^{-2/5}\leq 0$
 implies that 
\[
X(T^*+t) \le  (\f{2}{5})^{-\f54}\bigg(H(T^*)+\f{C_6}{k+1}\, \bigg[(1+T^*+t)^{1+k}-(1+T^*)^{1+k} \bigg]\bigg)^{-\f54},
\]
which gives the estimate for $X$ after the time $T^*$ described in the Lemma.
\end{proof}

\subsection{End of the proof of 
%Theorem \ref{localwell} and
 Theorem \ref{maintheorem1}} \label{sub26}
  We now are in a position to prove  %Theorem \ref{localwell} and
  Theorem \ref{maintheorem1}. 
% To do that, we first 
% verify two things. The first one is
We show  that  the {\it a priori} estimates obtained in Subsections 2.1 to 2.5 can be used to build a solution to eq. (\ref{landau})  -- (\ref{13d}), thanks to their application to the smooth solutions of an approximated equation. 
\medskip

%for the original equation are valid through the use of approximated 
% solutions. The second one is to prove the continuity of the relative entropy $H(t)$ which is crucial for the estimate of the lower bounds 
 %of the blow-up rate for the potential blow-up phenomenon.

%\subsubsection{Use of approximated solution}
% We first remark that 
 %the local existence, nonnegativity  and  uniqueness of solutions of Landau equation can be established  through the grazing collisions limit (see \cite{heyang}) and the Wasserstein distance(see \cite{Fournier}), thus in what follows we only need to provide the {\it a priori} estimates for the equation.  However to prove Theorem \ref{maintheorem1} that the solution can become smooth again after a period, we require that the approximated solutions are global well-posed. To do that, we
We introduce therefore the unique solution $f^\epsilon : = f^\epsilon(t,v) \ge 0$ to the approximated equation
 \ben\label{approeq}
 \pa_t f^\epsilon=Q^\epsilon(f^\epsilon, f^
 \epsilon),
 \een 
where $Q^\epsilon$ is defined by 
\beno Q^\epsilon(g,h)=\nabla_v \cdot {\Big (}[ a^\epsilon*g ]\;\nabla_v h- [ a^\epsilon*\nabla g ] \; h{\Big )},
\eeno
with 
\ben\label{ae}
a^\epsilon(z)=(|z|^2+\epsilon^2)^{-\f12} \, \left(Id -\frac{z\otimes z}{|z|^2} \right).
\een 
 We are therefore still considering a Landau equation, but with a regularized cross section. We also introduce  smooth and quickly decaying (when $|v| \to \infty$)  initial data, converging when $\epsilon \to 0$ towards
the original initial data $f_0$. The problem (\ref{approeq}) -- (\ref{ae}) 
%As a consequence, it
 satisfies the same conservation properties (propagation of nonnegativity, conservation of mass, momentum and kinetic energy, decay of the entropy) as the original equation
 (\ref{landau})  -- (\ref{13d}).
 \medskip

%\begin{rmk} It is easy to check that formally when $\epsilon$ tends to zero, the equation \eqref{approeq} will converge to equation \eqref{landau}. In fact, $(|z|^2+\epsilon^2)^{-\f12}$ appeared in \eqref{ae} is the standard mollifier of $|z|^{-1}$ appeared in \eqref{13d}.  Moreover, the equation \eqref{approeq} formally conserves the mass, momentum and the energy. Similar to  \eqref{entropyeq} and \eqref{DefHt}, the entropy formula  also holds for \eqref{approeq}. Without lose of generality, we can define the relative entropy $H^\epsilon$ for \eqref{approeq} by replacing $f$ by $f^\epsilon$ in \eqref{DefHt}, that is,
%\ben\label{DefHtep}
%H^\epsilon(t):=H(f|\mu_{\rho,u,T})(t) := \int \bigg(f^\epsilon(v) \log \f{f^\epsilon(v)}{\mu_{\rho,u,T}}-f^\epsilon+\mu_{\rho,u,T}\bigg) \, dv,\een 
%\end{rmk}

Next we briefly explain how to prove the

\begin{prop}\label{thmappro}  For $\epsilon>0$,  estimates (\ref{MonoFom}), (\ref{smoesti}) and (\ref{nnn}) -- (\ref{noblowupht})  
%statement  $(i)$ and statement $(ii)$ in 
%Theorem \ref{maintheorem1} 
hold for the unique smooth ($C^2_t(\mathcal{S})$) solution of  equation \eqref{approeq} -- (\ref{ae})
(with smoothed initial data),  with  constants which do not depend on $\epsilon$.
% Moreover, for $t<\mathcal{T}:=\f54(\|\na h_0\|_{L^2}^2+C_4k_2^{-1})^{-\f45}$,
%\[\|\na h^\epsilon(t)\|_{L^2}^2\le \big[(\|\na h_0\|_{L^2}^2+C_4k_2^{-1})^{-\f45}-\f45t\big]^{-\f54},\] 
%where $h^\epsilon:=f^\epsilon-\mu$.
%the same initial data $f_0$ as that in Theorem  \ref{maintheorem1}. 
\end{prop}
\begin{proof} \noindent{\it Step 1:}  Since (for $\epsilon>0$), there is no singularity in $a^\epsilon$,  equation \eqref{approeq} -- (\ref{ae}) behaves (from the point of view of regularity) like the Landau equation with Maxwell molecules (that is, when $\gamma=0$ in \eqref{agamma}). Hence, smoothness and moments can be proved to be propagated globally for this equation. This is easily checked by following the strategy used in \cite{HE12,HE14}. Thus equation (\ref{approeq})  -- (\ref{ae}) admits a unique (global) smooth solution (the initial data being themselves smooth).
% It implies that we only need to give the {\it a priori} estimates for the equation \eqref{approeq}. We divide the proof into several steps.

\smallskip

\noindent{\it Step 2:} Using Theorem 3 in \cite{D}, we see that estimates \eqref{Dfdissipation11}  and \eqref{Dfdissipation}  hold  when $a$ is replaced by $a^\epsilon$, with a constant that does not depend on $\epsilon$. 
% As a consequence, the same holds for estimate \eqref{Dfdissipation}. 
 It is then  possible to show, using the same method as in \cite{CDH},  that the long-time behavior estimates
% (and, as previously said, the estimates for the  propagation of $L^1$ moments) for the solution to equation \eqref{approeq}  -- (\ref{ae}) 
 are  the same for the solution to equation \eqref{approeq}  -- (\ref{ae})  as those for the solution to
 Landau equation with Coulomb potential (\ref{landau})  -- (\ref{13d}). In other words,  Theorem \ref{thm:decay} holds for 
the unique smooth solutions to
 equation \eqref{approeq}  -- (\ref{ae}), with constants which do not depend on $\epsilon$. 

%\noindent$\bullet$ Using Theorem 3 in \cite{D}, \eqref{Dfdissipation11} is  valid for  the equation \eqref{approeq}. Thus \eqref{Dfdissipation} also holds for the equation \eqref{approeq}. Next by Lemma 6 and the proof of Lemma 8 in \cite{CDH}, we get the propagation of the $L^1$ moment(that is, \eqref{momentpropa}) in Theorem \ref{thm:decay}. 

%\noindent$\bullet$ To show the decay estimate of the relative entropy \eqref{decayH}, we first notice that Proposition 4 and Proposition 5 in \cite{CDH} hold for the equation \eqref{approeq} which implies Theorem 1 in  \cite{CDH} is valid   
%for the equation \eqref{approeq}. Then \eqref{decayH} can be proved by following  the proof of Theorem 2 in \cite{CDH}. We emphasize that all the estimates derived   are  independent of $\epsilon$.
\smallskip

\noindent{\it Step 3:} We show that Proposition \ref{esqml} holds for the unique smooth solutions of  equation \eqref{approeq}  -- (\ref{ae}), with constants in the estimate which do not depend on $\epsilon$.

This amounts  to showing that the estimates in the proof still hold when $a$ is replaced by $a^\epsilon$. Noticing that 
$ b^\epsilon_i(z): = \sum_{j=1}^3 \partial_j a^\epsilon_{ij}(z) = -2 \,z_i\, |z|^{-2}(|z|^2+\epsilon^2)^{-\f12}$, 
 $ \sum_{i=1}^3 \pa_ib^\epsilon_i(z) =-2 \epsilon^{-3} |\epsilon^{-1} z|^{-2}\\\times(|\epsilon^{-1} z|^2+1)^{-\f32}$,
 we see that
$|a^\epsilon|\le |a|$, $|b_i^\epsilon|\le |b_i|$,
%and
%$$\int_{\R^3} (c^\epsilon*F) GHdv=-2\int_{\R^6} |v_*|^{-2}(v_*|^2+1)^{-\f32}F(v-\epsilon v_*)GHdv_*dv,$$ 
and those inequalities can be used to show that the estimates from above in Subsections \ref{sub23}, \ref{sub24} can be reproduced with the same constants for the approximated problem as for the original problem. 
 
%\een
 %which is slightly different from \eqref{Defbc}.
\par
%\underline{$\bullet$ Estimate related to $a^\epsilon$:} Since $|a^\epsilon|\le |a|$, we only
We then can directly check by inspecting the proofs that the coercivity estimate appearing in 
Proposition \ref{coercivity} and  Corollary \ref{coercivity2} can be reproduced when $a$ is replaced by $a^\epsilon$, with constants
that do not depend on $\epsilon$.
%Thanks to the proof of  Proposition \ref{coercivity}, it is easy to see that Corollary \ref{coercivity2} keeps well for \eqref{approeq}. 

%\underline{$\bullet$  Estimate related to $b^\epsilon_i$:}  Thanks to the fact $|b_i^\epsilon|\le |b_i|$, the estimates related to $b^\epsilon_i$  are as the same as those for $b_i$.

 %\underline{$\bullet$ Estimate related to $c^\epsilon$:}     We first point out that to prove Proposition \ref{esqml} two kinds of formulation are used for $c$(defined in \eqref{Defbc}) and can be concluded as follows:
%\beno && \int_{\R^3} (c*F) G^2dv\le0,\quad \mbox{for}\,\, F\ge0; \\&& \int_{\R^3} (c*F) GHdv=-8\pi \int_{\R^3} FGHdv. \eeno
%It is not difficult to see that the above formulations  hold in large for $c^\epsilon$, that is,
%\beno &&\int_{\R^3} (c^\epsilon*F) G^2dv\le0\quad \mbox{for}\,\, F\ge0; \\&&\int_{\R^3} (c^\epsilon*F) GHdv=-2\int_{\R^6} |v_*|^{-2}(v_*|^2+1)^{-\f32}F(v-\epsilon v_*)GHdv_*dv. \eeno
%From these, in particular, one may check that estimate of $\mathcal{II}$ in Proposition \ref{EstII} is still valid for the equation \eqref{approeq}. 
 \par
Since for $\epsilon>0$, the solution $f^\epsilon$ is smooth and quickly decaying when $|v| \to \infty$, the assumptions of Proposition \ref{esqml}  are fulfilled, so that 
%Patching together all the consideration, we can get the desired result.
estimate (\ref{EE1}) holds (for this solution),
%s of  equation \eqref{approeq}  -- (\ref{ae}),
 with constants in the estimate which do not depend on $\epsilon$. 
 \smallskip

\noindent{\it Step 4:} We now can apply Lemmas \ref{edo1} and \ref{LemODI1} to $X = \|\nabla h^\epsilon\|_{L^2}$, and obtain the estimates of Theorem \ref{maintheorem1}, 
for the unique smooth  solution of  equation \eqref{approeq}  -- (\ref{ae}), with constants in the estimate which do not depend on $\epsilon$.

%\noindent{\it Step 4:} To show the last assertion in the proposition, we recall that \eqref{Elocal} can be transformed to

%which yields the result.
%
%The statement of the Theorem 
 %    Now we are in a position to prove the theorem. Since the solution to the equation \eqref{approeq} is globally smooth, Proposition \ref{esqml}, Lemma \ref{edo1} and Lemma \ref{LemODI1} will yield the statement $(i)$ and  statement $(ii)$ in the theorem.  
\end{proof}

Finally we give the end of the proof of Theorem \ref{maintheorem1}.

\begin{proof}[End of the proof of Theorem \ref{maintheorem1}:] 
Note first that part (i) of Theorem \ref{maintheorem1} is immediately obtained (without using the approximation problem) by the use of 
Proposition \ref{esqml} and Lemma \ref{edo1}. 
\medskip

We now turn to parts (ii) and (iii).
As in \cite{D}, we consider $f^\epsilon$ the unique smooth solution of eq. (\ref{approeq}) -- (\ref{ae}) with initial data strongly converging to $f^0$.
  It is then possible to pass to the limit (in a weighted weak $L^1$ space, and up
to extracting a subsequence) when $\var \to 0$ in $f^\epsilon$, and get in this way a (well constructed) weak solution $f$
to the original 
equation  eq. (\ref{landau})  -- (\ref{13d}) with initial data $f^0$. 
%This weak solution satisfies the estimates  of statements $(i)$ and statement $(ii)$ in Theorem \ref{maintheorem1}.
%\par
%\textcolor{red}{This last  point is not completely obvious, and deserves maybe  an explanation. Laurent}
%\par
\smallskip

Due to the convexity of $x \mapsto x\log x$ and the lower semi-continuity of the weak convergence in $\dot{H}^1$, we obtain that 
\beno 
H(t)\le  \liminf\limits_{\epsilon\rightarrow0} H^\epsilon(t),\quad \|\na h\|_{L^2}\le\liminf\limits_{\epsilon\rightarrow0}\|\na h^\epsilon\|_{L^2},
\eeno
where $H^\epsilon(t)$ is the relative entropy of $f^\epsilon$ at time $t$. 
\medskip

Thanks to these properties, we can pass to the limit in the following estimates:
\begin{itemize}
\item
For the initial data under the threshold,
 \[ \forall t \ge 0,  \qquad 
\|h^\epsilon(t)\|_{\dot{H}^1} \le  (\f{2}{5})^{-\f54}\bigg( H^\epsilon(t) + \f{C_6}{k+1}\, \bigg[(1+t)^{1+k}- 1 \bigg]\bigg)^{-\f54};
\]
\item 
For general suitable initial data and $t>T^*$,
\beno 
 H^\epsilon(t) \le \f52\, \bigg[\|h^\epsilon(t)\|_{\dot{H}^1}^2 + B^*\, (1+t)^{1-k_2} \bigg]^{-2/5}.
 \eeno
  \[
\|h^\epsilon(t)\|_{\dot{H}^1} \le  (\f{2}{5})^{-\f54}\bigg( \f{C_6}{k+1}\, \bigg[(1+t)^{1+k}-(1+T^*)^{1+k} \bigg]\bigg)^{-\f54}.
\]
\end{itemize}
We conclude thus the proof 
of
 statements $(ii)$ and $(iii)$ of 
Theorem \ref{maintheorem1}.
% Note that the uniqueness can be obtained just following the proof of Theorem \ref{localwell}.  

%\subsubsection{Continuity of the relative entropy $H(t)$} 

\end{proof}

\section{Local solutions: proof of Proposition \ref{localwell}}

We present in this section the 
%\medskip
%\begin{proof}[P
Proof of Proposition \ref{localwell}. We start with the following Proposition, which is a variant of Proposition \ref{Esum2}:

 \begin{prop} \label{Esum22} 
Let $f \ge 0$  (and $h= f - \mu$) such that  $ \|f\|_{L^1_{2}(\R^3)} = 4$ and  $ \|f\|_{L^1_{\tau}(\R^3)} +   \|f\|_{L\log L} \le K$ with $K>0$ and $\tau >31$.
Then
% for all $\eta \in ]0,1[$ and $A>e$,
 the following estimate holds: 
\begin{equation}\label{sumlast}  I_1 + I_2 + I_3 + I_4 \le  -\f{C(K)}4 \,\|\na^2 h\|_{L^2_{-\frac32}}^2  - \f{C(K)}2 \, \|h\|_{L^1_{15/4}}^{-\f45} \, \|\na h\|_{L^2}^{\f{14}5} 
- \frac12\, I_{1,1}  
\end{equation}
$$ +\, C^*(K) \, \|h\|_{L^1_{99/4}}^{2}  +   \|\nabla h\|_{L^2}^{\f{18}5} , $$
where $C^*(K), C(K)>0$ only depend on $K$ and $\tau$.
% and $\tau$, and $C>0$ only depends on $\tau$.
\end{prop}

\begin{proof}
Using estimates (\ref{f31e}) and (\ref{aut}), we see that
$$  \|f\|_{L^{3,1}_{-3}}^2\|\na h\|_{L^2}^2 \le C\,\bigg( \|\na h\|_{L^2}^2 + \|\na h\|_{L^2}^{18/5}  \bigg). $$
 Then, recalling estimates (\ref{f31e}) and (\ref{aut2}),  we also see that (for $A>e$)
 $$ \|f\|_{L^{3,1}_{-3}}\|f^A\|_{L^{3,1}_6}\|\na h\|_{L^2}^2\le   C(K)\,(\log A)^{-\f{\tau-31}{5\tau}} \,\bigg( \|\na h\|_{L^2}^2 +  \|\na h\|_{L^2}^{\f{18}5} \bigg). $$
% Proof of Proposition \ref{localwell}}
Using the notation (\ref{DefI}) and bounds (\ref{bi2}) and (\ref{bi1}), this leads to the bound (for all $\eta \in ]0,1[$ and  $A>e$)
 \beno \label{newI}
\mathcal{I}&\le& \frac{\eta}2\, \|\nabla^2 h\|_{L^2_{-\f32}}^2+ \bigg( C\,\eta\,
+C(K)\eta^{-1}(\log A)^{-\f{\tau-31}{5\tau}}\bigg)\,
\|\nabla h\|_{L^2}^{\f{18}5}\\
&& + \, C\, (\eta + \eta^{-1} + \eta^{-3}A^2)\, \|\nabla h\|_{L^2_6}^2 +C(K)\, \eta^{-1}(\log A)^{-\f{\tau-31}{5\tau}}\|\nabla h\|_{L^2}^2 .
 \eeno
%where $C(K)>0$ is a constant depending only on $K$, and $C>0$ is an absolute constant. 
\medskip

Using the notation (\ref{DefII})  and estimate (\ref{aut3}), we also get the estimate (for all $\eta \in ]0,1[$ and  $A>e$)
 \beno \label{newII}
\mathcal{II}&\le& \frac{\eta}4\,\|\nabla^2 h\|_{L^2_{-\f32}}^2+ C(K)\,\eta^{-1}(\log A)^{-\f{\tau-31}{5\tau}} \, \|\nabla h\|_{L^2}^{\f{18}5}
+ C\, \eta^{-1}A^2\|h\|_{L^2_{\f32}}^2\\
&&+\, C(K)\, \eta^{-1}(\log A)^{-\f{\tau-31}{5\tau}}\|\nabla h\|_{L^2}^2,
 \eeno
where $C(K)>0$ only depend on $K$ and $\tau$, and $C>0$ only depends on $\tau$.
\medskip

Recalling now estimates (\ref{i3i4}) and inequality (\ref{i12i2}) (together with notations (\ref{DefI}) and    (\ref{DefII})), and remembering that $\eta <1$, $\log A >1$ and $-\f{\tau-31}{5\tau} <0$, we end up with the estimate  
$$ I_{1,2} + I_2 + I_3 + I_4 \le  \eta\, \|\nabla^2 h\|_{L^2_{-\frac32}}^2 
+ (C\,\eta + C(K)\,\eta^{-1}(\log A)^{-\f{\tau-31}{5\tau}})\,  \|\nabla h\|_{L^2}^{\f{18}5} $$
\begin{equation}\label{n12}
+\, C\, \eta^{-1}A^2\|h\|_{L^2_{2}}^2 +\, C(K)\, ( \eta^{-1} + \eta^{-3}\, A^2)\,  \|\nabla h\|_{L^2_6}^2, 
\end{equation}
where $C(K)>0$ only depend on $K$ (and $\tau$) and $C>0$ only depends on $\tau$.

Using estimates (\ref{rrter}),  (\ref{nineg}) and (\ref{nineg2}), we see that (using $C^*$ for constants which can be replaced by larger constants, and $C$ for constants which can be replaced by smaller constants)
\begin{equation}\label{sum1b}
  I_1 + I_2 + I_3 + I_4 \le   (2\eta  - \frac12\,C(K)) \,\|\na^2 h\|_{L^2_{-\frac32}}^2   - \frac12\,C(K) \, \|h\|_{L^1_{15/4}}^{-\f45} \, \|\na h\| _{L^2}^{\f{14}5} - \frac12\, I_{1,1}
\end{equation}
$$  +  \bigg[ \frac12\, C^*(K) + C^*\, \eta^{-1}\, A^2 \,(1+ \eta^{-5}\,A^5) ) + C^*(K)\, \eta^{-3}\, A^2 \,(1+ \eta^{-10}\,A^5)  \bigg]\, \|h\|_{L^1_{99/4}}^{2} $$
$$ + (C^*\,\eta + C^*(K)\,\eta^{-1}(\log A)^{-\f{\tau-31}{5\tau}})\,   \|\nabla h\|_{L^2}^{\f{18}5} , $$
so that when $\eta < C(K)/8$, 
\begin{equation}\label{sum2b}
I_1 + I_2 + I_3 + I_4 \le  -\f{C(K)}4 \,\|\na^2 h\|_{L^2_{-\frac32}}^2  - \f{C(K)}2 \, \|h\|_{L^1_{15/4}}^{-\f45} \, \|\na h\|_{L^2}^{\f{14}5} 
- \frac12\, I_{1,1}  
\end{equation}
$$  + \,C^*(K)\,  \eta^{-13}\,A^7 \, \|h\|_{L^1_{99/4}}^{2} + (C^*\,\eta + C^*(K)\,\eta^{-1}(\log A)^{-\f{\tau-31}{5\tau}})\,     \|\nabla h\|_{L^2}^{\f{18}5} . $$
 Selecting $\eta >0$ sufficiently small, and  $A>e$ such that $(\log A)^{-\f{\tau-31}{10\tau}} = \eta$, we see that estimate (\ref{sumlast}) holds.
%\begin{equation}\label{sumlast}  I_1 + I_2 + I_3 + I_4 \le  -\f{C(K)}4 \,\|\na^2 h\|_{L^2_{-\frac32}}^2  - \f{C(K)}2 \, \|h\|_{L^1_{15/4}}^{-\f45} \, \|\na h\|_{L^2}^{\f{14}5}  - \frac12\, I_{1,1}  \end{equation}
%$$ +\, C^*(K) \, \|h\|_{L^1_{99/4}}^{2}  +   \|\nabla h\|_{L^2}^{\f{18}5} . $$
\end{proof}

\begin{proof}[End of the proof of Proposition \ref{localwell}:]  We observe that inequality  (\ref{sumlast}) still holds
 when the kernel of the Landau equation is replaced by the kernel of the approximated 
equation (\ref{approeq}) -- (\ref{ae}), with all constants not depending on $\epsilon$. 
Then, when $f^\epsilon$ (and $h^\epsilon = f^\epsilon - \mu$) is the unique smooth solution of eq. (\ref{approeq}) -- (\ref{ae}) (with regularized initial data), and proceeding as in the proof
of Proposition \ref{esqml}, we get the estimate
\begin{equation}\label{avn}
  \f{d}{dt}  \|\nabla h^\epsilon\|_{L^2}^2+ C_{12} \,\|\na^2 h^\epsilon\|_{L^2_{-\frac32}}^2
+\f14I_{1,1}^\epsilon +\f14C_{10}\,(1+t)^{k_1} \, \|\na h^\epsilon\|_{L^2}^{\f{14}5}
\le    \|\nabla h^\epsilon \|_{L^2}^{\f{18}5}+C_{11}(1+t)^{-k_2}, 
\end{equation}
where $k_1>0$ and $k_2>7/2$ are defined as in Prop. \ref{esqml}, and $C_{10}$, $C_{11}$, $C_{12}>0$ only depend on $K$ such that
$\|f_0\|_{L^1_{55}} + \|f_0\|_{L\,\log L} \le K$.
\medskip

This differential inequality implies that 
$$  \f{d}{dt}  \|\nabla h^\epsilon\|_{L^2}^2 \le    \|\nabla h^\epsilon \|_{L^2}^{\f{18}5}+C_{11}, $$
so that 
$$  \f{d}{dt}  \bigg(  \|\nabla h^\epsilon\|_{L^2}^2 + C_{11}^{5/9} \bigg)  \le  \bigg(  \|\nabla h^\epsilon \|_{L^2}^2 + C_{11}^{5/9} \bigg)^{9/5} . $$
Therefore, for $t \le \mathcal{T} := \frac54\,  ( \|\nabla h^\epsilon(0) \|_{L^2}^2 + 2C_{11}^{5/9})^{-4/5}$,
\begin{equation}\label{nsd}
  \|\nabla h^\epsilon(t) \|_{L^2}^2 \le \bigg[ ( \|\nabla h^\epsilon(0) \|_{L^2}^2 + C_{11}^{5/9})^{-4/5}  - \frac45\,t \bigg]^{-5/4} -  C_{11}^{5/9} . 
\end{equation}
%The proof to \eqref{Elocal} is similar to that for \eqref{E1}. Let us point out the main difference. We first recall that all the dissipation terms in \eqref{Elocal} results from the term $I_{1,1}$(see \eqref{rrter}). Secondly, the term $D(f)$ in the proof of \eqref{Elocal} is just used to bound $\|f\|_{L^{3,1}_{-3}}$. Thus we may use \eqref{f31e}  instead of \eqref{Dfdissipation} in Proposition \ref{EstI1} and Proposition \ref{EstII}. Keeping these in mind, we will arrive at
%\beno
%  &&\f{d}{dt}  \|\nabla h\|_{L^2}^2     +\f{C(K)}2 \,\|\na^2 h\|_{L^2_{-\frac32}}^2+\f14I_{1,1}+\f14C_1\,(1+t)^{k_1} \, \|\na h\|_{L^2}^{\f{14}5}  
%\\&&   \le\eta\, C_3\, (1+\|\na h\|_{L^2}^{\f{4}{5}})\, \|\nabla h\|_{L^2}^{\f{14}5}+C_2\, \eta^{-13}\exp\{7\eta^{-\f{450}{14}}\} \,(1+t)^{-k_2}. \eeno
%We conclude the desired result by choosing $\eta$ sufficiently small.

%We are in a position to prove Theorem \ref{localwell}.
%\beno
%\f{d}{dt}\big(\|\na h^\epsilon(t)\|_{L^2}^2+C_4k_2^{-1}(1+t)^{-k_2+1}\big)
%\le \big(\|\na h^\epsilon(t)\|_{L^2}^2+C_4k_2^{-1}(1+t)^{-k_2+1}\big)^{\f95},\eeno
Passing to the limit when $\epsilon \to 0$ as in the end of the Proof of Theorem \ref{maintheorem1}, we get the existence of a weak solution of Landau equation (\ref{landau}) -- (\ref{13d}) on the interval 
$[0, \mathcal{T}]$ which is in fact strong in the sense that it lies in $L^{\infty}([0, \mathcal{T}]; H^1(\R^3))$. Note indeed that the first time of blowup (in $H^1$ norm) is {\it{strictly}} bigger than
$\mathcal{T}$ since part of the dissipative terms  were not used in the differental inequality in order to get the bound (\ref{nsd}).
%
%Thanks to the approximated equation \eqref{approeq} and Proposition \ref{esqml}, it is easy to construct a local solution $f(t,v)\in L^\infty([0,\mathcal{T}-\eta];\dot{H}^1)$ where $\mathcal{T}=\f54(\|f_0-\mu\|_{\dot{H}^1}^2+C_4k_2^{-1})^{-\f45}$ and $\eta\ll1$. Indeed, for $t<\mathcal{T}-\eta$,
%\ben\label{H2te1}\|\na h(t)\|_{L^2}^2\le \big[(\|\na h_0\|_{L^2}^2+C_4k_2^{-1})^{-\f45}-\f45t\big]^{-\f54}\le \eta^{-\f45},\een
%where $h=f-\mu$.
\medskip 

We now focus on the regularity of the obtained solution, and the consequences concerning the issue of uniqueness.
 Using Theorem \ref{thm:decay}, we see that  on the time interval $[0,\mathcal{T}]$, one has $h \in L^\infty_t(L^1_{55})$. Then  the estimates \eqref{nsd} and \eqref{avn} imply that
 $\nabla h \in  L^\infty_t(L^2)$, and $\na^2h \in L^2_t(L^2_{-\f32})$. Thanks to a Sobolev embedding,  we see that $h \in L^\infty_t(L^6)$.
 Interpolating with the estimate stating that $ h \in L^\infty_t(L^1_{55})$, we see that  $h\in L^\infty_t(L^2_{22})$. Interpolating again this estimate with the statement 
 $\na^2h \in L^2_t(L^2_{-\f32})$, we see that $h \in L^{16/7}_t(H^{7/4}_{23/16})$. 
% for some $s>3/2$, so that
 Thanks to yet another Sobolev embedding,  we obtain that $h\in L^2_t(L^\infty)$, which is sufficient to apply the stability result in \cite{Fournier}, and get the uniqueness of the strong
 solution built above, on the concerned interval of time. 
\medskip

We finally prove that  $f\in C([0,\mathcal{T}];\dot{H}^1)$.  Using estimate (\ref{avn}), we see that 
%  From \eqref{Elocal}, we know that
 $\na^2h \in L^2([0,\mathcal{T}];L^2_{-\f32})$, and $I_{1,1}\in L^1([0,\mathcal{T}])$. Recalling identities (\ref{i14}), (\ref{i1}) and estimate (\ref{n12}), we observe
 that $ \f{d}{dt}  \|\nabla h\|_{L^2}^2\in L^1([0,\mathcal{T}])$, so that $t \mapsto \|\nabla h(t)\|_{L^2}^2$ is continuous on the interval $[0,\mathcal{T}]$.
\par 
 Remembering the weak formulation \eqref{Qweak1} and the fact $\nabla h \in  L^\infty_t(L^2)$,  it is not difficult to check that  $ t \mapsto \int_{\R^3} \pa_ih(t,v)\, \phi(v) \, dv$ is continuous on $[0,\mathcal{T}]$, for any smooth and compactly supported function $\phi$.  We can conclude that $h\in C([0,\mathcal{T}]; \dot{H}^1)$ by patching together the above facts.
 Indeed, thanks to the continuity of  $t \mapsto \|\nabla h(t)\|_{L^2}^2$, we know that
\beno
\lim_{s\rightarrow t}\|\na (h(t)-h(s))\|_{L^2}^2=2\|\na h(t)\|_{L^2}^2-2\lim_{s\rightarrow t}(\na h(s),\na h(t)).
 \eeno
We conclude by approximating $\na h(t)$ in $L^2$ by a sequence  $\phi_n\in C_c^\infty$.
% verifies $\|\phi_n - \na h(t)\|_{L^2}\rightarrow0$ as $n\rightarrow \infty$, then 
%\[\lim_{s\rightarrow t}(\na h(s)- \na h(t),\na h(t))= \lim_{s\rightarrow t}(\na h(s)-\na h(t),\phi_n)+\lim_{s\rightarrow t}(\na h(s)-\na h(t),(\na h)(t)-\phi_n)=0,\]
%which shows  $\lim_{s\rightarrow t}\|\na (h(t)-h(s))\|_{L^2}^2=0$. We ends the proof of the continuity in $\dot{H}^1$
Note finally that the formula appearing in the definition of $\mathcal{T}$ in Proposition \ref{localwell} is obtained by defining $C_7 := \f12C_{11}^{-5/9}$.
\end{proof} 

%\section{Description of a possible blowup}

\section{Weighted $H^1$ estimates and  proof of Proposition \ref{mainresult2}}

The main goal of this section is to get  estimates for weighted $H^1$ norms of solutions to the Landau equation with Coulomb potential (\ref{landau}) -- (\ref{13d}),  and then to use them in order to prove 
Proposition~\ref{mainresult2}. 
%As we addressed at the beginning of the previous section, here we only need to provide the {\it a priori} estimates for the equation.

\subsection{Weighted $\dot{H}^{1}$ estimate}

%Let $w_{m}=\lr{v}^{m}$ with $m\ge4$. 
Multiplying the equation for the derivatives of the Landau equation with Coulomb potential (\ref{landau}) -- (\ref{13d}), that is
(remembering that $h= f - \mu$ and that $\mu$ is the normalized Maxwellian given by (\ref{Defmu})). 
\begin{equation}\label{Eqpah}
\partial_t (\partial_k h)=Q(f,\partial_k h)+Q(\partial_k f, h)+Q(\partial_k h, \mu)+Q(h,\partial_k \mu),
\end{equation}
by $\lr{v}^{m}\,\pa_{k}h$, integrating with respect to $v$ and 
summing for $k=1,2,3$, we obtain (at the formal level)
\begin{equation}\label{weighted-h11}
\frac{1}{2}\frac{d}{dt}\|\nabla h\|^{2}_{L^{2}_{m/2}}=W_{1}+W_{2}+W_{3}+W_{4},
\end{equation}  
where $W_{1}, W_{2}, W_{3}$ and $W_{4}$ correspond to the terms of the right-hand side of~\eqref{Eqpah}.
\medskip 

%Next we shall derive ODE from~\eqref{weighted-h11} by estimating~$W_{i},i=1,2,3,4$.
We start our study by estimating the most significant terms, that is  $W_1$ and $W_2$.

\subsubsection{Estimate for $W_1$ and $W_2$} 

%Besides Corollary~\ref{coercivity2}, most estimates we need for $W_{1}$ and $W_{2}$ are included in
 The  following proposition enables to treat a large part of the terms coming out of $W_{1}$ and $W_{2}$:
\begin{prop}\label{E22}
Let $f$ be a nonnegative function satisfying the normalization (\ref{f0}), and $h = f - \mu$.
%, where $\mu$ is the normalized Maxwellian given by (\ref{Defmu}).
\medskip

 Then, the following estimates hold (for all $m \ge 0$ and some (absolute) constant $C>0$):
\begin{equation}\label{weight1}
 \iint  \bigg[ |v-v_{*}|^{-1}+|v-v_{*}|^{-2} \bigg]\, f(v_{*})\; |\nabla h(v)|^{2} \, \lr{v}^{m}\;dv_{*}dv 
 \leq C \, (1+\|\nabla h\|_{L^{2}}) \, \|\nabla h\|^{2}_{L^{2}_{m/2}},
\end{equation}
and 
\begin{equation}\label{weight2}
 %\color{blue}
 \begin{split}
 &\iint |v-v_{*}|^{-2} |\nabla f(v_{*})|\; |h(v)|\, |\nabla h(v)|\, \lr{v}^{m}  \;dv_{*}dv \\
 &\leq C \, (1+\|\nabla h\|^{2}_{L^2}) \|h\|_{H^{1}_{m/2}}\|\nabla h\|_{L^2_{m/2}}.
\end{split}
\end{equation}
\end{prop}

\begin{proof} 
For estimate~\eqref{weight1}, we bound  the integral over  $|v-v_{*}|\leq 1$  in the following way:
\begin{equation} \label{p1e}
\begin{split} 
& \iint_{|v - v_*| \le 1}  \bigg[ |v-v_{*}|^{-1}+|v-v_{*}|^{-2} \bigg]\, f(v_{*})\; |\nabla h(v)|^{2} \, \lr{v}^{m}\;dv_{*}dv \\
& \le {\big\|} |(|\cdot|^{-1}_{|\cdot|\leq 1}+|\cdot|^{-2}_{|\cdot|\leq 1})*f{\big \|}_{L^{\infty}}\|\nabla h\|^{2}_{L^{2}_{m/2}} \\
& \leq \||\cdot|^{-1}_{|\cdot|\leq 1}+|\cdot|^{-2}_{|\cdot|\leq 1}\|_{L^{\frac{6}{5}}}\|f\|_{L^{6}}\|\nabla h\|^{2}_{L^{2}_{m/2}} \\
&\leq C\,\|\nabla f\|_{L^{2}}\|\nabla h\|^{2}_{L^{2}_{m/2}} \\
&\leq C \, (1+\|\nabla h\|_{L^{2}})\|\nabla h\|^{2}_{L^{2}_{m/2}}.
\end{split}
\end{equation}
The integral over  $|v-{v_{*}}|\geq 1$ satisfies
\begin{equation} \label{p2e}
 \iint_{|v - v_*| \ge 1}  \bigg[ |v-v_{*}|^{-1}+|v-v_{*}|^{-2} \bigg]\, f(v_{*})\; |\nabla h(v)|^{2} \, \lr{v}^{m}\;dv_{*}dv
  \le C\|f\|_{L^{1}} \|\nabla h\|^{2}_{L^{2}_{m/2}}.
\end{equation}
 Then, estimate  \eqref{weight1} is a consequence of the bounds (\ref{p1e}) and (\ref{p2e}). 
\medskip

For estimate~\eqref{weight2}, using  $ \frac1{\sqrt{3}}\, \langle v\rangle \le \langle v_*\rangle  \le \sqrt{3} \, \langle v\rangle$
% $\lr{v}\approx \lr{v}_{*}$
 when $|v-v_{*}|\leq 1$, we see that the integral over  $|v-v_{*}|\leq 1$ is bounded by
\begin{equation} \label{q1e}
%\color{blue}
\begin{split}
&\iint_{|v - v_*| \le 1} |v-v_{*}|^{-2} |\nabla f(v_{*})|\; |h(v)|\, |\nabla h(v)|\, \lr{v}^{m}  \;dv_{*}dv \\
& \le  {\big \|} |\cdot|^{-2}_{|\cdot| \le 1}*| \lr{\cdot}^{\frac{m}{2}}\nabla f| {\big \|}_{L^3} {\big \|}  |h| |\nabla h| \lr{\cdot}^{\frac{m}{2}} {\big \|}_{L^{\frac{3}{2}}}\\
&\leq  \|\nabla f\|_{L^2_{m/2}} {\big \|} |\cdot|^{-2}_{|\cdot|<1}{\big \|}_{L^{\frac{6}{5}}}\|h \|_{L^6} \|\nabla h\|_{L^2_{m/2}}\\
&\leq C\,(1+\|\nabla h\|_{L^2_{m/2}})\|\nabla h\|_{L^2}\|\nabla h\|_{L^2_{m/2}}.
\end{split}
\end{equation}
Since $|\cdot|^{-2}_{|\cdot| \ge 1}$ lies in $L^{2}$, the integral over  $|v-v_{*}|\geq 1$ is bounded in the following way:
\begin{equation} \label{q2e}
%\color{blue}
\begin{split}
&  \iint_{|v - v_*| \ge 1} |v-v_{*}|^{-2} |\nabla f(v_{*})|\; |h(v)|\, |\nabla h(v)|\, \lr{v}^{m}  \;dv_{*}dv\\
& \le  C\, \| \nabla f\|_{L^{2}} \|h\|_{L^{2}_{m/2}}  \|\nabla h\|_{L^2_{m/2}} \\
&  \leq C\, \|h\|_{L^{2}_{m/2}} \|\nabla h\|_{L^2_{m/2}} +   C\,  \|\nabla h\|_{L^2}  \|h\|_{L^{2}_{m/2}}\|\nabla h\|_{L^2_{m/2}}.
\end{split}
\end{equation}
We get estimate (\ref{weight2}) by collecting the bounds (\ref{q1e}) and (\ref{q2e}). 
%\[
% \color{blue}
 %\begin{split}
 %&\iint |v-v_{*}|^{-2} |\nabla f(v_{*})|\; |h(v)|\, |\nabla h(v)|\, \lr{v}^{m}  \;dv_{*}dv \\
% &\leq C \, {\Big (} \|\nabla h\|_{L^2} (\| \nabla h\|_{L^{2}_{m/2}}+\|h\|_{L^{2}_{m/2}}) \|\nabla h\|_{L^2_{m/2}}+
 % ( \|\nabla h\|_{L^2} + \|h\|_{L^{2}_{m/2}})\|\nabla h\|_{L^2_{m/2}} {\Big )}\\
 %&\leq C \, (1+\|\nabla h\|^{2}_{L^2}) \|h\|_{H^{1}_{m/2}}\|\nabla h\|_{L^2_{m/2}}
%\end{split}
%\]
\end{proof}
\medskip

Next we estimate the terms $W_{1}$ and $W_{2}$. We start with the

\begin{prop}\label{E222}
Let $f \ge 0$ be such that $\int_{\R^3} f(v) \, dv =1$,   $\int_{\R^3} f(v)\, |v|^2 \, dv =3$, and such that 
$ \|f\|_{L^1_{15/2}} +  \| f\|_{L\,\log L} \le K$, 
for some $K>0$. We denote $h=f - \mu$. 
Then for all $m \ge 0$ and some constants $C^*(K)$, $C(K)>0$ depending only on $K$:
\begin{equation}\label{nnf}
W_{1} :=\lr{Q(f,\pa_{k}h),\lr{v}^{m}\pa_{k}h} \le  - \frac78\, C(K)\,\|\nabla^2 h\|^{2}
_{L^{2}_{m/2-3/2}} +  C^*(K)\, (1+\|\nabla h\|_{L^{2}}^2) \, \|\nabla h\|^{2}_{L^{2}_{m/2}} . 
\end{equation}
\end{prop}

\begin{proof} 
Using an integration by parts, we see that
\[
\begin{split}
W_{1}&=\lr{Q(f,\pa_{k}h),   \lr{\cdot}^m\, \pa_{k}h}\\
&=-{\Big (}\sum_{k,i,j} \int (a_{ij}*f) (\pa_{j}\pa_{k}h)[\pa_{i}( \lr{\cdot}^m\, \pa_{k}h)]dv {\Big )}+{\Big (}\sum_{k,i}\int 
(b_{i}*f)(\pa_{k} h)[\pa_{i}(  \lr{\cdot}^m\, \pa_{k}h)] dv {\Big )}\\
&=- {\Big (}\sum_{k} \int (a*f):(\nabla\pa_{k} h)\otimes (\nabla\pa_{k}h) \,  \lr{\cdot}^m \;dv
-\frac{1}{2} \sum_{k,i} \int (b_{i}*f) \,[\pa_{k}h]^2 \pa_{i} \,  \lr{\cdot}^m\, dv\\
&-\frac{1}{2}\int \sum_{i,j,k}(a_{ij}*f) \,[\pa_{k}h]^2\, \pa_{j}\pa_{i} \lr{\cdot}^m \,dv {\Big )}+ {\Big (}\frac{1}{2}\sum_{k,i} \int (b_{i}*f) \, [\pa_{k} h]^2\, \pa_{i}  \lr{\cdot}^m \,  dv\\
&+4\pi \sum_{k} \int f\, [\pa_{k} h]^2 \,  \lr{\cdot}^m\, dv{\Big )} .
\end{split}
\]
Thanks to Corollary 2.1, the first term of the expression above satisfies 
\[
\sum_{k} \int (a*f):(\nabla\pa_{k} h)\otimes (\nabla\pa_{k}h) \,  \lr{\cdot}^m \;dv\geq C(K)\,\|\nabla^2 h\|^{2}_{L^{2}_{m/2-3/2}}.
\]
Then, thanks to  H\"{o}lder's inequality and Sobolev embedding (${\dot{H}}^1 \subset L^6$), \par
 
\[
%\color{blue}
\begin{split}
& {\Big |} \sum_{k} \int f\,[\pa_{k} h]^2 \,  \lr{\cdot}^m\, dv{\Big |}\leq C^* \|f\|_{L^3_{3/2}} \|\na h\|_{L^2_{m/2}}\|\nabla((\na h)\lr{\cdot}^{m/2-3/2})\|_{L^{2}} \\
&  \le  C^* \|f\|_{L^1_{15/2}}^{1/5} \,  \|f\|_{L^6}^{4/5}  \|\na h\|_{L^2_{m/2}} \bigg( \|\nabla^2 h\|_{L^{2}_{m/2-3/2}} 
+ \|\nabla h\|_{L^2_{m/2 - 5/2}}\bigg)  \\
& \le C^*(K)\,  \|\nabla f\|_{L^2}^{4/5}  \|\na h\|_{L^2_{m/2}}  \|\nabla^2 h\|_{L^{2}_{m/2-3/2}}  
 +  C^*(K)\,  \|\nabla f\|_{L^2}^{4/5}  \|\na h\|_{L^2_{m/2}}^2 \\
&  \le  \frac{C^*(K)}8 \, \|\nabla^2 h\|^{2}_{L^{2}_{m/2-3/2}} + C^*(K)  \,(1 + \|\nabla h\|_{L^2})^{8/5}  \|\na h\|_{L^2_{m/2}}^2
 + C^*(K)   \,(1 + \|\nabla h\|_{L^2})^{4/5}  \|\na h\|_{L^2_{m/2}}^2  .
\end{split}
\]
Using the estimates above and Proposition~\ref{E22}, eq. (\ref{weight1}),  we end up with estimate (\ref{nnf}).

% equation~\eqref{weighted-h11} implies
%\begin{equation}\label{weighted-h12}\begin{split}
%& \frac{1}{2}\frac{d}{dt}\|\nabla h\|^{2}_{L^{2}_{m/2}} +C_{K}\|\nabla\nabla h\|^{2}_{L^{2}_{m/2-3/2}}\\
%&{\hskip 1cm}\le C(1+\|\nabla h\|_{L^{2}_{m/2}}^2)\|\nabla h\|^{2}_{L^{2}_{m/2}}+W_{2}+W_{3}+W_{4}.
%\end{split} \end{equation} 
\end{proof}

We now turn to  the

\begin{prop}\label{E223}
Let $f \ge 0$ be such that $\int_{\R^3} f(v) \, dv =1$,   $\int_{\R^3} f(v)\, |v|^2 \, dv =3$, and such that 
$ \|f\|_{L^1_{45}} +  \| f\|_{L\,\log L} \le K$ 
for some $K>0$. We denote $h=f - \mu$. Let $C(K)$ be the same constant as in Proposition~\ref{E222} and $0\leq m\leq 76$.
Then there exists some constant $C^*(K)$ depending only on $K$ such that:
\begin{equation}\label{nnf3}
%\color{blue}
\begin{split}
W_{2} &:= \lr{Q(\pa_{k} f,h), \lr{v}^{m}\, \pa_{k}h}    \le   \frac{3C(K)}{8}\, \|\nabla^2 h\|^{2}
_{L^{2}_{m/2-3/2}} \\
& +\,C^*(K)\,  (1+\|\nabla h\|^{2}_{L^2}) \|h\|_{H^{1}_{m/2}}\|\nabla h\|_{L^2_{m/2}}.
\end{split}  
\end{equation}
%$$ + \, C_{\eta}\, (1+\|h\|_{L^{2}_{m/2}}^2 + \|\nabla h\|_{L^{2}_{m/2}}^2) \, \|\nabla h\|^{2}_{L^{2}_{m/2}}  +  C_{\eta}\, \|h\|_{L^{2}_{m/2}}^2. $$
\end{prop}

\begin{proof} 

Using an integration by parts, we see that
\[
\begin{split}
W_{2}&=\lr{Q(\pa_{k} f,h), \lr{v}^{m} \pa_{k}h}\\
&=\sum_{k,i,j} \int -(\pa_{k}a_{ij}*f)(\pa_{j}h) \bigg[(\pa_{i}\pa_{k}h) \, \lr{\cdot }^{m}+(\pa_{k}h) \pa_{i}  \lr{\cdot }^{m} \bigg] dv\\
&+\sum_{k,i}\int(b_{i}*\pa_{k}f)(h) {\bigg[} (\pa_{i}  \lr{\cdot }^{m})\pa_{k}h+ (\pa_i\pa_k h)   \lr{\cdot }^{m} \,{\bigg]} \; dv .
\end{split}
\]
Using first  Proposition~\ref{E22}, eq. (\ref{weight1}), we obtain the estimate 
\[
{\Big |}\sum_{k,i,j}\int (\pa_{k}a_{ij}*f)(\pa_{j}h)(\pa_{k}h) \,\lr{\cdot }^{m}\, dv {\Big |}\leq C (1+\|\nabla h\|_{L^{2}})\|\nabla h\|^{2}_{L^{2}_{m/2}}.
\]
%Proceeding as in  the proof of Proposition~\ref{E22} and using the facts that  $ \frac1{\sqrt{3}}\, \langle v\rangle \le \langle v_*\rangle  \le \sqrt{3} %\, \langle v\rangle$
%$\br{v}\approx\br{v_*}$
% when $|v-v_*|<1$ 
%and $|v-v_*|^{-2}\leq 2\, \br{v_*}^2\br{v}^{-2}$ when $|v-v_*|>1$ (in order to distribute the weights), we get  \par
 
Also, still treating separately $|v-v_*| \le 1$ and $|v-v_*|>1$, and observing that $|\cdot|^{-2}_{|\cdot|<1}\in L^{{4/3}}$, we 
compute
\[
%\color{blue}
\begin{split}
&{\Big |} \sum_{k,i,j} \int (\pa_{k}a_{ij}*f)(\pa_{j}h) (\pa_{i}\pa_{k}h) \lr{\cdot }^{m} dv{\Big |} \\
&\leq C^*{\Big (} \|  f\|_{L^4_{3/2}} \|\nabla h\|_{L^2_{m/2}} 
\|\nabla^2 h\|_{L^2_{m/2-3/2}}+\|f\|_{L^1_2}\|\nabla h\|_{L^2_{m/2-1/2}}
\|\nabla^2 h\|_{L^2_{m/2-3/2}}  {\Big )} \\
& \leq C^*\bigg( \|f\|_{L^1_{15}}^{1/10} \,  \|f\|_{L^6}^{9/10}  \|\na h\|_{L^2_{m/2}} +\|\nabla h\|_{L^2_{m/2-1/2}}\bigg)
\|\nabla^2 h\|_{L^{2}_{m/2-3/2}}\\
&   \le  \frac{C(K)}8 \, \|\nabla^2 h\|^{2}_{L^{2}_{m/2-3/2}} + C^*(K)  \,(1 + \|\nabla h\|_{L^2})^{9/5}  \|\na h\|_{L^2_{m/2}}^2 .
\end{split}
\]

\medskip

Then, thanks to Proposition~\ref{E22} again, 
%eq. (\ref{weight2}),
%if $m\geq 4$ we have 
\[
%\color{blue}
%\begin{split}
 {\Big |} \sum_{k,i}\int(b_{i}*\pa_{k}f)(h) (\pa_{i} \,\lr{\cdot }^{m}) \, \pa_{k}h \; dv {\Big |}
%&\leq C{\Big (} {\big \|} |\cdot|^{-2}_{|\cdot|<1}*|w_{\frac{m}{2}}\nabla f| {\big \|}_{L^3} {\big \|} |h| |\nabla h|w_{\frac{m}{2}} {\big \|}_{L^{\frac{3}{2}}}+\| \nabla f\|_{L^{2}} \|h\|_{L^{2}_{m/2-1}}  \|\nabla h\|_{L^2_{m/2}} {\Big )}\\
%&\leq C{\Big (} \|\nabla f\|_{L^2_{m/2}} {\big \|} |\cdot|^{-2}_{|\cdot|<1}{\big \|}_{L^{\frac{6}{5}}}\|h \|_{L^6} \|\nabla h\|_{L^2_{m/2}} +\|h\|_{L^{2}_{m/2-1}} (\|\nabla h\|_{L^2_{m}}+\|\nabla h\|^{2}_{L^2_{m/2}}) {\Big )}\\
\leq C^* \, (1+\|\nabla h\|^{2}_{L^2}) \|h\|_{H^{1}_{m/2}}\|\nabla h\|_{L^2_{m/2}}.
%\end{split}
\]
 
\par
Finally, we estimate  $|\sum_{k,i}\int(b_{i}*\pa_{k}f)(h) (\pa_i\pa_k h)\,\lr{\cdot }^{m}\,  \; dv|$. Recall that $f=h+\mu$. 
Thanks to an integration by parts, we have 
\[
\begin{split}
&\sum_{k,i}\int(b_{i}*\pa_{k}\mu)(h) (\pa_i\pa_k h)\,\lr{\cdot }^{m}\,  \; dv\\
&=-\sum_{k,i}\int(b_{i}*\pa_{k}\mu) (\pa_k h)\, \pa_i(h\lr{\cdot }^{m})\,  \; dv+8\pi\sum_{k}\int (\pa_{k}\mu) (\pa_k h)\, (h\lr{\cdot }^{m})\,dv.
\end{split}\]
Therefore, we deduce that
\[
 \bigg|\sum_{k,i}\int(b_{i}*\pa_{k}\mu)(h) (\pa_i\pa_k h)\,\lr{\cdot }^{m}\,  \; dv \bigg| \le C^*\, \|h\|_{H^1_{m/2}}\|\na h\|_{L^2_{m/2}}.
\]

 We now turn to the term $|\sum_{k,i}\int(b_{i}*\pa_{k}h)(h) (\pa_i\pa_k h)\,\lr{\cdot }^{m}\,  \; dv|$. 
%Note $|\cdot|^{-2}_{|\cdot|<1}\in L^{{4/3}}$. 
The integral over $|v-v_{*}|\leq 1$ is bounded by 
\[
%\color{blue}
\begin{split}
&C^*\iint_{|v - v_*| \le 1} |v-v_{*}|^{-2} |\nabla h(v_{*})|\; |h(v)|\, |\nabla^2 h(v)|\, \lr{v}^{m}  \;dv_{*}dv \\
& \le C^* \, {\big \|} |\cdot|^{-2}_{|\cdot| \le 1}*| \lr{\cdot}^{\frac{m}{2}}\nabla h| {\big \|}_{L^4} {\big \|}  |h| |\nabla^2 h| \lr{\cdot}^{\frac{m}{2}} 
{\big \|}_{L^{\frac{4}{3}}}\\
&\leq C^*\, \|\nabla h\|_{L^2_{m/2}} {\big \|} |\cdot|^{-2}_{|\cdot|<1}{\big \|}_{L^{\frac{4}{3}}}\|h \|_{L^4_{3/2}} \|\nabla^2 h\|_{L^2_{m/2-3/2}}\\
&\leq C^*\, \|\nabla h\|_{L^2_{m/2}}\|h\|^{1/10}_{L^{1}_{15}} \|\nabla h\|^{9/10}_{L^2}\|\nabla^2 h\|_{L^2_{m/2-3/2}}\\
&\leq \frac{C(K)}{8}\|\nabla^2 h\|^{2}_{L^2_{m/2-3/2}}+ C^*(K)(1+\|\nabla h\|^{2}_{L^2})\|\nabla h\|^{2}_{L^2_{m/2}}.
\end{split}
\]
Notice now that $|\cdot|^{-2}_{|\cdot|\geq 1}\in L^{{14/9}}$. Then,
the integral over $|v-v_{*}|\geq 1$ is bounded by 
\[
%\color{blue}
\begin{split}
&C^*\iint_{|v - v_*| \geq 1} |v-v_{*}|^{-2} |\nabla h(v_{*})|\; |h(v)|\, |\nabla^2 h(v)|\, \lr{v}^{m}  \;dv_{*}dv \\
& \le {\big \|} |\cdot|^{-2}_{|\cdot| \geq 1}*| \nabla h| {\big \|}_{L^7} {\big \|}  |h| |\nabla^2 h| \lr{\cdot}^{m} 
{\big \|}_{L^{\frac{7}{6}}}\\
&\leq \|\nabla h\|_{L^2} {\big \|} |\cdot|^{-2}_{|\cdot| \geq 1}{\big \|}_{L^{\frac{14}{9}}}\|h \|_{L^{14/5}_{m/2+3/2}} \|\nabla^2 h\|_{L^2_{m/2-3/2}}\\
&\leq C^*\, \|\nabla h\|_{L^2}\|h\|^{8/35}_{L^{1}_{m/2+105/16}} \|\nabla h\|^{27/35}_{L^2_{m/2}}\|\nabla^2 h\|_{L^2_{m/2-3/2}}\\
&\leq \frac{C(K)}{8}\|\nabla^2 h\|^{2}_{L^2_{m/2-3/2}}+ C^*(K) (1+\|\nabla h\|^{2}_{L^2})\|\nabla h\|^{2}_{L^2_{m/2}},
\end{split}
\]
since $m/2+105/16<45$ when $m\leq 76$.
\medskip

Finally, we get estimate (\ref{nnf3}) by regrouping all the estimates above.
%Applying Young's inequality  for product to above results, we have 
%\begin{equation} \begin{split}
%& \frac{1}{2}\frac{d}{dt}\|\nabla h\|^{2}_{L^{2}_{m/2}} +C_{K}\|\nabla\nabla h\|^{2}_{L^{2}_{m/2-3/2}}\\
%&{\hskip 1cm}\le C(1+\|\nabla h\|_{L^{2}_{m/2}}^2+\| h\|_{L^{2}_{m/2}}^2)\|\nabla h\|^{2}_{L^{2}_{m/2}}+C\| h\|_{L^{2}_{m/2}}^2+W_{3}+W_{4}.
%\end{split}\end{equation} 
%We end up the proof of the Proposition by using Young's inequality.
\end{proof}

\subsubsection{ Estimates for $W_{3}$ and $W_{4}$}  

We now estimate jointly  the terms $W_{3}$ and $W_{4}$. 

\begin{prop}\label{E234}
Let $f \ge 0$ be such that $\int_{\R^3} f(v) \, dv =1$,   $\int_{\R^3} f(v)\, |v|^2 \, dv =3$.
% and such that $ ||f||_{L^1_{5}(\R^3)} +  || f||_{L\,\log L} \le K$. We denote $h=f - \mu$. 
Then for all $m \ge 2$,  and some (absolute) constant $C>0$:
% depending only on $K$:
\begin{equation}\label{nnf34}
%\color{blue}
W_{3} + W_{4} : =\lr{Q(\pa_{k} h,\mu), \lr{\cdot }^{m} \,\pa_{k}h}  + \lr{Q( h, \pa_{k}\mu)\, \lr{\cdot }^{m}\, \pa_{k}h}
 \le  C\,\|\nabla h\|_{L^{2}}\|h\|_{H^{1}_{m/2}}. 
\end{equation}
\end{prop}
\begin{proof} 
Using integrations by parts, we compute 
\[
\begin{split}
W_{3}&=\lr{Q(\pa_{k} h,\mu), \lr{\cdot }^{m}\,\pa_{k}h}\\
&=\sum_{k,j} \int ( b_{j}*\pa_{k}h)(\pa_{j}\mu)  \lr{\cdot }^{m}\,\pa_{k} h \,dv+
\sum_{k,i,j} \int ( \pa_{k} a_{ij}*h)(\pa_{i}\pa_{j}\mu)  \lr{\cdot }^{m}\,\pa_{k} h\, dv\\
&+\sum_{k} 8\pi\int \mu \; |\pa_{k} h|^{2}\,\lr{\cdot }^{m}\,dv-\sum_{k,i} \int (b_{i}*\pa_{k}h)(\pa_{i}  \mu) (\pa_{k} h) \,
\lr{\cdot }^{m}\,  dv,
\end{split}
\]
and 
\[
\begin{split}
W_{4}&=\lr{Q( h, \pa_{k}\mu), \lr{\cdot }^{m}\,\pa_{k}h}\\
&=\sum_{k,j} \int  (b_{j}*h)(\pa_{j}\pa_{k}\mu)(\pa_{k} h) \, \lr{\cdot }^{m}\, dv
+\sum_{k,i,j} \int (a_{ij}*h)(\pa_{i}\pa_{j}\pa_{k}\mu) (\pa_{k} h)\, \lr{\cdot }^{m}\,dv \\
&-\sum_{k,i} {\Big (}\int (b_{i}*\pa_{i} h)( \pa_{k}\mu) (\pa_{k} h)\, \lr{\cdot }^{m}\, dv+
\int (b_{i}* h)( \pa_{i}\pa_{k}\mu) (\pa_{k} h) \, \lr{\cdot }^{m}\, dv{\Big )}.\\
\end{split}
\]
As a consequence, using the elementary inequality $<v_*>^2 \,1_{|v - v_*| \le 1} \le |v-v_*|^2\, e^{v^2/4}\, 1_{|v - v_*| \le 1}$,
\[
%\color{blue}
\begin{split}
W_{3} + W_{4}  & \le C\, \bigg[\, \int |b| * |\nabla h|\, |\nabla h| \, \mu^{1/2} 
+ \int |b| * |h|\,\, |\nabla h| \, \mu^{1/2} + \int  |\nabla h|^2 \, \mu^{1/2} \bigg] \\
& \le C\, ||\,  [ |b|\,1_{|\cdot| \le 1}] * (|h| + |\nabla h|)\, ||_{L^2}\, ||\nabla h ||_{L^2}\\
& +\, C \iint_{|v - v_*|\ge 1} \lr{v_{*}}^{-2}\bigg[ |h(v_*)| + |\nabla h(v_*)| \bigg] \, |\nabla h(v)| \, \mu^{1/2}(v)\, dv dv_* + 
C\, \|\nabla h\|^{2}_{L^{2}}\\
&\le C\, \|\nabla h\|^{2}_{L^{2}} + C\, \|h\|_{L^{2}}\, \|\nabla h\|_{L^{2}}  + C\,  \|\nabla h\|_{L^{2}} \,( ||h||_{L^2_1} + 
 ||\nabla h||_{L^2_1} ) \\
&  \le  C\,  \|\nabla h\|_{L^{2}} \,( ||h||_{L^2_{m/2}} + ||\nabla h||_{L^2_{m/2}} )\leq C\,\|\nabla h\|_{L^{2}}\|h\|_{H^{1}_{m/2}},
\end{split}
\]
remembering that $m\geq 2$.
%which concludes the proof of the Proposition.
%We finally conclude that \begin{equation}\label{weighted-h13}\begin{split}
%& \frac{1}{2}\frac{d}{dt}\|\nabla h\|^{2}_{L^{2}_{m/2}} +C_{K}\|\nabla\nabla h\|^{2}_{L^{2}_{m/2-3/2}}\\
%&{\hskip 1cm}\le C(1+\|\nabla h\|_{L^{2}_{m/2}}^2+\| h\|_{L^{2}_{m/2}}^2)\|\nabla h\|^{2}_{L^{2}_{m/2}}+C\| h\|_{L^{2}_{m/2}}^2.
%\end{split}\end{equation}
\end{proof}
 
\subsection{$L^{2}$ estimate}

\begin{prop}\label{E25p}
Let $f \ge 0$ be such that $\int_{\R^3} f(v) \, dv =1$,   $\int_{\R^3} f(v)\, |v|^2 \, dv =3$.
 and such that $ \|f\|_{L^1_{45}(\R^3)} +  \| f\|_{L\,\log L} \le K$ for some $K>0$. We denote $h=f - \mu$. 
 Let $C(K)$ be the same constant as in Proposition~\ref{E222}.
Then for all $m \ge 4$,  and some constant $C^*(K)$ depending only on $K$:
% depending only on $K$:
\begin{equation}\label{nnf35}
 \lr{Q(f, h),  \lr{\cdot }^{m}\, h}+\lr{Q( h, \mu),  \lr{\cdot }^{m}\, h}
\end{equation}
$$  \le   - C(K) \, \|\nabla h\|^{2}_{m/2-3/2} +
C^*(K)\, (1+\|\nabla h\|_{L^{2}})(\|\nabla h\|^{2}_{L^{2}_{m/2}}+\| h\|^{2}_{L^{2}_{m/2}}) . $$
\end{prop}

\begin{proof} 
 Using integration by parts, we obtain the decomposition: 
\[
\begin{split}
\lr{Q(f, h),  \lr{\cdot }^{m}\,h}&=-\int (a*f):(\nabla h)\otimes (\nabla h) \,  \lr{\cdot }^{m}\,\;dv\\
&-\sum_{i,j}\int (a_{ij}*f) (\pa_{j} h) h (\pa_{i}  \lr{\cdot }^{m}\,) dv
+\sum_i \int (b_{i}*f)( h)[\pa_{i}( \lr{\cdot }^{m}\, h)] dv\\
&=: -\mathbb{E}_1-\mathbb{E}_2+\mathbb{E}_3. \\
\end{split}
\]
Using Proposition 2.1 (and keeping in mind the arguments used in the proof of Corollary 2.1), we see that
\begin{equation}\label{coercivity-l2}
\mathbb{E}_{1}\geq C(K)\, \|\nabla h\|^{2}_{m/2-3/2} .
\end{equation}
Using then the same computations as in the proof of Proposition~\ref{E22}, we see that
\begin{equation}\label{U1b}
|\mathbb{E}_2| \le C\,  \|\nabla f\|_{L^{2}}\, \|h\|_{L^{2}_{m/2}}\|\nabla h\|_{L^{2}_{m/2}} 
+ C\, ||f||_{L^1}  \, \, \|h\|_{L^{2}_{m/2}}\|\nabla h\|_{L^{2}_{m/2}}  
\end{equation}
$$ \le 
 C  \, (1+\|\nabla h\|_{L^{2}}) \, \|h\|_{L^{2}_{m/2}}\|\nabla h\|_{L^{2}_{m/2}}. $$
Similarly 
 \begin{equation}\label{U1bb}
|\mathbb{E}_3|\le 
 C  \, (1+\|\nabla h\|_{L^{2}})(\|h\|_{L^{2}_{m/2}}\|\nabla h\|_{L^{2}_{m/2}}+\| h\|^{2}_{L^{2}_{m/2}}).
\end{equation}

Using again an integration by parts, we see also that 
\begin{equation}\label{U2}
|\lr{Q( h, \mu),  \lr{\cdot }^{m}\, h}|
 ={\Big |}\sum_{i} \int {\Big (}\sum_{j}-(a_{ij}*h)(\pa_{j}\mu)  [\pa_{i} ( \lr{\cdot }^{m}\, h)]+
(b_{i}*h)( \mu)[\pa_{i} ( \lr{\cdot }^{m}\, h)] {\Big )} dv {\Big |}
\end{equation}
$$
 \le C\, \int (|a| +|b|) * |h|\, (|h| + |\nabla h|) \, \mu^{1/2} \, dv $$
 $$ \le C\, \|\,  [(|a| + |b|)\,1_{|\cdot| \le 1}] * |h| \, \|_{L^2}\,(\|h \|_{L^2} +  \|\nabla h \|_{L^2})$$
$$ +\, C \iint_{|v - v_*|\ge 1} \lr{v_{*}}^{-1} |h(v_*)|  \, \bigg[ |h(v)| +  |\nabla h(v)| \bigg] \, \mu^{1/4}(v)\, dv dv_* $$
$$ \le C\, \bigg( \|h\|_{L^{2}}  +  \|\nabla h\|_{L^{2}} \bigg) \, \|h\|_{L^2_2}, $$
where $|v-v_{*}|^{-1}\mathrm{1}_{|v-v_*|\ge1}\leq \lr{v_{*}}^{-1}\lr{v}$ is used.
\medskip

%\textcolor{red}{It seems to me that  $ ||h||_{L^1_{-1}} \le C_{\delta}\, ||h||_{L^2_{1+ \delta}}$ for all $\delta>0$, but here it is used above for $
%\delta=0$. Am I wrong? Laurent}
\par
Collecting all terms and remembering that $m \ge 4$, we conclude the proof of Proposition~\ref{E25p}. 
\end{proof}

\subsection{End of the proof of Proposition \ref{mainresult2}}

 For the end of the proof, we perform the computations for a smooth $C^2_t(\mathcal{S})$ solution $f\ge 0$ of Landau equation with Coulomb potential (\ref{landau}) -- (\ref{13d}).
 We should in fact repeat here the process of approximation presented in the proofs of Theorem \ref{maintheorem1} and Proposition \ref{localwell}. We do not write it for the sake of readability, since no new
argument is used to deal with the approximation process. 
\medskip

 We first observe that thanks to the   assumptions of Theorem~ \ref{mainresult2}
 and Lemma~\ref{ml}, there exists a constant $K>0$ such that
 \[
 \sup_{t>0} (\|f(t)\|_{L^{1}_{45}}+\|f(t)\|_{L\log L})\leq K.
 \]
 Then we compute
 (for $4\le  m \le 76$) the quantity $\frac{1}{2}\frac{d}{dt}\|\nabla h\|^{2}_{L^{2}_{m/2}}$.  By
 using the computations (\ref{Eqpah}), (\ref{weighted-h11}) and Proposition  \ref{E222}, Proposition \ref{E223}, Proposition \ref{E234}, we end up with the estimate
 \begin{equation}\label{weighted-h13}
% \color{blue}
 \frac{1}{2}\frac{d}{dt}\|\nabla h\|^{2}_{L^{2}_{m/2}} + \frac{C(K)}2 \,\|\nabla^2 h\|^{2}_{L^{2}_{m/2-3/2}}
\leq C^*(K) \, (1+\|\nabla h\|^{2}_{L^2}) \|h\|_{H^{1}_{m/2}}\|\nabla h\|_{L^2_{m/2}}.
\end{equation}

Then, multiplying eq.~\eqref{Eqh} by $ \lr{\cdot }^{m}$, and integrating with respect to $v$, we compute
\begin{equation}\label{h11}
\frac{1}{2}\frac{d}{dt}\| h\|^{2}_{L^{2}_{m/2}}=\lr{Q(f, h),   \lr{\cdot }^{m}\, h}+\lr{Q( h, \mu),   \lr{\cdot }^{m}\, h}.
\end{equation}  
Using Proposition \ref{E25p} and computation (\ref{h11}),  we get the differential inequality   
\begin{equation}\label{h11-in}
\frac{1}{2}\frac{d}{dt}\| h\|^{2}_{L^{2}_{m/2}}+ C(K)\,\|\nabla h\|^{2}_{m/2-3/2}\leq 
C^*(K) \, (1+\|\nabla h\|_{L^{2}})(\|\nabla h\|^{2}_{L^{2}_{m/2}}+\| h\|^{2}_{L^{2}_{m/2}}).
\end{equation}  

 Patching together inequalities \eqref{h11-in} and
  \eqref{weighted-h13}, we finally obtain  the differential inequality 
\begin{equation}\label{weightedH1}
%\color{blue}
 \frac{1}{2}\frac{d}{dt}\| h\|^{2}_{H^{1}_{m/2}} + \frac{C(K)}2 \,\|\nabla  h\|^{2}
_{H^{1}_{m/2-3/2}} 
 \le C^*(K)\,
%(\|f_0\|_{L^1_{55}}, \|f_0\|_{L\log L}) \,
 (1+\|\nabla h\|_{L^{2}}^2)\|  h\|^{2}_{H^{1}_{m/2}}.
\end{equation} 
We emphasize that from the proof of Proposition  \ref{E222}, Proposition \ref{E223}, Proposition \ref{E234} and Proposition \ref{E25p}, 
the constants $C^*(K)$, $C(K)>0$ in the above inequality only depend on $K$ such that
 $\|f_0\|_{L^1_{55}} + \|f_0\|_{L\log L} \le K$. 
\smallskip 
%Now we  are in a position to prove Theorem \ref{mainresult2}.

%\begin{proof}[Proof of Theorem \ref{mainresult2}]
Thanks to Proposition \ref{interpH1a}, we know that, for some $C, C_3,C_4 >0$, and some $k_3>2/5$ (we take $l=55$, $\theta = 15/4+7$ and $ q_{l,\theta} \sim -3.79$ with the notations of Lemma \ref{ml}, then $k_3>3$), 
%\textcolor{red}{Using the values above, we get  $k_3 > (4/5) * 2.693..$ which does not imply  $k_3>5/2$. Am I right? Note that we could use $l=55$ rather than $l=45$ if we keep the statement of Prop. \ref{mainresult2} as it is. Laurent} \textcolor{blue}{Yes, we should 
%use $l=55$ for it is used in other places also. In the proof of this Proposition, we only need  $k_3>2/5$ while $k_3>5/2$ was a typo.} 
\beno 
%\color{blue}
 \frac{C(K)}2\|\nabla  h\|^{2}
_{H^{1}_{1/2}}\ge C_3\, \|h\|_{H^1_{2}}^{\f{14}5}\|h\|_{L^1_{15/4+7}}^{-\f45}-C\|h\|_{L^1_{-3/2}}^2\ge   C_4\,(1+t)^{k_3}\|h\|_{H^1_{2}}^{\f{14}5} -C\|h\|_{L^2_2}^2. 
\eeno 
%where $k_{3}=q_{55,15/4+7}>2/5$ by~\eqref{qltheta}.
In the inequality above and in the rest of the proof, we do not make explicit the (existing) dependence of $C, C_3,C_4 >0$, and $k_3>2/5$ with respect to $K$. 
\medskip

%\textcolor{red}{In the text below, it would be useful to make more explicit how small $\epsilon_0$ is chosen, in order to write a more precise %proposition. Laurent}\par
 Denoting $Y^2(t):=\|h(t)\|_{H^1_{2}}^2$, we therefore get the differential inequality (for some $C_5>0$ only depending on $K$):
%it satisfies the following ordinary differential inequality:
\ben\label{WODE}\frac{d}{dt}Y^2(t)+C_4\,(1+t)^{k_3}\,Y(t)^{\f{14}5}\le C_5\, (Y^4(t)+Y^2(t)).\een 
Remembering that $\|h_{0}\|_{L^{1}_{45}}$  is bounded
and that the initial condition is supposed to satisfy  $ \|h_0\lr{\cdot}^2\|_{\dot{H^1}}\le \epsilon_0 \ll 1$, we see that 
by interpolation,  the differential  inequality (\ref{WODE}) is complemented with the initial datum
 $Y^2(0)=\tilde{\epsilon} \ll 1$ (note that here and in the sequel,  the way in which $\tilde{\epsilon}$ is small depends in fact (only) on $K$).
%(which leads to the condition $\epsilon_0\ll K^{-\f23}$ in Proposition \ref{mainresult2}).  
%\textcolor{red}{I cheked that the smallness of $\epsilon_0$ which is needed depends on $K$ in a complicated way because it depends on $C_4$ and $C_5$ which themselves depend on $K$. So I propose to not give the explicit dependence and just say that it depends only on $K$.  Laurent} \textcolor{blue}{Agree}

 We now consider $T^*:=\sup\{t>0\,\,|\,\,Y^4(t)\le Y^2(t)\} = \sup\{t>0\,\,|\,\,Y(t)\le 1\} $.
 For $t\in[0,T^*]$, the differential  inequality 
\beno \frac{d}{dt}Y^2(t)\le 2\,C_5\, Y^2(t)\eeno
holds. 
It implies that $\forall t\in[0, T_1 := (2\,C_5)^{-1}|\log (\f12|\log\tilde{\epsilon}|^{-1}\tilde{\epsilon})^{-1}|]$, 
the inequality $Y^2(t)\le \f12|\log\tilde{\epsilon}|^{-1}\le1$ also holds.  Thus, $T^*\ge T_1$.
\medskip

 We now use a contradiction argument in order to show that solutions of inequality \eqref{WODE} globally exist.
 If the set $\{t>0\, |\, Y^2(t)=|\log \tilde{\epsilon}|^{-1}\}$ is empty, then  this is automatically true.
If it is not the case, we define 
$T^{**}:=\inf\{t>0\,\,|\,\,Y^2(t)=|\log \tilde{\epsilon}|^{-1}\}$. Then there exists a time $T_2$ defined by
 $T_2:=\sup\{t\le T^{**}\,\,|\,\,Y^2(t)=\f12|\log \tilde{\epsilon}|^{-1}\}$. 
Because of the definition of $T_1$ and $T_2$, we see that  $T^* > T^{**}> T_2 \ge T_1$, and $Y^2(t)\, |\log \tilde{\epsilon}| \in [1/2, 1]$ when $t\in[T_2,T^{**}]$.  In particular, in the interval $[T_2,T^{**}]$, we have
\beno \frac{d}{dt}Y^2(t)+C_4(1+T_1)^{k_3}Y(t)^{\f{4}5} Y^2(t)\le 2\,C_5\, Y^2(t), \eeno
where
\beno  (1+|\log \tilde{\epsilon}|)^{k_3}|\log \tilde{\epsilon}|^{-\f25} = O_{\tilde{\epsilon} \to 0} \bigg[  C_4(1+T_1)^{k_3}Y(t)^{\f{4}5} \bigg].\eeno
It implies that if $k_3>\f25$ and $\tilde{\epsilon} >0$ is sufficiently small (depending on $K$ again),
%\par
%\textcolor{red}{At this point we should maybe explain in which sense $\tilde{\epsilon}$ has to be small, does the smallness depend only on K? Laurent}
%\par
 $C_4(1+T_1)^{k_3}Y(t)^{\f{4}5}\ge 2C_5$. Then $Y^2$ is decreasing on the interval $[T_2,T^{**}]$, so that   $Y^2(T^{**})\le Y^2(T_2)=\f12|\log \tilde{\epsilon}|^{-1}$. This is not compatible   with the definition of $T^{**}$, which entails that the set $\{t>0\,\,|\,\,Y^2(t)=|\log \tilde{\epsilon}|^{-1}\}$ is empty.
As a consequence,  we get the global existence for solutions of \eqref{WODE}, and those solutions moreover satisfy the bound
$\sup_{t\ge0}Y^2(t)\le |\log \tilde{\epsilon}|^{-1}$.
\medskip

They satisfy therefore the following modified differential inequality
% rewrite \eqref{WODE} as follows:
\beno \frac{d}{dt}Y^2(t)+C_4(1+t)^{k_3}Y(t)^{\f{14}5}\le 2\,C_5 \,Y^2(t).\eeno
Splitting the interval $[0,\infty]$ into the two sets  $\{t>0\,\,|\,\,C_4(1+t)^{k_3}Y(t)^{\f{14}5}\le 4\,C_5\, Y^2(t)\}$ and 
$\{t>0\,\,|\,\,C_4(1+t)^{k_3}Y(t)^{\f{14}5}>  4\,C_5\, Y^2(t)\}$, we conclude that for some constant $C>0$ (remembering that $k_3>3$)
\beno Y(t)\le C\,(1+t)^{- \f54k_3} \le C\,(1+t)^{- \f{15}4}. \eeno
% \end{proof}
\medskip

 We recall that the estimates obtained in this subsection hold for a smooth solution of the Landau equation  (\ref{landau}) -- (\ref{13d}), and that,
 as in Proposition \ref{thmappro}, they also hold uniformly w.r.t. $\epsilon \in ]0,1[$ for smooth solutions of 
the approximated equation (\ref{approeq}) -- (\ref{ae}), with suitably mollified initial datum (we recall that such solutions are known to exist and be unique).  It is then possible to pass to the (weak weighted $L^1$) limit in the final estimate  
$$  \|h^\epsilon\|_{H^1_2} \le C\,(1+t)^{- \f54k_3} \le C\,(1+t)^{- \f{15}4}, $$ and get the existence of the strong global nonnegative solution to  Landau equation  (\ref{landau}) -- (\ref{13d}) announced in the Proposition \ref{mainresult2}  . The uniqueness is obtained 
thanks to a variant of the arguments used  in the proof of Theorem \ref{maintheorem1} and Proposition \ref{localwell}.

\section{Investigation of a potential blowup}

Here, we prove Proposition \ref{des}, which provides estimates describing the potential blowup (in $\dot{H}^1$) of solutions to eq. (\ref{landau}) -- (\ref{13d}).
\medskip

%with quantitative estimates for \eqref{E1}.
 We present first the following (abstract) Lemma:

\begin{lem} \label{LemODI2}
Let $\bar{T} >0$, $X, H$ be  $C^1$ functions from $[0, \bar{T}[$ to $\R_+$, $C_1, C_3, k_1>0$, $k_2> 7/2$, and $D := - H'$. We  suppose that $X$ is  solution to the following ordinary differential inequality for all $\eta>0$ small enough:
\ben\label{InqetaX} \f{d}{dt}X^2(t)+C_1(1+t)^{k_{1}}X(t)^{\f{14}5}\le  \eta\,C_3 \,D(t)\,X(t)^{\f{14}5}+\mathcal{B}(\eta)(1+t)^{-k_{2}}, \een
and that $ \lim_{t \to  \bar{T} } X(t) = +\infty$.  In the estimate above, $\mathcal{B}$ is a continuous decreasing nonnegative function.
%and $C_2$ is its factor.
% the notation $\mathcal{B}(x)=C_2\,x^{-13} \exp\{7x^{-\f{450}{14}}\}$.
% Then $T^*<\bigg(\f{1+k}{C_2}\mathcal{M}(0)+1\bigg)^{\f1{k+1}}-1$.
\medskip

 Then the following quantitative estimates hold for some $c, C>0$ (depending on $C_1, C_3, k_1, k_2$ and $\mathcal{B}$) and $k:= \min( \frac{2k_2 -7}5, k_1)$, when 
$\bar{T} -t>0$  is small enough:
%and $c\ll 1$,
\ben\label{blowest1} &&X(t)\ge C\, (H(t)- \bar{H})^{-\f54} \quad\mbox{while}\quad H(t)-\bar{H}\ge C\,(\bar{T}-t)\, (1+\bar{T})^{k}, \\
&& \inf_{s\in[t,\bar{T}]} X(s)\le \bigg(\mathcal{B}(c\,[\bar{T}-t])\, \f{2}{C_1} \,(1+\bar{T})^{-(k_{1}+k_{2})}\bigg)^{\f5{14}}.\label{blowest2} \een 
In the estimates above, we used the notation $\bar{H} := \lim_{t \to \bar{T}} H(t)$.
\end{lem}

\begin{proof}
%Thanks to Lemma \ref{LemODI1}, the blow up time $T^*$ should verify $T^*<\bigg(\f{1+k}{C_2}\mathcal{M}(0)+1\bigg)^{\f1{k+1}}-1$. Otherwise the solution will remain bounded for all time.
We can first use Lemma \ref{edo1} with $\eta = C_3^{-1}$ (up to choosing $C_3>0$ large enough), $B^* := \mathcal{B}(\eta)$. Estimate  \eqref{InqM1} implies that  
% $\mathcal{M}(T^*)=H(T^*)$, which implies
for $\delta>0$ small enough, and some $C_6>0$ given by Lemma \ref{edo1},
\ben\label{Xtblow1}  H(\bar{T} - \delta)+C_6\int_{t}^{\bar{T} - \delta} (1+s)^{k}ds\le H(t)-\f52[X(t)^2 + B^*\, (1+t)^{1-k_2}]^{-\f25}, \een 
which is enough to get the first part of estimate \eqref{blowest1}, by letting $\delta \to 0$.
\medskip

Using again estimate (\ref{Xtblow1}) and letting $\delta \to 0$, we see that 
$$ X(t)^2  +  B^*\, (1+t)^{1-k_2} \ge \bigg[ \frac25\, (H(t) - \bar{H}) \bigg]^{-5/2}. $$
By definition, $\lim_{t \to \bar{T}} H(t) =  \bar{H}$ so that $ (1+t)^{1-k_2}  = o_{t \to \bar{T}}   (H(t) - \bar{H})^{-5/2}$, 
and we get therefore the second part of estimate (\ref{blowest1}).
\medskip

In order to prove estimate \eqref{blowest2}, we go back to assumption \eqref{InqetaX}. Dividing it  by $X^{-\f{14}5}$, we get
\beno  -\f52\f{d}{dt}X(t)^{-\f45}+C_1(1+t)^{k_{1}} \le \eta \, C_3\,D(t)+\mathcal{B}(\eta)(1+t)^{-k_{2}}X(t)^{-\f{14}5}, \eeno 
which gives
% (remember that $X(T^*) = +\infty$)
\ben\label{Xtblow2} &&\f52X(t)^{-\f45}  -  \f52X(\bar{T} - \delta)^{-\f45} +C_1\int_t^{\bar{T}-\delta}(1+s)^{k_{1}}ds\notag\\&&\le \eta \, C_3\, H(0)+\mathcal{B}(\eta)\int_t^{\bar{T} - \delta}(1+s)^{-k_{2}}X(s)^{-\f{14}5}\, ds.\een
From this, we deduce  (remember that $ \lim_{t \to \bar{T}} X(t) = +\infty$) that
\beno  \sup_{s\in[t, \bar{T}]} \big(X^{-\f{14}5}(s)\big)(1+ t)^{-k_{2}}(\bar{T}-t)\,
\mathcal{B}(\eta)\ge C_1 (1+t)^{k_{1}}(\bar{T}-t)-\eta\, C_3\, H(0).\eeno
Let $\eta:=c\,(\bar{T}-t)$, for $c>0$ chosen small enough. Then for $\bar{T}-t$ small enough, we get
\beno  \sup_{s\in[t,\bar{T}]} X^{-1}(s) \ge \bigg(\big(\mathcal{B}(c\,[\bar{T}-t])\big)^{-1} \f{C_1}{2}\,(1+ \bar{T})^{k_{1}+k_{2}}\bigg)^{\f5{14}},\eeno 
which yields estimate \eqref{blowest2}.
\end{proof}
\medskip

\begin{rmk} 
The entropy $H$ plays an important role in the estimates giving hints about the way that a possible blowup could occur for
eq. \eqref{landau}.  We
observe that this quantity is continuous (with respect to time, on $[0, \bar{T})$) under our assumptions (namely when 
 $f:= f(t,v)$ is a nonnegative solution to eq. \eqref{landau} lying in $C([0, \bar{T});\dot{H}^1)\cap L^\infty_{loc}([0, \bar{T});L^1_5)$).
% Then $H(t)\in C[0,T)$.
% Using the definition \eqref{DefHt}, it is easy to see that
%\beno H(t_1)-H(t_2)=\int_{\R^3} [f(t_1)\log f(t_1)-f(t_2)\log f(t_2)] dv. \eeno
From the inequality
\beno |a\log a-b\log b|\le C_p|a-b|^{1/p}+|a-b|\log^{+}(|a-b|)+2\sqrt{a\wedge b}\sqrt{|a-b|},
\eeno
which is proved in Proposition \ref{logineqprop} (for $p>1$, $C_p=p/(e(p-1))$ and $a\wedge b=\min\{a,b\}$), used
%taking
when  $p:= 4/3$, we see that (for $0 \le t_1, t_2 <\bar{T}$)
$$ |H(t_1)-H(t_2)|\le C\,\bigg(\|f(t_1)-f(t_2)\|_{L^{\f32}_{3}}^{\f34}  +\, \|f(t_1)-f(t_2)\|^2_{L^2}+\|f(t_1)-f(t_2)\|_{L^{\f32}}^{\f12}\|f(t_1)+f(t_2)\|_{L^{\f32}_{3}}^{\f12} \bigg). $$
%where we take $p=\f43$.
 Thanks to the interpolation inequalities (based on H\"older's inequality and Sobolev embeddings), 
\beno \|f\|_{L^{\f32}_3}\le \|f\|_{L^1_5}^{\f35} \, \|f\|_{L^6}^{\f25}\le C\, \|f\|_{L^1_5}^{\f35} \, \|\na f\|_{L^2}^{\f25},
\qquad 
 \|f\|_{L^{2}}\le \|f\|_{L^1}^{\f25}\, \|\na f\|_{L^2}^{\f35},
\eeno
 we finally get the estimate (for some $C$ depending on $\|f\|_{L^\infty_t (L^1_5)}$ and $ \|f\|_{L^\infty_t (\dot{H}^1)}$, those norms being taken on $[0, \sup(t_1,t_2)]$)
\ben\label{compactnessH1}\qquad  |H(t_1)-H(t_2)| \le C\, \bigg( \|f(t_1)-f(t_2)\|_{\dot{H}^1}^{\f3{10}}+ \|f(t_1)-f(t_2)\|_{\dot{H}^1}^{\f65} +\,\|f(t_1)-f(t_2)\|_{\dot{H}^1}^{\f15}\bigg), \een
which is sufficient  to conclude.
% the proof of Proposition \ref{contiH}.
 \end{rmk}

We are in a position to prove the result.

\begin{proof}[Proof of Proposition \ref{des}] 
We begin with this proof in the case when $f$ is a smooth and quickly decaying (when $|v| \to \infty$) solution to eq. (\ref{landau}) -- (\ref{13d}) on a time interval $[0, \bar{T}[$. Thanks to estimate (\ref{EE1}) in Proposition~\ref{esqml}, we see that assumption 
(\ref{InqetaX}) holds with  $\mathcal{B}(x) := C_2\,x^{-13} \exp\{7x^{-\f{450}{14}}\}$. We can then apply Lemma \ref{LemODI2} 
to $X(t) := \|\nabla f(t)\|_{L^2}$. 

\medskip

We now briefly explain how to prove Proposition \ref{des} without assuming that $f$ is smooth and quickly decaying (when $|v| \to \infty$).
%Finally we want to prove \eqref{smoothingH2}. 
We consider an interval of time on which $f\in L^{\infty}_t(H^1 \cap L^1_{55})$. We  first observe that thanks to Proposition \ref{localwell}, we have $f \in L^2_t(H^2_{-3/2})$. 
% can find $t_0>0$ as small as we wish such that Let us fix $t_0\in]0,\mathcal{T}-\eta[$. Since $\na^2h\in L^2([0,\mathcal{T}-\eta];L^2_{-\f32})$, by the uniqueness result, we may choose $0<t_1\in[0,t_0/2]$ such that $\na^2 h(t_1)\in L^2_{-\f32}$ verifying
%\beno \|\na^2 h(t_1)\|_{L^2_{-\f32}}^2\le Ct_0^{-1}\|\na^2h\|_{L^2([0,\mathcal{T}-\eta];L^2_{-\f32})}^2\le C(t_0,\|f_0\|_{L\log L\cap %L^1_{55}\cap H^1}). \eeno
 Since $f\in L^\infty_t(L^1_{55})$,  we see that thanks to Proposition \ref{interpH1a}, $f\in  L^2_t(H^1_{12})$.   Using  now estimate \eqref{weightedH1} and the uniqueness result,
 we see then that 
%for $t\ge t_0>0$, 
$f \in L^{\infty}_t(H^1_{19/2}) \cap L^2_t(H^2_8)$ on  all compact intervals of $[t_0,\bar{T}[$ where $t_0>0$.
% and \eqref{H2te1}, one may have 
%\ben\label{H2te2} \|h\|_{L^\infty([t_1,\mathcal{T}-\eta], H^1_{19/2})}^2+\|h\|_{L^2([t_1,\mathcal{T}-\eta], H^2_{8})}^2\le C(t_0,\|f_0\|_{L\log L\cap L^1_{55}\cap H^1}).\een In particular,  we may choose $t_1<t_2\in[\f12t_0,t_0]$ such that \ben \label{H2te3}\|\na^2 h(t_2)\|_{L^2}\le C(t_0,\|f_0\|_{L\log L\cap L^1_{55}\cap H^1}).\een  Thus the desired result is reduced to prove the propagation of the regularity in $\dot{H}^2$ space  from this moment. We shall do the energy estimates for the original equation.

Using the equation satisfied by second order derivatives of $f$ and computing the time derivative of the square of the $H^2$ norm of $f$,
we can use Corollary \ref{coercivity2} and estimates like in Propositions \ref{E22} to \ref{E234}, 
and end up with the bound
 \beno  &&\f12\f{d}{dt} \|f\|_{\dot{H}^2}^2+ C(K)\|\na f\|_{\dot{H}^2_{-\f32}}^2\le C (\|f\|_{H^2}^2+\|f\|_{L^2}^4)\|f\|_{H^2_{\f92}}^2.\eeno
%Thanks to \eqref{H2te1}, \eqref{H2te2} and \eqref{H2te3} , by Gronwall's inequality, we deduce the desired result.
%This completes the proof of the theorem.
Thanks to the fact $f \in  L^2_t(H^2_8)$, we see that 
%(for $t \ge t_0>0$)
 $f \in L^{\infty}_t(H^2) \cap L^2_t(H^3_{-\frac32})$   on  all compact intervals of $[t_0,\bar{T}[$ where ${t}_0>0$.
\medskip

Using the estimates above for solutions $f^\epsilon$ of the approximated problem (\ref{approeq}) -- (\ref{ae}),
 we get that $f^\epsilon$ is bounded  in $L^{\infty}(H^2)$ on any interval $[t_1, t_2] \subset ]0, \bar{T}[$. Using also
%From this together with
 \eqref{compactnessH1},
this is sufficient to pass to the limit in the inequality 
\ben\label{IntegratedMFe} \mathcal{M}^\epsilon(t_2)+C_6\,\int_{t_1}^{t_2}  (1+t)^{k}dt\le  \mathcal{M}^\epsilon(t_1),
 \een
where $\mathcal{M}(t)=H^\epsilon(t)-\f52\bigg(\|h^\epsilon(t)\|_{\dot{H}^1}^2+ B^*\, (1+t)^{- k_2 +1}\bigg)^{-\f25}$.
%We now prove statement $(iii)$ of the Theorem.  Thanks to Theorem \ref{localwell}, we may assume that  $T^*<\infty$ is the first blow-up time in $\dot{H}^1$. We first show that on the interval $]0,T^*[$, the monotone formula \eqref{MonoFom} holds in the form of integration inequality. Precisely, for any $0<t_1<t_2<T^*$, we have 
We end up with the inequality (\ref{MonoFom}) in an "integrated in time" form:
\ben\label{IntegratedMF} \mathcal{M}(t_2)+C_6\,\int_{t_1}^{t_2}  (1+t)^{k}dt\le  \mathcal{M}(t_1).
 \een
The same construction can be used to obtain estimates (\ref{Xtblow1}) and (\ref{Xtblow2}) and conclude the proof of Proposition \ref{des}
when $f\in L^{\infty}_t(H^1 \cap L^1_{55})$ on all compact intervals of $[0, \bar{T}[$.
\end{proof}

\section{Appendix}

In this appendix, we present some results which are used in the paper. We start with interpolation results and properties of Lorentz spaces.

\subsection{Dyadic decompositions}

 We  start by recalling some aspects of the Littlewood-Paley decomposition.
Let $B_{{4}/{3}}:=  \{x \in \R^3 ~|~ |x| < {4}/{3} \} $
and $ R_{3/4, 8/3}:=     \{ x \in \R^3 ~|~ {3}/{4} < |x| < {8}/{3} \} $. Then one introduces two radially symmetric functions $ \psi \in
C_0^{\infty}(B_{{4}/{3}})$ and $ \varphi \in C_0^{\infty}(R_{3/4, 8/3})$  which satisfy
\begin{eqnarray}\label{defpsivarphi} \psi, \varphi \ge0,\quad \mbox{and} \quad \psi(x) + \sum_{j\geq 0} \varphi(2^{-j} \,x) =1,~\qquad x \in \R^3.
  \end{eqnarray}

%\subsubsection{Dyadic decomposition  in the phase space.}  
The dyadic operator $\mathcal{P}_j$ is defined  for $j\geq -1$ by
  \begin{eqnarray*} \mathcal{P}_{-1}f(x) :=  \psi(x)f(x),~~~~\qquad 
\mathcal{P}_{j}f(x) := \varphi(2^{-j}x)f(x)  ,~(j\geq 0). \end{eqnarray*}
We recall that $\mathcal{P}_{j}\mathcal{P}_{k}=0$ if $|j-k|>N_0$, for some $N_0 \in \N$.
% Let
%$\tilde{ \mathcal{P}}_{j}f(x) :=\sum_{|k-j|\le N_0, k \ge -1}\mathcal{P}_{k}f(x) $ and $  \mathcal{U}_{j}f(x) := \sum_{k\le j}\mathcal{P}_{k}f(x)$ where $N_0$ is a integer defined in  such a way 
%that $\mathcal{P}_{j}\mathcal{P}_{k}=0$ if $|j-k|>N_0$. For any function $f$, we use the decomposition
%$ f = \mathcal{P}_{-1} f +\sum_{j\geq 0} \mathcal{P}_j f.$
\medskip

%\subsubsection {Dyadic decomposition  in the frequency  space.} 
%\textcolor{red}{The decomposition in frequency below does not seem to be used. Laurent}
%\par
 %  We denote $ \tilde{m}:=\mathfrak{F} ^{-1} \psi $ and $\tilde{\phi} := \mathfrak{F}^{-1} \varphi $, where they are the inverse Fourier Transform of $\varphi$ and $\psi$.
%If we set $\tilde{\phi}_j(x):=2^{3j}\tilde{\phi}(2^{j}x)$, then the dyadic operator  in the frequency space $ \mathfrak{F}_j$ can be defined as
%follows \begin{eqnarray}\label{DefFj} \mathfrak{F}_{-1}f(x) = \int_{\mathrm{R}^3}\tilde{m}(x-y) f(y)dy,~~~~
%\mathfrak{F}_{j}f(x) =  \int_{\mathrm{R}^3}\tilde{\phi}_j(x-y) f(y)dy,~(j\geq 0).  \end{eqnarray}Let  $\tilde{\mathfrak{F}}_{j}f(x)=\sum_{|k-j|\le 3N_0}\mathfrak{F}_{k}f(x)
%$ and $\mathcal{S}_{j}f(x) = \sum_{-1 \leq k \le j}\mathfrak{F}_{k}f . $
%Then for any $f \in \mathcal{S}'(\mathbb{R}^3)$, it holds
%$ f = \mathfrak{F}_{-1} f +\sum_{j\geq 0} \mathfrak{F}_j f. $

% As a consequence,
We present a  norm based on the dyadic decomposition which is equivalent to the usual norm of the   weighted Sobolev spaces $H^s_l(\R^3)$:

\begin{prop}\label{baslem3}(\cite{HE16}) Let $s,l\in \R$. Then for   $f\in H^s_l$,
$ \sum_{k=-1}^\infty 2^{2kl}\|\mathcal{P}_{k}f\|_{H^s}^2\sim \|f\|_{H^s_l}^2.$ 
\end{prop} 
 
\medskip

% \textcolor{red}{I propose to put here a subsection (longer than the definition given in the introduction) presenting the different definitions of Lorentz spaces, and  their link with usual $L^p$ spaces. Laurent}
%\textcolor{blue}{Please check below to see if this presentation is proper.}

\subsection{Definition, norms and quasi-norms of Lorentz spaces}\label{S:Lorentz} 

For the convenience of the readers and the sake of self content, we collect some facts about Lorentz spaces  
from~\cite{AdFu, SW} which are useful for us. 
Considering $\R^{n}$ with Lebesgue measure $|\cdot|$.
In Section 1, we define the norm in Lorentz space $L^{p,q},\;p \in[1,\infty[, q\in[1,\infty]$ (or $p=q=\infty$, using the convention $t^{1/\infty}=1,\;t\geq 0$)
\begin{equation}
 \|f\|_{L^{p,q}}:= 
 \left\{
 \begin{array}{l}
 \bigg(\int_0^{\infty} \big(t^{1/p}f^{**}(t) \big)^{q}\frac{dt}{t}\bigg)^{1/q},\;1\leq q<\infty\\
 \sup\limits_{t>0}\; t^{1/p}f^{**}(t), \;q=\infty,
 \end{array}
 \right. \tag{\ref{Lorentz}}
 \end{equation}
which is different from the following (commonly used) definition
% (appears in many literatures) , 
 \begin{equation}\label{Lorentz2}
 \|f\|^{*}_{L^{p,q}}:= 
 \left\{
 \begin{array}{l}
 \bigg(\int_0^{\infty} \big(t^{1/p}f^{*}(t) \big)^{q}\frac{dt}{t}\bigg)^{1/q},\;1\leq q<\infty\\
 \sup\limits_{t>0}\; t^{1/p}f^{*}(t), \;q=\infty.
 \end{array}
 \right.
 \end{equation}

Here 
\[
f^{**}(t)=\frac{1}{t}\int_{0}^{t} f^{*}(s)ds,\;f^{*}(s)=\inf\{t\geq 0:a_{f}(t)\leq s\},
\]
where $a_{f}$ is the distribution function of $f$ given by 
\[
a_{f}(t)=|\{x\in\R^{n}:|f(x)|>t\}|.
\]

For $p\in (1,\infty)$ and $q\in [1,\infty]$, we note that the functional $\|\cdot\|^{*}_{L^{p,q}}$ is a norm only 
when $q\leq p$ and a quasi-norm otherwise, 
on the other hand $\|\cdot\|_{L^{p,q}}$ is always a norm. For $p\in (1,\infty)$ and $q\in [1,\infty]$, the following 
comparison inequality holds:
\[
\|f\|^{*}_{L^{p,q}} \leq \|f\|_{L^{p,q}} \leq \frac{p}{p-1} \|f\|^{*}_{L^{p,q}}.
 \]
Clearly for $1<p<\infty$, we have $\|f\|^{*}_{L^{p,p}}= \|f\|_{L^{p}}$ and thus $L^{p,p}=L^{p}$. 
For $p=1$ the situation is different (See also~\cite{AdFu} p. 224), one can indeed check that
\[
\|f\|_{L^{1,\infty}}=\sup_{t>0} t f^{**}(t)=\sup_{t>0}\int_{0}^{t} f^{*}(s)ds=\int_{0}^{\infty}f^{*}(s)ds=\|f\|_{L^{1}}.
\]
Finally,  for $p=\infty$ (See also~\cite{AdFu} p. 224),  one can also check that 
\[
\|f\|_{L^{\infty,\infty}}=\sup_{t>0}  f^{**}(t)=
\sup_{t>0}\frac{1}{t}\int_{0}^{t} f^{*}(s)ds=f^*(0)=\|f\|_{L^{\infty}} .
\]

\subsection{Inequalities and Interpolation}

We  begin with Sobolev embedding theorem and O'Neil inequality in Lorentz spaces.

\begin{prop}[see \cite{AdFu} and \cite{ONeil}]\label{oneil} 
(i). If $f\in H^1(\R^3)$, then $f\in L^{6,2}(\R^3)$ and
\beno \|f\|_{L^{6,2}(\R^3)} \le C\|f\|_{H^1(\R^3)}. \eeno

(ii). For $p_1,p_2,q_1,q_2\in[1,\infty]$ with $1/p=1/p_1+1/p_2$ and $1/q\le1/q_1+1/q_2$, there exists a computable constant $C$ depending only on $p_1,q_1,p_2,q_2$ such that
 \beno
%\color{blue}
 \|fg\|_{L^{p,q}}&\le& C\, \|f\|_{L^{p_1,q_1}}\|g\|_{L^{p_2,q_2}}.
  \eeno

 (iii). If  $f\in L^{p,q_1}, g\in L^{p',q_2}$ where  $p,p',q_1,q_2\in[1,\infty]$ such that $1/p+1/p'=1$ and $1/q_1+1/q_2\ge1$, then
 $f*g\in L^\infty$ and
 \beno \|f*g\|_{L^\infty}\le \|f\|_{L^{p,q_1}}\|g\|_{L^{p',q_2}}. \eeno
\end{prop}

\medskip

Next  we will prove some useful interpolation inequalities which are widely used throughout  the paper. 

\begin{prop}\label{interpL31} For $m\in\R$, and some constant $C>0$ depending only on $m$,
 \beno
\|f\|_{L^{3,1}_m}\le C\, \|f\|_{L^1_{5m+1}}^{\f15}\|f\|_{H^1}^{\f45}.\eeno
\end{prop}

\begin{proof} We split the proof into two parts. The first step is devoted to showing that
\beno \|f\|_{L^{3,1}}\le \|f\|^{\f15}_{L^1}\|f\|_{L^{6,2}}^{\f45}\le  C\,   \|f\|^{\f15}_{L^1}\| f\|_{H^1}^{\f45}.\eeno  
By the definition of Lorentz spaces, one gets
\beno \|f\|_{L^{3,1}}&=&\int_0^\infty t^{\f13}f^{**}(t) \f{dt}{t}\\
&\le&(\int_0^R (t^{\f16}f^{**}(t))^2 \f{dt}{t}\big)^{\f12}\big(\int_0^R t^{\f13} \f{dt}{t}\big)^{\f12}+ (\sup_{t>0}tf^{**}(t))\int_R^\infty t^{-\f23}\f{dt}{t} \\
&\le&  \|f\|_{L^{6,2}}R^{\f16}+\|f\|_{L^{1,\infty}}R^{-\f23}. \eeno
We conclude  by optimizing $R$ and by using the identity $\|f\|_{L^{1,\infty}}=\|f\|_{L^{1}}$(see 
subsection~\ref{S:Lorentz} and~\cite{AdFu} p. 224).
% \textcolor{red}{Point above to explain, as usually it is assumed that $L^{p,p} = L^p$. Laurent} \textcolor{blue}{It is only for $1<p\leq \infty$, not $p=1$}\textcolor{red}{I propose to explain that in detail in the subsection on Lorentz spaces. Laurent} 
% \textcolor{blue}{Agree. Done.}

In next step, we extend the above result to the general case (the one with weights appearing in the norms) using a dyadic decomposition. We observe that 
\beno \|f\|_{L^{3,1}_m}&=&\|f\lr{\cdot}^m\|_{L^{3,1}}\le\sum_{k=-1}^\infty \|\mathcal{P}_kf\lr{\cdot}^m\|_{L^{3,1}}\le C\sum_{k=-1}^\infty \|\mathcal{P}_kf\|_{L^{3,1}}\|\mathcal{P}_k\lr{\cdot}^m\|_{L^{\infty,\infty}}\\
&\le &C\sum_{k=-1}^\infty \|\mathcal{P}_kf\|_{L^{3,1}}2^{km}\le C \sum_{k=-1}^\infty \big(2^{5km}\|\mathcal{P}_kf\|_{L^1}\big)^{\f15}\|\mathcal{P}_kf\|_{H^1}^{\f45}\\&\le &
C \bigg(\sum_{k=-1}^\infty 2^{\f53km}\|\mathcal{P}_{k}f\|_{L^1}^{\f13}\bigg)^{\f35}\bigg(\sum_{k=-1}^\infty \|\mathcal{P}_{k}f\|_{H^1}^2\bigg)^{\f25},
\eeno where we use O'Neil inequality \eqref{oneil}  and the fact $ \|f\|_{L^{\infty,\infty}}=\|f\|_{L^\infty}$(see 
subsection~\ref{S:Lorentz} and~\cite{AdFu} p. 224).
From this together with the computation
\beno \sum_{k=-1}^\infty 2^{\f53km}\|\mathcal{P}_{k}f\|_{L^1}^{\f13}\le C\sum_{k=-1}^\infty 2^{\f53km}2^{-\f13(5m+1)k}\|f\|_{L^1_{5m+1}}^{\f13}\le C\|f\|_{L^1_{5m+1}}^{\f13},\eeno   we finally get the inequality  
\beno \|f\|_{L^{3,1}_m}\le C\|f\|_{L^1_{5m+1}}^{\f15}\|f\|_{H^1}^{\f45}.
\eeno

%We complete the proof for the desired result.
\end{proof}

\begin{prop}\label{interpH1a}
For $m\in\R$, 
\beno
\|f\|_{H^1_m}\le C\, \|f\|_{L^1_{15/4+7m/2}}^{\f27}(\|f\|_{L^1_{-3/2}}+\|\na^2 f\|_{L^2_{-3/2}})^{\f{5}7},
\eeno
\beno
%\color{blue}
\|f\|_{H^1_m}\le C\, \|f\|_{L^1_{5/4+7m/2}}^{\f27}(\|f\|_{L^1_{-1/2}}+\|\na^2 f\|_{L^2_{-1/2}})^{\f{5}7},
\eeno
where $C>0$ is a constant depending only on $m$.
\end{prop}

\begin{proof}  We first claim that
\beno \|f\|_{H^1}\le C\, \|f\|_{L^1}^{\f27}(\|f\|_{L^1}+\|\na^2f\|_{L^2})^{\f57}.\eeno
Indeed, since $\|f\|_{H^1}^2\sim \int_{\R^3} (1+|\xi|)^2\hat{f}(\xi)^2d\xi$, then (for $R\ge 1$)
\beno  \|f\|_{H^1}^2\le C\, ( R^5\|f\|_{L^1}^2+R^{-2}\|f\|_{H^2}^2).\eeno 
We conclude by taking $R^7=\|f\|_{H^2}^2/\|f\|_{L^1}^2+1$, recalling that  $\|f\|_{H^2}\sim \|f\|_{L^1}+\|\na^2f\|_{L^2}$. Thanks to Proposition
 \ref{baslem3}, we see that
\beno \|f\|_{H^1_m}^2&\sim& \sum_{k=-1}^\infty 2^{2km}\|\mathcal{P}_{k}f\|_{H^{1}}^2 \le C \sum_{k=-1}^\infty 2^{2km}\|\mathcal{P}_{k}f\|_{L^1}^{\f47}(\|\mathcal{P}_{k}f\|_{ \dot{H}^2}+\|\mathcal{P}_{k}f\|_{L^1})^{\f{10}7}\\
&\le&C \bigg(\sum_{k=-1}^\infty 2^{(7m+15/2)k}\|\mathcal{P}_{k}f\|_{L^1}^2\bigg)^{\f27}\bigg(\sum_{k=-1}^\infty 2^{-3k}(\|\mathcal{P}_{k}f\|^2_{\dot{H}^2}+\|\mathcal{P}_{k}f\|_{L^1}^2)\bigg)^{\f57}\\
&\le& C\|f\|_{L^1_{15/4+7m/2}}^{\f47}(\|f\|_{L^1_{-3/2}}+\|\na^2 f\|_{L^2_{-3/2}})^{\f{10}7}.\eeno 
The proof of the second inequality is similar.
\end{proof}

  \begin{prop}[\cite{HLP}]\label{logineqprop} For $a,b\ge 0$ and $1<p<\infty$, the following inequality holds:
%there exist a constant $C_p=\f{p}{e(p-1)}$ such that 
\ben\label{logineq} |a\log a-b\log b|\le C_p\, |a-b|^{1/p}+|a-b|\log^{+}(|a-b|)+2\sqrt{a\wedge b}\sqrt{|a-b|},
\een where
$a\wedge b=\min\{a,b\},\; C_p := \f{p}{e(p-1)}$ and 
\[
\log^+ |x|=\left\{
\begin{array}{ll}
\log x & {\rm if}\;\; x\geq 1\\
0 & {\rm if}\;\; x<1.
\end{array}
\right.
\]
\end{prop}

\begin{proof} We first observe that the following inequalities hold:
\begin{equation}\label{ll}
 \log (1+x)\le \sqrt{x},\quad x\ge 0, \qquad
 |\log x|\le \f{1}{e\alpha}\, x^{-\alpha},\quad 0< x\le 1, \alpha>0 .
\end{equation} 

%Inequality \eqref{*1} is indeed equivalent to
%$$ \log (1+x^2)\le x,\quad x\ge 0, $$
%which can be proved thanks to the fact that the function $g(x):=\log (1+x^2)-x$ is decreasing on $[0,\infty).$
%In order to prove \eqref{*2},  we define $$\varphi(x) :=x^{\alpha}|\log x|=- x^{\alpha}\log x,\quad x\in ]0,1], $$
%together with $\varphi(0)=0$. Then $\varphi\in C([0,1])\cap C^{\infty}((0,1))$, and
%\beno \varphi'(x)=\alpha
%x^{\alpha-1}\log\Big(\f{e^{-1/\alpha}}{x}\Big),\quad 0<x<1.\eeno
%We observe that  $\varphi'(x)>0 $ when $ 0<x<e^{-1/\alpha}$ and that $\varphi'(x)<0$ when $e^{-1/\alpha}<x<1$. 
%Therefore $$\max_{0\le x\le 1}\varphi(x)=\varphi(e^{-1/\alpha})=  \f{1}{e\alpha}.$$
%As a consequence, we see that  \eqref{*2} holds.
%\medskip

Then, let $q>1$ satisfy $1/p+1/q=1$.
%,  i.e.
%$q=\f{p}{p-1}.$
%Up to exchanging the role of $a$ and $b$ in \eqref{logineq}, we may assume that $a>b$.  
%The case $b=0$ can be treated easily
% Using inequality \eqref{ll} with $\alpha=1/q$). 
In what follows, we  assume that $a>b>0$. 

We first observe that
$$ |a\log a-b\log b|
%=|(a-b)\log a+b(\log a-\log b)|= |(a-b)\log a+b\log\Big(\f{a}{b}\Big)|\\
\le (a-b)|\log a|+b\log\Big(\f{a}{b}\Big) . $$
Using estimate \eqref{ll}, we see that
\beno&&  b\log\Big(\f{a}{b}\Big)=b\log\Big(1+\f{a-b}{b}\Big)
\le b\sqrt{\f{a-b}{b}}=\sqrt{b}\sqrt{a-b}.\eeno
Next we compute
$$
|\log a| =\Big|\log\Big((a-b)\Big(1+\f{b}{a-b}\Big)\Big)\Big|
%&&=\Big|\log(a-b)+\log\Big(1+\f{b}{a-b}\Big)\Big|
%\le |\log(a-b)|+\log\Big(1+\f{b}{a-b}\Big)
%\\
\le\f{q}{e}(a-b)^{-1/q}+\log^{+}(a-b)+\sqrt{\f{b}{a-b}}, $$
where in the case when $a-b\le 1$, we use estimate
\eqref{ll} with $\alpha=1/q$.
This gives
$$(a-b)|\log a|\le \f{q}{e}(a-b)^{1/p}+(a-b)\log^{+}(a-b)+\sqrt{b}\sqrt{a-b},
$$
which enables to conclude.
% that $$ |a\log a-b\log b|\le \f{q}{e} (a-b)^{1/p}+(a-b)\log^{+}(a-b)+2\sqrt{b}\sqrt{a-b},\quad a>b>0.$$
%This ends the proof of the Proposition.
\end{proof}

\subsection{A remark on initial data}

Finally we show that there exist initial data for Theorem~\ref{maintheorem1} whose initial relative entropy $H(0)$ is not big, while 
their $\dot{H}^{1}$ norm is large. 
%Thus their $\dot{H}^{\f12}$ norm, the so call critical space for incompressible Naviver-Stokes equations, could be very large.
 See also the last comment of Theorem~\ref{maintheorem1} in the introduction. 
\begin{prop}\label{example} Let $\epsilon,\eta\ll1$ and $\eta:=\epsilon^{11/9}$. 
Assume the Maxwellian $\mu$ and a smooth $\phi_0\geq 0$ both 
satisfy the normalization \eqref{f0}.  Then
\ben\label{Speexample2} \qquad f_0(v) := (1-\eta+\eta\epsilon^{2})^{\f32}\bigg[(1-\eta)\mu\big((1-\eta+\eta\epsilon^{2})^{\f12}v\big)+\eta\epsilon^{-3}\phi_0\big(\epsilon^{-1}(1-\eta+\eta\epsilon^{2})^{\f12}v\big)\bigg]\een
 also satisfies the normalization \eqref{f0},  and
\[
 \mathcal{M}(0) := H(0)- \f52\big(\|h(0)\|_{\dot{H}^1}^2+ B\, \big)^{-\f25}\leq 0,
\]
while $\|h_0\|_{\dot{H}^{\f12}}\sim \epsilon^{-\f{7}{9}}$ (where $h_0 = f_0 - \mu$, and $H(0)$ is the relative entropy of $f_0$).
%\textcolor{red}{Is this true for all $\epsilon,\eta\ll1$ or does one need  $\eta:=\epsilon^{11/9}$? Laurent} \textcolor{blue}{Yes}
\end{prop}

\begin{proof}
%\textcolor{red}{Please add here details about why $f_0$ satisfies the normalization \eqref{f0}. Laurent}
 We check that $f_0$ satisfies the third condition of normalization~\eqref{f0} 
since the other two are easier to check. Thanks to a change of variables,
\[
\begin{split}
&(1-\eta+\eta\epsilon^{2})^{\f32}\bigg[(1-\eta)\int \mu\big((1-\eta+\eta\epsilon^{2})^{\f12}v\big)|v|^2dv+\eta\epsilon^{-3}\int \phi_0\big(\epsilon^{-1}(1-\eta+\eta\epsilon^{2})^{\f12}v\big)|v|^2dv
\bigg]\\
&=\frac{1-\eta}{1-\eta+\eta\epsilon^{2}}\int\mu(v)|v|^2dv + 
\frac{\eta\epsilon^2}{1-\eta+\eta\epsilon^{2}} \int \phi_0(v)|v|^2dv=3.
\end{split}
\] 

Next let us estimate $\mathcal{M}(0)$. We first observe that $\eta\epsilon^{-\f32}\ge1$. Then for   any $s>0$, $ \|h_0\|_{\dot{H}^{s}}\sim \eta\epsilon^{-s-\f32}$.
% By fundamental theorem of calculus, 
The relative entropy $H(0)$ is bounded from above by 
\beno  H(0)&=&\int_{\R^3} \bigg(\f{f_0}{\mu}\log(\f{f_0}{\mu})-\f{f_0}{\mu}+1\bigg)\,\mu \,dv\\
&\le& \int_{\R^3}\int_0^1 \bigg|\log(\f{f_\theta}{\mu}) \bigg| \, \bigg|\f{f_0}{\mu}-1 \bigg|\,\mu\, d\theta dv, \eeno
 for some $\theta \in [0,1]$, with the notation $f_\theta=(1-\theta) f_0+\theta \mu$. We from now on denote by $C$ any strictly positive constant.
\par
 At points where $f_0\ge \mu$, we see that
 \beno |\log(\f{f_\theta}{\mu})|= \log(\f{f_\theta}{\mu})\le C(|v|^2+\log(\eta\epsilon^{-3})),\eeno 
while at points where $f_0\le \mu$, 
 \beno |\log(\f{f_\theta}{\mu})|=  \log(\f{\mu}{f_\theta})\le  \log(\f{\mu}{f_0}) \le C(|v|^2+1). \eeno 
From these estimates, we deduce that
\beno H(0)\le C(1+\log(\eta\epsilon^{-3}))\|f_0-\mu\|_{L^1_2}\le C\, \eta\, (1+\log(\eta\epsilon^{-3})). \eeno
Thus,
\beno \mathcal{M}(0)\le C\eta (1+\log(\eta\epsilon^{-3}))-C\eta^{-\f45}\epsilon^{2}.\eeno
Remembering that $\eta\sim \epsilon^{11/9}$, and $\epsilon\ll1$, we see that
\beno \mathcal{M}(0)\le C\epsilon^{11/9}\log(\epsilon^{-1}) -C \epsilon^{ 46/45}\le 0. \eeno
Finally,  $\|h_0\|_{\dot{H}^{\f12}}\sim \epsilon^{-\f{7}{9}}$,
so that $h_0$ is a large initial  datum for Landau equation in $\dot{H}^{\f12}$ (the critical space for incompressible Navier-Stokes equations).
\end{proof}

\section{Acknowledgement}
The research of L.-B. He was supported by NSF of China under Grant No.11771236.
The research of J.-C. Jiang was supported in part by National Sci-Tech Grant MOST 107-2115-M-007-002-MY2
and MOST 109-2115-M-007-002-MY3.

\end{document}